\documentclass[10pt, reqno]{amsart}

\usepackage{amsmath,amssymb,stmaryrd}
\usepackage[mathscr]{euscript}
\usepackage[usenames,dvipsnames]{color}
\usepackage{pstricks}
\usepackage{graphicx, fancyhdr}
\usepackage[all]{xy}
\usepackage{tabularx}

\usepackage{hyperref} 
\usepackage{amssymb, mathrsfs,bbold}
\usepackage{amsmath, amsthm}
\usepackage{enumerate}
\usepackage{dsfont}
\usepackage{color}
\usepackage[a4paper,margin=2.8cm,bottom=3cm]{geometry}
\usepackage[utf8]{inputenc}   
\usepackage{titlesec, wasysym}
\usepackage{appendix}
\usepackage{todonotes, fancyhdr}
\usepackage{chngcntr}
\usepackage{cancel}

\newtheoremstyle{mystyle}
{3pt}               
{3pt}               
{\it }                      
{}                      
{\sffamily\bfseries}             
{}                      
{0.5em}                 
{#1 #2{\hspace{0.2cm}--\hspace{-0.2cm}}  }   

\theoremstyle{mystyle}

\newtheorem{thm}{Theorem}[section]
\newtheorem{prop}[thm]{Proposition}
\newtheorem{defi}[thm]{Definition}
\newtheorem{lemm}[thm]{Lemma}
\newtheorem{coro}[thm]{Corollary}
\newtheorem{ex}[thm]{Example}

\newtheorem*{defi*}{Definition}
\newtheorem*{lemm*} {Lemma}
\newtheorem*{thm*}{Theorem}
\newtheorem*{prop*}{Proposition}

\renewenvironment{proof}[1][\unskip]{%
    \begin{list}{\hspace{1.15cm}\textbf{\textsf{Proof #1 --}}}{
        \setlength{\topsep}{0pt}%
        \setlength{\leftmargin}{0pt}%
        \setlength{\rightmargin}{0pt}%
        \setlength{\listparindent}{0pt}%
        \setlength{\itemindent}{0pt}%
        \setlength{\parsep}{0pt}%
        \addtolength{\leftmargin}{0pt} 
        \addtolength{\rightmargin}{0pt}%
    } \item }{\hfill $\rhd$ \end{list}\smallskip}

\colorlet{darkgreen}{green!80!black}

\newcommand{\mcC}{\mathcal{C}}
\newcommand{\bbR}{\mathbb{R}}
\newcommand{\mcB}{\mathcal{B}}

\newcommand{\rmd}{\mathrm{d}}

\renewcommand\thesection       {\arabic{section}}
\renewcommand\thesubsection    {\thesection{\boldmath $.$}\arabic{subsection}}
\renewcommand\thesubsubsection    {\thesection{\boldmath $.$}\arabic{subsection}{\boldmath $.$}\arabic{subsubsection}} 

\titleformat{\section}[block] 
{\filcenter\normalfont\sffamily\bfseries\large}  
{{\hspace{-0.87cm}}\thesection \hspace{0.2em} --\vspace{0cm}}{0.5em}{} 

\titleformat{\subsection}[runin]
{\filcenter\normalfont\sffamily\bfseries}  
{{\hspace{0cm}}\thesubsection \hspace{0.15em} -- \vspace{0.1cm}}{.2em}{}   
\titlespacing{\subsection}{-0pc}{1.5ex plus .1ex minus .2ex}{0pc}   

\titleformat{\subsubsection}[runin]
{\filcenter\normalfont\sffamily\bfseries}   
{\filright\sffamily{\hspace{0cm}}\thesubsubsection\hspace{0.2em} --}{.5em}{}\titlespacing{\subsection}{-0pc}{1.5ex plus .1ex minus .2ex}{0pc}

\numberwithin{subsection}{section}
\numberwithin{subsubsection}{subsection}
\numberwithin{equation}{section} 

\usepackage{tikz-cd}
\usetikzlibrary{calc}
\usetikzlibrary{shapes.misc}
\usetikzlibrary{shapes.symbols}
\usetikzlibrary{snakes}
\usetikzlibrary{decorations}
\usetikzlibrary{decorations.markings}

\makeatletter
\pgfdeclareshape{crosscircle}
{
  \inheritsavedanchors[from=circle] 
  \inheritanchorborder[from=circle]
  \inheritanchor[from=circle]{north}
  \inheritanchor[from=circle]{north west}
  \inheritanchor[from=circle]{north east}
  \inheritanchor[from=circle]{center}
  \inheritanchor[from=circle]{west}
  \inheritanchor[from=circle]{east}
  \inheritanchor[from=circle]{mid}
  \inheritanchor[from=circle]{mid west}
  \inheritanchor[from=circle]{mid east}
  \inheritanchor[from=circle]{base}
  \inheritanchor[from=circle]{base west}
  \inheritanchor[from=circle]{base east}
  \inheritanchor[from=circle]{south}
  \inheritanchor[from=circle]{south west}
  \inheritanchor[from=circle]{south east}
  \inheritbackgroundpath[from=circle]
  \foregroundpath{
    \centerpoint%
    \pgf@xc=\pgf@x%
    \pgf@yc=\pgf@y%
    \pgfutil@tempdima=\radius%
    \pgfmathsetlength{\pgf@xb}{\pgfkeysvalueof{/pgf/outer xsep}}%
    \pgfmathsetlength{\pgf@yb}{\pgfkeysvalueof{/pgf/outer ysep}}%
    \ifdim\pgf@xb<\pgf@yb%
      \advance\pgfutil@tempdima by-\pgf@yb%
    \else%
      \advance\pgfutil@tempdima by-\pgf@xb%
    \fi%
    \pgfpathmoveto{\pgfpointadd{\pgfqpoint{\pgf@xc}{\pgf@yc}}{\pgfqpoint{-0.707107\pgfutil@tempdima}{-0.707107\pgfutil@tempdima}}}
    \pgfpathlineto{\pgfpointadd{\pgfqpoint{\pgf@xc}{\pgf@yc}}{\pgfqpoint{0.707107\pgfutil@tempdima}{0.707107\pgfutil@tempdima}}}
    \pgfpathmoveto{\pgfpointadd{\pgfqpoint{\pgf@xc}{\pgf@yc}}{\pgfqpoint{-0.707107\pgfutil@tempdima}{0.707107\pgfutil@tempdima}}}
    \pgfpathlineto{\pgfpointadd{\pgfqpoint{\pgf@xc}{\pgf@yc}}{\pgfqpoint{0.707107\pgfutil@tempdima}{-0.707107\pgfutil@tempdima}}}
  }
}
\makeatother

\colorlet{symbols}{black}    
\colorlet{testcolor}{green!60!black}
\colorlet{supcolor}{red!60!black}

\tikzset{
	root/.style={circle, fill=testcolor!70, draw=testcolor, inner sep=1pt, minimum size=0.5mm},
	dot/.style={circle, draw=black, fill=black, inner sep=0pt, minimum size=0.2mm},
	noise/.style={circle, draw=black, fill=white, inner sep=0pt, minimum size=1mm},
	noiseblue/.style={circle, fill=blue!40, draw=blue, inner sep=0pt, minimum size=1mm},
	noisegray/.style={circle, fill=gray!40, draw=gray, inner sep=0pt, minimum size=1mm},
	blackdot/.style={circle, draw=black, fill=black, inner sep=0pt, minimum size=1.2mm},
	K/.style= {semithick, shorten >=0pt,shorten <=0pt,-},
	DK/.style={thick, densely dotted, shorten >=0pt,shorten <=0pt},   
	}

\newcommand{\X}{\raisebox{-.2ex}{\includegraphics[scale=1]{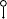}}}
\newcommand{\Xtwo}{\raisebox{0ex}{\includegraphics[scale=1]{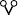}}}

\newcommand{\IXtwo}{\raisebox{-.2ex}{\includegraphics[scale=1]{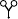}}}
\newcommand{\IXthree}{\raisebox{-.3ex}{\includegraphics[scale=1]{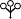}}}

\makeatletter
\newcommand*{\defeq}{\mathrel{\rlap{%
                     \raisebox{0.3ex}{$\m@th\cdot$}}%
                     \raisebox{-0.3ex}{$\m@th\cdot$}}%
                     =}
\makeatother

\makeatletter
\newcommand*{\eqdef}{=\mathrel{\rlap{\raisebox{0.3ex}{$\m@th\cdot$}}%
					                     \raisebox{-0.3ex}{$\m@th\cdot$}}%
                     }
\makeatother

\begin{document}
\thispagestyle{fancy}

\begin{center}
{\LARGE\sffamily{Global harmonic analysis for $\Phi^4_3$ on closed Riemannian manifolds   \vspace{0.5cm}}}
\end{center}

\begin{center}
{\sf I. BAILLEUL and  N.V. DANG and L. FERDINAND and T.D.  T\^O}
\end{center}

\vspace{1cm}

\begin{center}
\begin{minipage}{0.8\textwidth}
\renewcommand\baselinestretch{0.7} \scriptsize \textbf{\textsf{\noindent Abstract.}} 
Following Parisi \& Wu's paradigm of stochastic quantization, we constructed in \cite{BDFT} a $\Phi^4$ measure on an arbitrary closed, compact Riemannian manifold of dimension $3$ as an invariant measure of a singular stochastic partial differential equation. This is a progress on a longstanding open problem in quantum fields on curved backgrounds. In the present work, we build all the harmonic and microlocal analysis tools that are needed in \cite{BDFT}. 
\end{minipage}
\end{center}

\setcounter{tocdepth}{2}

\makeatletter
\def\@dotsep{1.5}
\renewcommand{\l@section}  {\@dottedtocline{1}{0em}{1em}}
\renewcommand{\l@subsection}  {\@dottedtocline{2}{2em}{2em}}
\makeatother

{\small 
\tableofcontents
}

\section{Introduction}
\label{SectionIntro}

Thanks to the recent breakthroughs of Hairer~\cite{Hairer} and Gubinelli, Imkeller \& Perkowski \cite{GIP}, a certain class of stochastic partial differential equations (SPDEs) with low regularity coefficients now have a robust solution theory. Examples of equations in this class include the KPZ equation
\begin{eqnarray*}
(\partial_t-\partial_x^2)u=(\partial_xu)^2 + \xi,
\end{eqnarray*}  
the parabolic $\Phi^4$ equation
\begin{eqnarray} \label{EqPhi43}
(\partial_t -\Delta+1)u+u^3 = \xi,
\end{eqnarray}
and the generalized parabolic Anderson model
\begin{eqnarray*}
\left(\partial_t-\Delta+1\right)u=F(u)\xi,
\end{eqnarray*}  
A common feature to the above equations is that they do not make sense in a classical sense due to the low regularity of the driving noises which are denoted $\xi$. One needs to renormalize the PDE, subtracting some infinite counterterm in the equation itself, to have a well-defined notion of solution. These recent seminal works, and the body of work that followed, allowed a number of authors  \cite{HM,MW17, MW20,GH, BG} to recover the existence of the celebrated $\Phi^4_3$ quantum field theory measure first constructed by Glimm \& Jaffe \cite{Glimm,GJ} in the $2$ and $3$ dimensional Euclidean space in the 70s. The extension of such results to a curved setting is a longstanding open problem that matters from the point of view of constructive quantum field theory. We gave in the work \cite{BDFT} the first construction of the dynamical $\Phi^4$ model on any compact, boundaryless, $3$-dimensional Riemannian manifold $M$ and deduced from some functional properties of the long time behaviour of the semigroup generated by the SPDE \eqref{EqPhi43} the existence and non-triviality of a $\Phi^4$ Gibbs measure
\begin{equation} \label{EqGibbsMeasure}
\nu_{\Phi^4}(\rmd u)\propto {\exp\Big({-\frac{1}{2}\int_M u(1-\Delta)u -\frac{1}{4}\int_M {u^4} \Big) }} \rmd u\,,
\end{equation}
on $M$. This measure is seen as an invariant measure for the semigroup. Our construction is naturally deeply rooted in the recent developments in the area of singular stochastic PDEs.

The construction in \cite{BDFT} relies on several results of harmonic and microlocal analysis in the manifold setting which are considered in the present paper. Our goal is to build in a relatively simple and self-contained way all the tools from paradifferential calculus and microlocal analysis on compact manifolds that we used in the existence proof of a $\phi^4_3$ measure on $M$ in \cite{BDFT}. Most of its content is independent of that precise problem, though, so we hope the reader will take profit from what follows to investigate a number of other problems.

\medskip

{\it {\Large \S}1 -- The $\Phi^4_3$ measure on a closed $3$-dimensional manifold.} We fix from now on a smooth closed (i.e. compact, boundaryless) $3$-dimensional Riemannian manifold $(M,g)$ and let $\Delta\defeq \Delta_g$ stand for the (negative) Laplace-Beltrami operator on $M$. 
We will denote by
$$
P\defeq  1-\Delta 
$$ 
the massive Laplacian and by $e^{-tP}$  the corresponding heat kernel. The eigenfunctions $f_\lambda$ of $P$ form an orthonormal basis of $L^2(M)$. Let $\xi^\lambda$ stand for a collection of independent real valued one dimensional white-noises indexed by the eigenvalues $\lambda$ of $P$, defined on some probability space $(\Omega, \mathcal{F}, \mathbb{P})$. Spacetime white noise on $\mathbb{R}\times M$ can be constructed as the random series
$$
\xi \defeq \sum_\lambda \xi^\lambda\otimes f_\lambda.
$$
It belongs almost surely to the parabolic H\"older space of regularity exponent $-5/2-\varepsilon$, for any $\varepsilon>0$. We let
$$
\xi_r \defeq e^{r(\Delta-1)}\xi
$$
stand for the space regularized spacetime white noise -- so $\xi_r$ is still white in time. The work \cite{BDFT} is dedicated to constructing a Gibbs measure that formally writes \eqref{EqGibbsMeasure} as an invariant probability measure of the parabolic $\Phi^4$ parabolic dynamics
\begin{equation} \label{EqSPDE}
\partial_tu = \sqrt{2}\xi + (\Delta-1) u - u^3.
\end{equation}
Set
\begin{equation} \label{EqDefnArBr}
a_r \defeq \frac{r^{-1/2}}{4\sqrt{2} \, \pi^{3/2}}\,, \qquad b_r \defeq \frac{\vert\hspace{-0.03cm}\log r\vert}{32\,\pi^2}\,.
\end{equation}
One of the main results from \cite{BDFT} reads as follows. The constant $0<\varepsilon$ is small enough and fixed throughout.

\smallskip

\begin{thm*}
Pick $\phi\in \mcC^{-1/2-\varepsilon}(M)$. The equation
\begin{equation} \label{EqRenormalizedSPDE}
(\partial_t - \Delta + 1)u_r = \sqrt{2}\xi_r - u_r^3 + 3(a_r - b_r)u_r
\end{equation}
with initial condition $\phi$, has a unique solution over $[0,\infty)\times M$ in some appropriate function space. For any $0<T<\infty$ this random variable converges in probability in 
$$
C\big([0,T], \mcC^{-1/2-\varepsilon}(M)\big)
$$ 
as $r>0$ goes to $0$ to a limit $u$. 
\end{thm*}

\smallskip

We note that obtaining a local in time well-posedness result for \eqref{EqRenormalizedSPDE} is relatively elementary. The non-trivial points in the preceding statement are the long time existence of that local in time solution and its convergence as $r>0$ goes to $0$. The function $u$ is what we define as the solution to equation \eqref{EqSPDE}; it turns out to be a Markov process. 

\smallskip

\begin{thm*}
The dynamics of $u$ is Markovian and its associated semigroup on $\mcC^{-1/2-\varepsilon}(M)$ has an invariant non-Gaussian probability measure.
\end{thm*}

\smallskip

 We defined in \cite{BDFT} a $\Phi^4_3$ measure as {\sl an} invariant measure of this Markovian dynamics on $\mcC^{-1/2-\varepsilon}(M)$. The uniqueness of such an invariant measure is proved in \cite{BailleulUniqueness}, so we freely talk in the sequel of {\sl the} $\Phi^4_3$ measure. These results are proved by building on Jagannath \& Perkowski's insight \cite{JP} that a clever change of variable turns the stochastic PDE \eqref{EqRenormalizedSPDE} into a PDE
\begin{equation} \label{EqJPFOrmulation}
(\partial_t-\Delta+1) v_r = -A_rv_r^3 - B_r\nabla v_r + Z_{2r}v_r ^2 + Z_{1r}v_r + Z_{0r}
\end{equation}
with random coefficients, whose solution theory is elementary, provided the coefficients are controlled in some appropriate spaces. The problems related to the low regularity of the spacetime white noise $\xi$ and the singular character of \eqref{EqSPDE} are all transferred to the question of proving the convergence in an appropriate space of the random coefficients. When formulated in this way there is no need to use the tools of regularity structures or paracontrolled calculus to set up an analytic framework for the study of \eqref{EqSPDE}. However the question of the convergence of the random coefficients is fairly non-trivial and remains to be dealt with separately. As a matter of fact the $r$-uniform control of one of the terms that appear in Jagannath \& Perkowski's reformulation can be obtained using a number of basic tools from paracontrolled calculus. Sections \ref{SectionParaproducts}, \ref{s:paradiffcommutators} and \ref{SectionCommutationParamultiplication} are dedicated to present these tools in a self-contained way. They are used in Section \ref{SectionProofThmA2} to prove a crucial $r$-uniform control on the above mentioned term. The random coefficients $A_r,B_r,Z_{2,r}, Z_{1r}, Z_{0r}$ are all continuous polynomial functions of eight distributions built from the Gaussian regularized noise $\xi_r$ by some elementary operations. Building on moment estimates, we formulate the problem of the convergence of the random coefficients as a problem of extension of some distributions which are a priori defined only outside a union of diagonals. In doing so, we follow Epstein \& Glaser's approach to renormalization. The analysis of this extension problem requires some tools from microlocal analysis that we explain in detail in sections \ref{sect_microlocal_est}, \ref{SectionProofMainThm2} and \ref{sect_compostion_para_ker}.

\smallskip

With this context in mind, we can now set up the specific notation. Denote by 
$$
\mathcal{L}\defeq\partial_t-\Delta+1
$$ 
the heat operator and by $\underline{\mathcal{L}}^{-1}$ its inverse with null initial condition at $t=-\infty$. Set 
$$
\X_r \defeq \sqrt{2}\underline{\mathcal{L}}^{-1}(\xi_r), \quad\Xtwo_r \defeq :\X^2\hspace{-0.08cm}:_r,\quad \IXtwo_r \defeq \underline{\mathcal{L}}^{-1} (\Xtwo_r), \quad \IXthree_r \defeq \underline{\mathcal{L}}^{-1}( :\X^3\hspace{-0.08cm}:_r ).
$$ 
These stochastic terms are first regularized, since $\xi_r$ is mollified, and then Wick renormalized. 

\medskip

{\it {\Large \S}2 --  A Cole-Hopf transform.} The main idea of \cite{JP} is to introduce a new Cole-Hopf transform which yields an optimal way to decompose the solution $u_r$ of \eqref{EqRenormalizedSPDE} in such a way that all the singularities of the SPDE are well-isolated. Setting
\begin{eqnarray*} \label{eq:decomp1}
u_r = \X_r - \IXthree_r + e^{-3\IXtwo_r} \big(v_{\textrm{ref},r} + v_r \big)
\end{eqnarray*}
where $v_{\textrm{ref},r}$ solves the equation
$$
\mathcal{L} v_{\textrm{ref},r} = 3e^{3\IXtwo_r} \left(\IXthree_r \Xtwo_r - b_r(\X_r + \IXthree_r)\right), \quad v_{\textrm{ref},r}(0)= 0,
$$ 
the function $v_r$ is the solution of Equation \eqref{EqJPFOrmulation} for some appropriate coefficients $A_r,B_r, Z_{i,r}$ and initial condition.

Here and throughout the paper, $C_T\mcC^\alpha$ denotes the space of those functions which are continuous in time on $[0,T]$ with values in $\mcC^\alpha$, the H\"older-Besov space introduced in Section~\ref{ss:Besovrecollection}.

\smallskip

In order to prove the local-in-time well-posedness of the equation for the remainder $v_r$, we need to control some random field \eqref{EqFormerThm1.1} built explicitly from the noise. Achieving this control is part of the present work, and is stated in the important Theorem \ref{thm_A_2_mfd}. We use paradifferential calculus on compact manifolds as a key ingredient in our proof of Theorem \ref{thm_A_2_mfd}. Sections \ref{SectionParaproducts} to \ref{SectionCommutationParamultiplication} are dedicated to giving a detailed exposition of paraproduct operators and para-decomposition operators on arbitrary closed manifolds, and to establish several commutator estimates on these objects, in the spirit of Gubinelli, Imkeller \& Perkowski' seminal work \cite{GIP}, as in Bailleul \& Bernicot's works \cite{BB1,BB2,BB3}. Our setting in these sections is general and our estimates hold for closed manifolds of any dimension. (A very nice related work for the first part is the recent paper of Guillarmou \& Poyferré \cite{BGP} where they developed some paradifferential calculus on manifolds. Some of our commutator estimates can be obtained from \cite{BGP} but we thought it would be useful to include some detailed proofs here to make our work more self-contained, more pedagogical and show clearly the mechanism behind the proofs.). 

\medskip

{\it {\Large \S}3 -- Convergence of the random coefficients in \eqref{EqJPFOrmulation}.} A major step in order to construct the $\Phi^4$ measure is to make sense of the following eight random (well-defined) distributions
\begin{equation*} \label{eq:stochasticobjects}
\left( \X_r, \Xtwo_r , \IXtwo_r, \IXthree_r, \IXthree_r\odot_i \X_r, \IXtwo_r\odot_i \Xtwo_r - \chi_i\frac{b_r}{3},  \chi_i\vert \nabla \IXtwo_r\vert^2-\chi_i\frac{b_r}{3}, \IXthree_r\odot_i \Xtwo_r - \chi_i b_r\X_r \right).
\end{equation*}
(They are well-defined as the noise used in their definition has been regularized.) The operator $\odot_i$ that appears here is the localized  resonance operator introduced in Section \ref{SectionParaproducts} and $\chi_i\in C^\infty_c(U_i)$ for a local chart $U_i\subset M$. In Section 4 of \cite{BDFT}, using renormalization, we proved that the preceding list of stochastic objects belongs to some appropriate Banach space of distributions, uniformly in $r>0$, provided the following distributional kernels are controlled microlocally. 

\smallskip

\begin{defi*} [Propagators and vertices in Feynman amplitudes] \label{def:listkernels}
We define the following collection of distributional kernels
{
\begin{equation*} \begin{split}
&\underline{\mathcal{L}}^{-1}\big((t,x),(s,y)\big) \defeq {\bf 1}_{(-\infty,t]}(s) \, e^{(t-s)(\Delta-1)}(x,y) \in \mathcal{D}^\prime( \mathbb{R}^2\times M^2 ) \\ 
& \Delta \underline{\mathcal{L}}^{-1}\big((t,x),(s,y)\big) \defeq {\bf 1}_{(-\infty,t]}(s) \, \left(\Delta e^{(t-s)(\Delta-1)}\right)(x,y) \in \mathcal{D}^\prime( \mathbb{R}^2\times M^2 )  \\
&G^{(p)}\big((t,x),(s,y)\big) \defeq  \left(\Big\{e^{\vert t-s\vert (\Delta-1)}(1-\Delta)^{-1}\Big\}(x,y)\right)^p \in  \mathcal{D}^\prime( \mathbb{R}^2\times M^2 ) , \qquad (1\leqslant p\leqslant3) 
 \\
&[\odot_i](x,y,z) \defeq \sum_{\vert k-\ell\vert\leqslant 1} P^i_k(x,y) \, \widetilde{P}^i_\ell(y,z)  .\\
\end{split} 
\end{equation*}
}
where $P^i_k$ and $\widetilde{P}^i_\ell$ stand for some generalized Littlewood-Paley-Stein projectors that we introduce in Paragraph \ref{sss:definitionLPSproj}, and  $(\eta_i)_{i\in I}$ a  partition of unity.
\end{defi*}

\smallskip

These kernels form the \emph{elementary building blocks} of the Feynman amplitudes that one needs to control analytically to probe the regularity of the stochastic objects in the limit where $r>0$ goes to $0$. The microlocal description of these kernels is given in Theorem \ref{mainthm2}. We need to recall some terminology from microlocal analysis before we can state it. Let $\mathcal{X}$ denote some ambient manifold and
$\mathcal{U}\subset \mathcal{X}$ some open subset of $\mathcal{X}$. For every closed conic set $\Gamma\subset T^\bullet\mathcal{U}$, we denote by $\mathcal{D}^\prime_\Gamma(\mathcal{U})$ the space of distributions whose wave front set lies in $\Gamma$. This space is considered as a locally convex topological vector space endowed with the normal topology~\cite{BDFT}. Let $\mathcal{Y}\subset \mathcal{X}$ denotes a submanifold of $\mathcal{X} $, $\rho$ a scaling field relative to $\mathcal{Y}$, $\mathcal{U}$ some open subset which is stable under the semiflow of $\rho$: $e^{-s\rho}(\mathcal{U})\subset \mathcal{U}$ and $\Gamma\subset T^\bullet \mathcal{U}$ a closed conic set which is stable under the semiflow of $\rho$. Then we will denote by $\mathcal{S}^a_\Gamma(\mathcal{U})$ the set of distributions $T$ such that the family of distributions $ \left(e^{as} e^{-s\rho*}T\right)_{s\geqslant 0} $ is bounded in $\mathcal{D}^\prime_\Gamma(\mathcal{U})$.

The other result which we prove in the present article is that these kernels satisfy some precise micro-local estimates. Theorem \ref{mainthm2} states that the kernels of the operators {$\underline{\mathcal{L}}^{-1},  \Delta\underline{\mathcal{L}}^{-1}, G_r^{(i)}$} and $[\odot_i]$  are in different functional spaces of the form $\mathcal{S}^a_\Gamma(\mathcal{U})$ for some ambient spaces $\mathcal{U}$, scaling exponents $a$ and wavefront sets $\Gamma$. 

\smallskip

\begin{thm} \label{mainthm2}
In the conventions introduced above, we have the following microlocal estimates:
\begin{itemize}
	\item[--] The kernel $\underline{\mathcal{L}}^{-1}$ has scaling exponent $-3$ and wavefront set 
	$$
	N^*\left(\{t=s\}\times \{x=y\}\subset \mathbb{R}^2\times M^2 \right).
	$$
	{
	\item[--] The kernel $\Delta\underline{\mathcal{L}}^{-1}$ has scaling exponent $-5$ and wavefront set 
	$$
	N^*\left(\{t=s\}\times \{x=y\}\subset \mathbb{R}^2\times M^2 \right).
	$$}
	\item[--] Uniformly in $r\in[0,1]$, the kernels $G_r^{(p)}$ have scaling exponent $-p$ and wavefront set 
	$$
	N^*\left(\{t=s\}\subset \mathbb{R}^2\times M^2 \right)\cup N^*\left(\{t=s\}\times \{x=y\}\subset \mathbb{R}^2\times M^2 \right).
	$$
	\item[--] The kernel $[\odot_i]$ has scaling exponent $-6$ and wavefront set
	$$
	 N^*\left(\{x=y=z\}\subset M^3 \right)
	$$
\end{itemize}
Any kernel $K$ of the above list satisfies some local diagonal bounds of the form
$$
\vert\partial_{\sqrt{t},\sqrt{s},x,y}^\alpha K\vert\lesssim \left(\sqrt{\vert t-s \vert}+\vert x-y\vert \right)^{a-\vert\alpha\vert}
$$
for the corresponding scaling exponent $a$. 
\end{thm}

\smallskip

\begin{proof}
It is proved in Lemmas~\ref{l:heatvsparabolic} and Proposition~\ref{p:propagatorinparabolic} that the kernels $K_{t-s}(x,y)$ of the two operators $e^{(t-s)(1-\Delta)}$ and $e^{(t-s)(1-\Delta)}(1-\Delta)^{-1}$ belong to the parabolic calculus $\Psi_P^\alpha$ introduced in Definition~\ref{defi:flatparabolic} respectively for $\alpha=-1$ and $\alpha=-2$. They therefore satisfy some bounds of the form
\begin{equation*} 
\big\vert\partial^\beta_{\sqrt{t},\sqrt{s},x,y} K_{t-s}(x,y)\big\vert \lesssim \left(\sqrt{|t-s|}+\vert y-x\vert\right)^{a- \vert\beta\vert}\,,
\end{equation*}
for $a=-d-2-2\alpha=-5-2\alpha$. It is immediate to read from the above that $\underline{\mathcal{L}}^{-1}$ and $G_r^{(p)}$ verify the same bounds respectively for $a=-3$ and $-p$, $\underline{\mathcal{L}}^{-1}$ being obtained from $e^{(t-s)(1-\Delta)}(x,y)$ by multiplication by a bounded indicator function, and the kernel of $G_r^{(p)}$ being just the $p$-th power of that of $e^{(t-s)(1-\Delta)}(1-\Delta)^{-1}(x,y)$. They therefore satisfy the assumptions of Lemma~\ref{l:topologytwopointfunctiong}, so that the claim holds for these propagators. 

Let us now deal with the kernel of $\Delta\underline{\mathcal{L}}^{-1}$. For simplicity, we just focus on the kernel $K_{t-s}(x,y)$ of the operator $\Delta e^{(t-s)(\Delta-1)}$. First, we would like to insist on the fact that $K$ is only defined in the sense of distributions, since it is not an $L^1_{\mathrm{loc}}$ function, so that the Laplacian acts on the heat kernel in the sense of distributions. 
To prove the statement, we choose some scaling field $\rho$ w.r.t. the diagonal $\{t=s\}\times \mathbf{d}_2$. Then, we note that
the family of differential operators of order $2$
defined as
$$ Q_u:=e^{-2u}e^{-u\rho*}\Delta e^{u\rho*} $$
is bounded as a differential operator of degree $2$ with $C^\infty$ coefficients uniformly in $u\in [0,+\infty)$.

This can be seen in local coordinates where we choose a scaling which
takes the simple form:
$$ (t,s,x,y)\longmapsto (e^{-2u}(t-s)+s,s,e^{-u}(x-y)+y,y), $$
if $\Delta=g^{\mu\nu}(x)\partial_{\mu}\partial_\nu+b^\mu(x)\partial_\mu$,
then $Q_u$ reads in local coordinates as
\begin{align*}
 Q_u=e^{-2u}(e^{2u}g^{\mu\nu}(e^{-u}(x-y)+y)\partial_{\mu}\partial_\nu+e^ub^\mu(e^{-u}(x-y)+y)\partial_\mu)
 \\=g^{\mu\nu}(e^{-u}(x-y)+y)\partial_{\mu}\partial_\nu+e^{-u}b^\mu(e^{-u}(x-y)+y)\partial_\mu 
 \end{align*}
which is clearly bounded in $C^\infty$ uniformly in $u\in [0,+\infty)$. Then, we have the exact identity: 
$$ e^{-5u} e^{-u\rho*} \big(\Delta\underline{\mathcal{L}}^{-1}\big)=Q_u \big(\underbrace{e^{-3u} e^{-u\rho*}\underline{\mathcal{L}}^{-1}}\big)\,,$$
where the term underbraced is bounded in $\mathcal{D}^\prime_{N^*(\{t=s\}\times \mathbf{d}_2)}(\mathbb{R}^2\times M^2)$
by the previous result on $\underline{\mathcal{L}}^{-1} $, therefore
$Q_u e^{-3u} e^{-u\rho*}\underline{\mathcal{L}}^{-1}$ is bounded in 
$\mathcal{D}^\prime_{N^*(\{t=s\}\times \mathbf{d}_2)}(\mathbb{R}^2\times M^2)$ since the action of bounded differential operators is bounded in the $\mathcal{D}^\prime_\Gamma$ topology, from which we can conclude that indeed $\Delta\underline{\mathcal{L}}^{-1}$ belongs to $\mathcal{S}^{-5}_{N^*(\{t=s\}\times \mathbf{d}_2)}$.

Finally, the case of the kernel $[\odot_i](x,y,z)$ is treated in Lemma~\ref{lem:7_2}.
\end{proof}

\smallskip

The present work is organised as follows. In Section~\ref{SectionParaproducts}, we introduce some paraproducts on compact manifolds, and show that the usual properties of Besov spaces on the plane extend to this context. We also include in Section \ref{sect_LPS_meet_parabolic_cal} several commutator estimates that involve both pseudo-differential operators and Littlewood-Paley-Stein projectors. In Sections~\ref{s:paradiffcommutators} and \ref{SectionCommutationParamultiplication}, we prove a commutator estimate between the heat kernel and paraproducts. Theorem~\ref{thm_A_2_mfd} is proved in Section~\ref{SectionProofThmA2}. In Section~\ref{sect_microlocal_est}, we develop for the purpose of proving the estimates of Theorem~\ref{mainthm2} on $\underline{\mathcal L}^{-1}$ and $G^{(p)}$ a calculus of operators on closed Riemannian manifolds of any dimension whose Schwartz kernels have parabolic singularities, called parabolic calculus. The proof of Theorem~\ref{mainthm2} is completed in Section~\ref{SectionProofMainThm2}, where we prove the estimate on the kernel of the resonant product $[\odot_i]$. Finally, Section~\ref{sect_compostion_para_ker} is devoted to the proof of an estimate on the composition of elements of the parabolic calculus with elements of the classical heat calculus, which is used in Section 5 of \cite{BDFT} to obtain the explicit expressions~\eqref{EqDefnArBr} for the counterterms $a_r,b_r$.

\bigskip

\subsubsection*{Related works.}

We tried to reach a certain compromise in writing this work.
\begin{itemize}
	\item Our approach is less general than the works of Bailleul, Bernicot \& Frey ~\cite{BB1,BB2,BB3} in the sense we restrict our study to \textsl{smooth} compact Riemannian manifolds whereas the cited works 
works in the more general geometric background of metric measure spaces that have the volume doubling property. Another limitation compared to the mentioned work is that we only develop paracontrolled calculus in the first order setting. This is all we need in our study of the dynamics \eqref{EqSPDE}.   

	\item Following an established classical tradition in microlocal analysis, we develop most objects first on $\mathbb{R}^d$ for operators with variable coefficients. Then, using partitions of unity and local charts, we explain what kind of results can be transferred to the manifold setting. This implies that most of the objects we define in our analysis, the quantizations, the projectors, are non-canonical with respect to the metric $g$. We also heavily rely on Fourier analysis. This makes some of our proofs easier than in~\cite{BB1,BB2,BB3} since we can use existing results on paradifferential and microlocal analysis on $\mathbb{R}^d$. We loose in generality, covariance and the intermediate analytical objects we use are not geometrically intrinsic. We gain in simplicity and flexibility, while a number of these results seems out of reach for the methods of \cite{BB1,BB2,BB3} as we can always compare our analytical objects with those coming from classical pseudo-differential analysis in the $C^\infty$ setting. This was for instance very useful in \cite{BDFT} where we used commutator estimates to give explicit expression for the counterterms. Our work includes a detailed microlocal study of all kernels appearing in the stochastic estimates necessary to prove the well-posedness of \eqref{EqSPDE}, which is where our work might differ from \cite{BB1,BB2,BB3}. For instance, we rely heavily on the notion of wavefront set in our smooth setting (this would not be obvious to define on metric measure spaces with doubling property) and control the wavefront set of Feynman amplitudes appearing in the estimate for the stochastic terms.	 

\item  Finally, we think that for the reasons explained above, the techniques we develop provide a useful toolbox for any SPDE approach using first order paracontrolled expansions on manifolds and we believe our analytical results are not limited to the $\Phi^4_3$ model with the Jagannath--Perkowski ansatz. For instance, all our results  would also provide some core analytical estimates for the compact manifold version of the Anderson model in dimension $ \leqslant 3$ and could also be used (with modifications) to deal with other paradifferential approaches to $\Phi^4_3$ such as the one of Catellier--Chouck~\cite{CatellierChouk}, Mourrat-Weber~\cite{MW17},  Gubinelli--Hofmanova~\cite{GH}, Barashkov--Gubinelli~\cite{BG},  Nam--Zhu-Zhu~\cite{NZZ}.

\end{itemize}

The way we treat paradifferential analysis is quite close to the one of Guillarmou, Poyferré \& Bonthonneau~\cite{BGP} in the sense all the ``hard analysis'' is worked out on $\mathbb{R}^d$ and the globalization is done separately. For instance, our decomposition of the product of smooth functions in paraproducts and resonant products is closely related to~\cite[Prop 2.17]{BGP}. The results from section~\ref{s:paradiffcommutators} rely on results of Bony, H\"ormander and Meyer on the twisted class $\tilde{\Psi}_{1,1}$ and certain results from this section probably follow from~\cite[subsection 2.1]{BGP}. However we preferred to give detailed proofs which are hard to find in the literature and our approach is more pedestrian as we hope it may help non experts to enter the subject. However, since the tools of the present paper deal with \textbf{nonlinear} PDE's, we need slightly stronger results than those in~\cite{BGP}: the multiple products decompositions from subsection~\ref{SubsectionTripleParamultiplication} and more involved multilinear commutator estimates which are common use in all the recent paracontrolled approach.

\medskip

\noindent \textit{Notation -- We will denote by $d$ the Riemannian distance on $(M,g)$.}

\bigskip

\noindent {\textbf{\textsf{Acknowledgements --}}}
First, we would like to warmly thank a careful reviewer for a very detailed report which helped us correct many mistakes and make the paper become much better.
We would like to thank C. Bellingeri, C. Brouder, C. Dappiaggi, P. Duch, C. Gérard, C. Guillarmou, F. Hélein, K. L\^e, D. Manchon, A. Mouzard, P.T. Nam, S. Nonnenmacher, V. Rivasseau, G. Rivière, F. Vignes-Tourneret for interesting questions, remarks, comments on the present work when we were in some preliminary stage and also simply for expressing some interest and motivating us to pursue. Special thanks are due to Y. Bonthonneau, C. Brouder, J. Derezinski, M. Hairer and M. Wrochna for their very useful comments on a preliminary version of~\cite{BDFT}. N.V.D acknowledges the support of the Institut Universitaire de France. The authors would like to thank the ANR grant SMOOTH "ANR-22-CE40-0017", QFG "ANR-20-CE40-0018" and MARGE "ANR-21-CE40-0011-01" for support.

\section{Paraproducts on compact manifolds made simple}
\label{SectionParaproducts}

We provide in this section a direct construction of some (family of) paraproduct and resonant operators on $M$ from their $\bbR^d$ analogue. This approach has the advantage that we can directly import on $M$ the results known on $\bbR^d$ at low cost and is closely related to~\cite[Prop 2.17]{BGP}. There are obviously other approaches to the subject with different advantages. Bernicot's approach via the heat semigroup \cite{Bernicot, BernicotSire} probably has the most general geometric scope. It needs to be refined as in Bailleul \& Bernicot's work \cite{BB1} to deal with Besov spaces of negative regularity. See e.g. Mouzard's works \cite{MouzardWeyl, MouzardThesis} for an implementation of this approach in the setting of a smooth closed manifold.

\smallskip

We will use along some of the proof of this section some results stated and proved in Section \ref{s:paradiffcommutators} and Section \ref{SectionCommutationParamultiplication} that are independent of the content of the present section.

\smallskip

We recall in Section \ref{SubsectionParaproductRd} the definition of the paraproduct operator on $\bbR^d$. The Besov spaces over $M$ are introduced in Section \ref{ss:Besovrecollection}, where we extend to these spaces the well-known fractional Leibniz and interpolation estimates and prove some Schauder-type estimate for a certain class of pseudo-differential operators. A family of paraproduct and resonant operators is introduced in Section \ref{ss:definitionLPSprojectorsparaproduct}. These objects naturally come in family as they depend on partitions of unity and similar side functions. They have the analytic properties that we expect. Last, Section \ref{SubsectionTripleParamultiplication} deals with the iteration of two paraproduct operators and paralinearisation.

\medskip

\subsection{Paraproduct on $\mathbb{R}^d$$\boldmath{.}$ \hspace{0.1cm}}
\label{SubsectionParaproductRd}

Recall we can find functions $\chi$ and $\psi$ such that
$$
1 = \chi + \sum_{k=1}^\infty \psi(2^{-k}\cdot),
$$ 
wtih $\text{support}(\chi)\subset \{ 0\leqslant \vert \xi\vert\leqslant 4\}$ and $\text{support}(\psi)\subset \{ 1\leqslant \vert \xi\vert\leqslant 4\}$. We denote by $\Delta_0=\chi(D)$ and $\Delta_j=\psi(2^{-j}D)$ for $j\geqslant1$,  that is 
$$
\Delta_jf\defeq \mathcal{F}^{-1}\left(\psi(2^{-j}\cdot)\widehat{f} \, \right) 
$$  
localizing (this is not a projector) $f$ on the Fourier dyadic shell of size $\big\{ 2^{j-1}\leqslant \vert \xi\vert\leqslant 2^{j+1} \big\}$ -- a corona in Fourier space. Let $\eta_j:= \mathcal{F}^{-1} ( \psi(2^{-j}\cdot)),\, j\geqslant 0$, and $\eta_{-1} = \mathcal{F}^{-1}(\chi)$, then we have
\begin{equation}\label{eq_LP_block_convolution}
\Delta_j f= \eta_k\star f.
\end{equation}
The element $\Delta_jf$ is sometimes called a Littlewood-Paley block. The Littlewood-Paley decomposition of a distribution $f$ reads
$$ \Delta_0(f) +\sum_{j=1}^\infty \Delta_j(f) . $$
We define projectors on lower Fourier modes as: 
\begin{eqnarray*}
S_j(f) \defeq \Delta_0(f) + \sum_{k=1}^j \Delta_k(f).
\end{eqnarray*}
Then the paraproduct is defined by
\begin{eqnarray} \label{eq_flat_paraproduct}
f\prec g=  \sum_j S_{j-2}(f)\Delta_j(g)\,,
\end{eqnarray}
where the product $S_{j-2}(f)\Delta_j(g)$ is supported in Fourier spacev in some enlarged  corona
$ \frac{1}{4}\,2^{j}\leqslant \vert \xi\vert\leqslant \frac{9}{4} \,2^{j} $. This observation relies on the fundamental fact that the Fourier support of $fg$ is contained in the sum of the Fourier supports of $f$ and $g$, we refer to~\cite[p280-291]{Meyer2} for more details.

We also define the resonant term by 
\begin{eqnarray} \label{eq_flat_resonnat}
f\odot g=  \sum_{|i-j|\leqslant1} S_{i}(f)\Delta_i(g)\,.
\end{eqnarray}

Recall the classical definition of Besov norms $ \Vert . \Vert_{B^s_{p,q}(\mathbb{R}^d)}$ for $s\in \mathbb{R}, (p,q)\in [1,+\infty]^2$
\begin{eqnarray*} \label{def:besovnormRd}
\Vert u \Vert_{B^s_{p,q}(\mathbb{R}^d)}\defeq \big\Vert \Vert 2^{js}\Delta_ju \Vert_{L^p(\mathbb{R}^d)} \big\Vert_{\ell^q(\mathbb{N})}.
\end{eqnarray*}
The corresponding Banach space $B^s_{p,q}(\bbR^d)$ is obtained by completion of Schwartz functions on $\mathbb{R}^d$ using the above norms.

\subsection{Besov, Leibniz and Schauder$\boldmath{.}$ \hspace{0.1cm}}
\label{ss:Besovrecollection}

Since $M$ is compact we can define the Besov space on $M$ via a finite cover $(U_i,\kappa_i)_{i\in I}$  and partition of unity $(\chi_i)_{i\in I}$ subordinated to $(U_i)_{i\in I}$
\begin{equation*}
B^{s}_{p,q}(M)=\left\lbrace u\in \mathcal{D}'(M): \Vert  u\Vert_{{B}^s_{p,q}(M)} \defeq\sum_{i}^N  \Vert \kappa_{i*}(\chi_{i}u) \Vert_{B ^s_{p,q} (\mathbb{R}^d)} <\infty \right\rbrace.
\end{equation*}
This choice of norm depends on the cover and the partition of unity. Different choices lead to equivalent norms on the same space. The space ${B}^\alpha_{\infty, \infty}(M)$, for $\alpha\in \bbR$, will be denoted by $\mathcal{C}^\alpha(M)$ and 
$$
\Vert u\Vert_\alpha\defeq\Vert  u\Vert_{{B}^\alpha_{\infty, \infty}(M)}.
$$ 
From the above definition of Besov space, we deduce the fractional Leibniz rule on $M$, which is stated below.

\smallskip

\begin{prop} \label{PropLeibniz}
Let $\alpha>0, r\in \mathbb{N}$ and $ p,p_1,p_2,q \in [1,\infty]$ such that
\[  \frac{1}{p} =\frac{1}{p_1} + \frac{1}{p_2}.
\]
Then \[\| u^{r+1} \|_{\mcB^\alpha_{p,q}(M)}\lesssim \| u^r \|_{L^{p_1}(M)} \| u \|_{\mcB^{\alpha}_{p_2,q}(M)}.\]
\end{prop}

\smallskip

\begin{proof} 
Let $(\widetilde\chi_i)_{i\in I}$ be another partition of unity subordinated to $(U_i)_{i\in I}$ and such that $\widetilde\chi_i=1$ on the support of $\chi_i$. We have
\begin{eqnarray*}
\Vert u^{r+1} \Vert_{\mcB^\alpha_{p,q}(M)}=\sum_i \Vert  \kappa_{i*}\left(\chi_iu^{r+1}\right) \Vert_{B^\alpha_{p,q}(\mathbb{R}^d)}  \leqslant
\sum_i \big\Vert  \kappa_{i*}\left(\chi_i \right)  \kappa_{i*}\left((\widetilde{\chi}_iu)^{r+1}\right) \big\Vert_{B^\alpha_{p,q}(\mathbb{R}^d)}   \\
\lesssim \sum_i\Vert  \kappa_{i*}  \left((\widetilde{\chi}_iu)^{r+1} \right)\Vert_{B^\alpha_{p,q}(\mathbb{R}^d)}\lesssim 
\sum_i\Vert   \kappa_{i*}\left( (\widetilde{\chi}_iu)^{r} \right)\Vert_{L^{p_1}(\mathbb{R}^d)}\Vert  \kappa_{i*} \left( \widetilde{\chi}_iu\right) \Vert_{B^\alpha_{p_2,q}(\mathbb{R}^d)},
\end{eqnarray*}
where we used the fact that the
multiplication by $C^\infty$ functions is continuous on Besov spaces and the fractional Leibniz estimate holds on $\mathbb{R}^d$~\cite[Proposition A 7 and Corollary A.8]{MW17}, the implicit constant in the above estimate only depends on $(\chi_i)_i$. For every fixed $i$, 
$$
\big\Vert  \kappa_{i*}  \left((\widetilde{\chi}_iu)^{r}\right) \big\Vert_{L^{p_1}(\mathbb{R}^d)}\lesssim \Vert (\widetilde{\chi}_iu)^r \Vert_{L^{p_1}(M)}
$$ 
where the implicit constant depends only on the Jacobian of $\kappa_i$ and also

\begin{equation*} \begin{split}
\Vert  \kappa_{i*} \left( \widetilde{\chi}_iu\right) \Vert_{B^\alpha_{p_2,q}(\mathbb{R}^d)} &= \big\Vert  \kappa_{i*} \left( \sum_j\widetilde{\chi}_i\chi_ju\right) \big\Vert_{B^\alpha_{p_2,q}(\mathbb{R}^d)} \lesssim \sum_j\Vert  \kappa_{i*} \left( \widetilde{\chi}_i\chi_ju\right) \Vert_{B^\alpha_{p_2,q}(\mathbb{R}^d)}   \\
&\lesssim \sum_j \big\Vert \left(\kappa_j\circ\kappa_i^{-1}\right)_* \kappa_{i*} \left( \widetilde{\chi}_i\chi_ju\right) \big\Vert_{B^\alpha_{p_2,q}(\mathbb{R}^d)}
=\sum_j \Vert  \kappa_{j*} \left( \widetilde{\chi}_i\chi_ju\right) \Vert_{B^\alpha_{p_2,q}(\mathbb{R}^d)}   \\
& \lesssim \sum_j \Vert  \kappa_{j*} \left( \chi_ju\right) \Vert_{B^\alpha_{p_2,q}(\mathbb{R}^d)} =\Vert u \Vert_{B^\alpha_{p_2,q}(M)}
\end{split} \end{equation*}
where we used in a crucial way the diffeomorphism invariance of Besov spaces~\cite{BCD}, the compact support in $U_i\cap U_j$ of each piece $  \widetilde{\chi}_i\chi_ju $ and where for every $j$ we transported the function by the local diffeomorphism $\kappa_j\circ \kappa_{i}^{-1}: \kappa_i(U_i)\mapsto \kappa_j(U_j) $. The last equality just follows from the definition. We finally get
\begin{eqnarray*}
\Vert u^{r+1} \Vert_{B^\alpha_{p,q}(M)}\lesssim \sum_i\Vert (\widetilde{\chi}_iu)^r \Vert_{L^{p_1}(M)}
\Vert u \Vert_{B^\alpha_{p_2,q}(M)}\lesssim \Vert u^r \Vert_{L^{p_1}(M)}
\Vert u \Vert_{B^\alpha_{p_2,q}(M)}
\end{eqnarray*}
since the multiplication by $\widetilde{\chi}_i$ is continuous on the $L^p$ spaces and since the sum over $i$ is finite as $M$ is compact.
\end{proof}

\smallskip

\begin{prop} \label{PropInterpolation}
Let $\alpha_1, \alpha_2 \in \mathbb{R}$ and $ p_1,p_2,q_1,q_2 \in [1,\infty]$ and $\theta\in [0,1]$. Define $\alpha= \theta \alpha_1+(1-\theta)\alpha_2 $, and $p,q\in [1,\infty]$ by
\[  \frac{1}{p} =\frac{\theta}{p_1} + \frac{1-\theta}{p_2}\, \text{ and  } \, \frac{1}{q} =\frac{\theta}{q_1} + \frac{1-\theta}{q_2}.
\]
Then \[  \| u \|_{\mcB^\alpha_{p,q}(M)} \lesssim \| u \|^\theta_{\mcB^{\alpha_1}_{p_1,q_1}(M)}  \| u \|^{1-\theta}_{\mcB^{\alpha_2}_{p_2,q_2}(M)}.
\]
\end{prop}

\smallskip

\begin{proof}
We have
\begin{eqnarray*}
 \Vert u \Vert_{\mcB^\alpha_{p,q}(M)} = \sum_i \Vert \kappa_{i*} \left(\chi_i u\right) \Vert_{B^\alpha_{p,q}(\mathbb{R}^d)} \leqslant \sum_i \Vert \kappa_{i*} \left(\chi_i u\right) \Vert^\theta_{B^{\alpha_1}_{p_1,q_1}(\mathbb{R}^d)}\Vert \kappa_{i*} \left(\chi_i u\right) \Vert^{1-\theta}_{B^{\alpha_2}_{p_2,q_2}(\mathbb{R}^d)} \\   \leqslant
 \sum_i \Vert \kappa_{i*} \left( u\right) \Vert^\theta_{B^{\alpha_1}_{p_1,q_1}(\mathbb{R}^d)}\Vert \kappa_{i*} \left( u\right) \Vert^{1-\theta}_{B^{\alpha_2}_{p_2,q_2}(\mathbb{R}^d)} 
 \lesssim 
\Vert  u \Vert^\theta_{\mcB^{\alpha_1}_{p_1,q_1}(M)}\Vert u \Vert^{1-\theta}_{\mcB^{\alpha_2}_{p_2,q_2}(M)}  
\end{eqnarray*}
where we used again finiteness of the sum over $i$, continuity of the multiplication by $C^\infty$ functions and the interpolation inequality on $\mathbb{R}^d$.
\end{proof}

\smallskip

These two above results play an essential role in the proof of the long time existence and the coming down from infinity for the dynamical $\Phi^4_3$ model \cite[Section 2.2]{BDFT}.

\smallskip
 We recall  here the weighted spaces $\llparenthesis \alpha,\beta \rrparenthesis$ for $\alpha>0$ and $ \beta \geqslant0$   (cf. \cite[Section 2.1]{BDFT}) made up of all functions $v\in C((0,T],\mcC^\beta(M))$ such that
\begin{equation}\label{eq_lim_at_0}
t^\alpha \Vert u(t)\Vert_{L^\infty} \underset{t\rightarrow 0}{\longrightarrow} 0
\end{equation}
and
\begin{equation}\label{eq_w_norm}
\Vert v\Vert_{\llparenthesis \alpha,\beta \rrparenthesis} \defeq  \max\left\{ \underset{0<t\leqslant T}{\sup}\,t^\alpha\Vert v(t)\Vert_{\mcC^\beta},\sup_{0\leqslant t\neq s\leqslant T} \frac{\| t^\alpha v(t)-s^\alpha v(s)  \|_{L^\infty}}{|s-t|^{\beta/2}} \right\} < \infty.
\end{equation}

Then we have the following lemma

\begin{lemm}
For all $\beta\in \mathbb{R}_{\geqslant0}\setminus \mathbb{N}$, and $\alpha\in (0,1)$,
then  we have the following  Schauder estimate

\begin{eqnarray}\label{eq_Schauder_w}
\| \mathcal{L}^{-1} u\|_{\llparenthesis \alpha,\beta \rrparenthesis} \lesssim \underset{0<t\leqslant T}{\sup}\,t^\alpha\Vert u(t)\Vert_{\mcC^{\beta-2}}\, ,
\end{eqnarray} 
where $ \mathcal{L}^{-1} u=\int_{0}^t e^{(t-s)(\Delta-1)} u(s, \cdot)\rmd s,$
and  for all $0\leqslant\delta\leqslant\min(\beta, 2\alpha)$
\begin{eqnarray}\label{eq_weighted_est}
\|u\|_{\llparenthesis \alpha-\delta/2, \beta-\delta  \rrparenthesis }\lesssim \|u\|_{\llparenthesis \alpha,\beta \rrparenthesis}.
\end{eqnarray}
\end{lemm}
\begin{proof}
The proof of the Schauder estimate~\eqref{eq_Schauder_w} follows as in the flat case (cf. the proof of Lemma 2.6  \cite[p. 265-266]{GubinelliPerkowskiKPZReloaded} where they used the following estimate (cf. \cite[Lemma A.9]{GIP}) 
\begin{eqnarray}\label{eq_est_Schauder_1}
t^{\alpha}\| \mathcal{L}^{-1}u (t) \|_{\mcC^\beta} \lesssim \sup_{s\in [0,t]}( s^{\alpha}\|u(s)\|_{\mcC^{\beta-2}}).
\end{eqnarray}
For the reader's convenience, we  give  here a proof for the  estimate \eqref{eq_est_Schauder_1} in our case.  
Denote $M= \sup_{s\in [0,t]} s^{\alpha}\|u(s)\|_{\beta-2}$.  Using  the smoothing properties of the heat kernel (see for example Corollary \ref{lemma2}), we have
\begin{align*}
t^{\alpha}\| \mathcal{L}^{-1}u (t) \|_{\mcC^\beta} &=   t^\alpha\|  \int_{0}^t   e^{s(\Delta-1)}  u(t-s, \cdot)\rmd s\|_{\mcC^\beta} \leqslant    t^\alpha \int_{0}^t  \| e^{s(\Delta-1)}  u(t-s, \cdot) \|_{\mcC^\beta}\rmd s\\
&\leqslant   t^\alpha \int_{0}^t  t^{-1}\|  u(t-s, \cdot) \|_{\mcC^{\beta-2} } \leqslant M \int_0^t t^{-1+\alpha}  (t-s)^{-\alpha}\rmd s\\
&=M\int^1_0(1-u)^{-\alpha} \rmd u\lesssim M,
\end{align*} 
 as required.

\medskip

 The second estimate \eqref{eq_weighted_est}  follows from the same  argument of the  proof for Lemma 6.8 \cite[p. 266]{GubinelliPerkowskiKPZReloaded}. 
First, by the definition of $\llparenthesis \alpha,\beta \rrparenthesis$ we have $t^\alpha\|\Delta_j \kappa_{i*} (\chi_i u)\|_{L^\infty } \leqslant \min\{2^{-j\beta}, t^{\beta/2}\} \|u\|_{\llparenthesis \alpha,\beta \rrparenthesis} $.  Hence, for $0\leqslant\delta \leqslant\min \{\beta, 2\alpha\} $, we have
$t^\alpha\|\Delta_j \kappa_{i*} (\chi_i u)\|_{L^\infty} \leqslant 2^{-j(\beta-\delta)} t^{\delta/2} \|u\|_{\llparenthesis \alpha,\beta \rrparenthesis} .$ This implies $t^{\alpha-\delta/2} \| u\|_{\beta-\delta}\lesssim \|u\|_{\llparenthesis \alpha,\beta \rrparenthesis}$. Then the  estimate about the temporal regularity in \eqref{eq_w_norm} is deduced from the condition \eqref{eq_lim_at_0} using  \cite[Lemma A.1.]{GubinelliPerkowskiKPZReloaded}.
\end{proof}

\medskip

\begin{defi*}  
Given  $0\leqslant \delta, \rho \leqslant 1, m\in \mathbb{R}$ and an open set $U\subset \mathbb{R}^d$, a function $a\in C^\infty(U\times \mathbb{R}^d)$ is said to be in the class $S^m_{\rho,\delta}(U\times \mathbb{R}^d)$ if 
$$
\sup_{x\in K} \big\vert \partial_x^\alpha\partial_\xi^\beta a(x,\xi)\big\vert \leqslant C_{\alpha\beta, K}(1+\vert \xi\vert)^ {{m-\vert\beta\vert\rho+\vert\alpha\vert\delta} }
$$  
for all multi--indices $(\alpha,\beta)$ and compact subset $K\Subset U$. We define $Op(a)$ as the operator acting on $u\in \mathcal{S}(\mathbb{R}^d)$ as
\begin{eqnarray*}
Op(a)(u) \defeq \frac{1}{(2\pi)^d}\int_{\mathbb{R}^d\times \mathbb{R}^d} a(x,\xi)e^{i\xi\cdot(x-y)}u(y)\rmd\xi \rmd y.
\end{eqnarray*} 
When $a\in S^m_{\rho,\delta}(\mathbb{R}^d\times\mathbb{R}^d)$ the operator $Op(a)$ is said to belong to the class $\Psi^m_{\rho,\delta}(\mathbb{R}^d)$.
\end{defi*}

\smallskip

The next statement gives a Schauder-type estimate for a special kind of  pseudo-differential operators. The reader will find more details on this class of operators, and on the reason for considering them in this work, in Section~\ref{s:paradiffcommutators}.

\smallskip

\begin{prop}[Schauder estimates for some pararegularized $S^{m}_{1,1}$] \label{l:simplifiedprooftwistedclass}
Let $\alpha\in \mathbb{R}$, $(p,q)\in [1,+\infty]^2$, let  $c\in S^m_{1,1}(\mathbb{R}^d)$ be such that there is $0<K<1$ so that, denoting by $\widehat{c}$ the Fourier transform of $c(x,\xi)$ in the variable $x$, we have that $\widehat{c}(\eta,\xi)$ is supported in $\{\vert \eta\vert\leqslant K\vert\xi\vert\}$. Then the map
$$
Op(c):B_{p,q}^\alpha(\bbR^d) \rightarrow B_{p,q}^{\alpha-m}(\bbR^d)
$$ 
is well-defined and continuous for all real numbers $\alpha$.
\end{prop}

\smallskip

\begin{proof}
Pick $v\in B^\alpha_{pq}(\bbR^d)$. Our goal is to control the $L^p$ norm of $\psi(2^{-i}\vert D_x\vert)Op(c)v$ when the index $i$ gets large. We start with the explicit identity: 
\begin{eqnarray*}
\psi(2^{-i}\vert D_x\vert)Op(c)v=\mathcal{F}^{-1}_\eta\left(\psi(2^{-i}\eta)\int_{\xi\in \mathbb{R}^d}  \widehat{c}(\eta-\xi,\xi)\widehat{v}(\xi)\rmd\xi \right).
\end{eqnarray*}
This follows from the definition of $Op$ and elementary computations with the Fourier transform. Note that the integrand in $\int_{\xi\in \mathbb{R}^d}  \widehat{c}(\eta-\xi,\xi) \widehat{v}(\xi)\rmd\xi$ is supported in $\vert \eta -\xi \vert\leqslant K\vert \xi\vert$ by assumption on the support of $\widehat{c}$. Hence for fixed $\eta$, the integral restricts to the corona $ (1+K)^{-1}\vert\eta\vert\leqslant \vert \xi\vert\leqslant (1-K)^{-1}\vert\eta\vert  $ by the triangular inequality. Therefore the above identity for $\psi(2^{-i}D)Op(c)v$ rewrites
\begin{equation*} \begin{split}
\psi(2^{-i}D)Op(c)v &= \mathcal{F}^{-1}_\eta\left(\psi(2^{-i}\eta)\int_{(1+K)^{-1}\vert\eta\vert\leqslant \vert \xi\vert\leqslant (1-K)^{-1}\vert\eta\vert } \widehat{c}(\eta-\xi,\xi) \widehat{v}(\xi)\rmd\xi \right)   \\
&= \mathcal{F}^{-1}_\eta\left(\psi(2^{-i}\eta)\int_{\mathbb{R}^d}  \widehat{c}(\eta-\xi,\xi)  \widetilde{\chi}^2(2^{-i}\xi) \widehat{v}(\xi)\rmd\xi \right),
\end{split} \end{equation*}
where $\widetilde{\chi}(2^{-i}.)$ is an extra cut-off function which localizes the integral of $\xi$ on some larger corona of radius $\vert \xi\vert\sim 2^i$, $\widetilde{\chi}=1$ on the region $\big\{(1+K)^{-1}\vert\eta\vert\leqslant \vert \xi\vert\leqslant (1-K)^{-1}\vert\eta\vert\big\}$. For the moment we get a bound of the form 
\begin{equation*} \begin{split}
\big\Vert\psi(2^{-i}D)Op(c)v \big\Vert_{L^p(\mathbb{R}^d)} &= \Big\Vert \left(\widehat{\psi(2^{-i}.)}\star Op(c)v\right) \Big\Vert_{L^p(\mathbb{R}^d)}   \\
&= \Big\Vert \left(\left(2^{id} \widehat{\psi}(2^i.)\right)\star Op(c)v\right)(y) \Big\Vert_{L^p(\mathbb{R}^d)}   \\
&\leqslant\Vert \widehat{\psi}\Vert_{L^1(\mathbb{R}^d)}  \Big\Vert \int_{\mathbb{R}^d} c(x,\xi) e^{i\xi\cdot x} \widetilde{\chi}^2(2^{-i}\xi) \widehat{v}(\xi)\rmd\xi \Big\Vert_{L^p_x(\mathbb{R}^d)}   \\
&\leqslant \Vert \widehat{\psi}\Vert_{L^1(\mathbb{R}^d)}  \Big\Vert  c(x;D)\widetilde{\chi}(2^{-i}\vert D\vert )\left( \widetilde{\chi}(2^{-i}\vert D\vert ) v \right)  \Big\Vert_{L^p(\mathbb{R}^d)}.
\end{split} \end{equation*}
Set the sequence of Schwartz kernels
\begin{eqnarray*}
A_i(x,x-y) \defeq \int_{\mathbb{R}^d} c(x,\xi) e^{i\xi\cdot(x-y)} \widetilde{\chi}(2^{-i}\xi) \rmd\xi;
\end{eqnarray*}
this is the Schwartz kernel of the operators $c(x;D)\widetilde{\chi}(2^{-i}\vert D\vert )$ whose symbol is localized on the frequency shell $\vert \xi\vert\simeq 2^i $ and also define a sequence of functions
\begin{eqnarray*}
B_i \defeq \widetilde{\chi}(2^{-i}\vert D\vert ) v=\mathcal{F}^{-1}_\xi \left(\widetilde{\chi}(2^{-i}\cdot) \widehat{v}\right) 
\end{eqnarray*}
which also corresponds to $v$ localized to frequency shell $\vert \xi\vert\simeq 2^i$. First, we deal with the easier term $B_i$, $v\in B^\alpha_{p,q}$ means that
$$
\sum_{i=1}^\infty \big(2^{i\alpha} \big\Vert \widetilde{\chi}(2^{-i}\vert D_x\vert)v \big\Vert_{L^p(\mathbb{R}^d)}\big)^q=\sum_{i=1}^\infty \left(2^{i\alpha}\Vert B_i \Vert_{L^p(\mathbb{R}^d)}\right)^q <+\infty. 
$$
Next, we deal with the more subtle operator term $A_i$. The idea is to treat the operator $c(x;D)\widetilde{\chi}(2^{-i}\vert D\vert )$ as some semiclassical pseudo-differential operator and use the continuity properties of semiclassical pseudo-differentials acting on $L^p$ spaces.
First, we make a change of variables changing the position of the dyadic factor, the aim is to localize the frequency on the shell $\vert\xi\vert\simeq 2$
\begin{align*}
A_i(x,x-y)&=\int_{\mathbb{R}^d} c(x,\xi) e^{i\xi\cdot(x-y)} \widetilde{\chi}(2^{-i}\xi) \rmd\xi=2^{id}\int_{\mathbb{R}^d} c(x,2^i\xi) e^{i\xi\cdot 2^i(x-y)} \widetilde{\chi}(\xi) \rmd\xi\\
&= 2^{im} 2^{id}\int_{\mathbb{R}^d} \left(2^{-im}c(x,2^i\xi)\widetilde{\chi}(\xi)\right) e^{i\xi\cdot 2^i(x-y)}  \rmd\xi.
\end{align*}
We make two crucial observations. First $c\in S^m_{1,1}(\bbR^d)$ therefore the sequence of cut-off symbols $\left(2^{-im}c(x,2^i\xi)\widetilde{\chi}(\xi)\right)$ forms a bounded family of smooth functions in $\xi$ with compact support on some frequency shell $\{a\leqslant \vert\xi\vert\leqslant b\}, 0<a<b$ uniformly in $i$ and $x$. Second, for the family of rescaled kernels
\begin{eqnarray*}
K_i(x,h) = \int_{\mathbb{R}^d} \left(2^{-im}c(x,2^i\xi)\widetilde{\chi}(\xi)\right) e^{i\xi\cdot h}  \rmd\xi,
\end{eqnarray*}
the above observation implies that
\begin{eqnarray*}
\vert K_i(x,h)\vert=  \int_{\mathbb{R}^d} \left((1-\sum_{i=1}^d\partial_{\xi_i}^2)^{\frac{[d+2]}{2}} \left(2^{-im}c(x,2^i\xi)\widetilde{\chi}(\xi)\right)\right) (1+\vert h\vert^2)^{\frac{[d+2]}{2}} e^{i\xi\cdot h}  \rmd\xi \leqslant C(1+\vert h\vert)^{-[d+2]}
\end{eqnarray*}
where the constant $C$ does not depend on $x$. We used the fact that
\begin{equation*} \begin{split}
\big\vert\partial_\xi\left(2^{-im}c(x,2^i\xi)\widetilde{\chi}(\xi)\right) \big\vert &\lesssim  2^{-im}2^i (1+2^i\vert\xi\vert)^{m-1}\widetilde{\chi}(\xi) + 2^{-im} \big\vert (1+2^i\vert\xi\vert)^{m}\partial_\xi\widetilde{\chi}(\xi)\big\vert  \\
&\lesssim  2^{i(1-m)} (1+2^i)^{m-1}+ 2^{-im}2^{im}  \lesssim 1,
\end{split} \end{equation*}
since the support of $\widetilde{\chi}$ in the shell $\{a\leqslant \vert\xi\vert\leqslant b\}$ forces some decay of $c(x,2^i\xi)\widetilde{\chi}(\xi)$ and its derivative in $\xi$. The above decay bound on the kernels $K_i$ implies that
\begin{eqnarray*}
\int_{\mathbb{R}^d}\sup_x \vert K_i(x,h)\vert \rmd h\lesssim C \int_{\mathbb{R}^d}(1+\vert h\vert)^{-[d+2]} \rmd h<+\infty
\end{eqnarray*}
The kernel $K_i$ is related to $A_i$ via the exact scaling relation
\begin{eqnarray*}
A_i(x,x-y) = 2^{im} 2^{id} K_i\big(x,2^i(x-y)\big),
\end{eqnarray*}
therefore by scaling invariance of the $L^1$ norm together with Young inequality, we get
\begin{equation*} \begin{split}
\big\Vert  c(x;D)\widetilde{\chi}(2^{-i}\vert D\vert )\left( \widetilde{\chi}(2^{-i}\vert D\vert ) v \right) \big\Vert_{L^p(\mathbb{R}^d)} &= \Big\Vert \int_{\mathbb{R}^d} A_i(x,x-y)B_i(y) \rmd y \Big\Vert_{L^p_x(\mathbb{R}^d)}   \\
&= 2^{im} \Big\Vert \int_{\mathbb{R}^d} 2^{id}K_i(x,2^i(x-y))B_i(y) \rmd y \Big\Vert_{L^p_x(\mathbb{R}^d)}   \\
&\leqslant 2^{im} \Big\Vert \int_{\mathbb{R}^d}\sup_{z\in \mathbb{R}^d} \big\vert 2^{id}K_i(z,2^i(x-y)) \big\vert| B_i(y)| \rmd y \Big\Vert_{L^p_x(\mathbb{R}^d)}  \\
&\leqslant 2^{im} \big\Vert\sup_{z\in \mathbb{R}^d} K_i(z,\cdot)\big\Vert_{L^1(\mathbb{R}^d)} \Vert B_i\Vert_{L^p(\mathbb{R}^d)}.
\end{split} \end{equation*}
In the end we get a bound of the form
$$
\big\Vert\psi(2^{-i}D)Op(c)v\big\Vert_{L^p(\mathbb{R}^d)} \leqslant \Vert \widehat{\psi}\Vert_{L^1(\mathbb{R}^d)} 2^{im} \big\Vert\sup_{X\in \mathbb{R}^d} K_i(X,\cdot)\big\Vert_{L^1(\mathbb{R}^d)} \Vert B_i\Vert_{L^p(\mathbb{R}^d)}\lesssim 2^{im} \Vert B_i\Vert_{L^p(\mathbb{R}^d)}
$$
which concludes since 
\begin{eqnarray*}
\sum_i \left(2^{i(\alpha-m)}\Vert\psi(2^{-i}D)Op(c)v\Vert_{L^p(\mathbb{R}^d)}\right)^q\lesssim \sum_{i=1}^\infty  \left(2^{i\alpha}\Vert B_i\Vert_{L^p(\mathbb{R}^d)}\right)^q\lesssim \Vert v\Vert_{B^{\alpha}_{p,q}}^q,
\end{eqnarray*}
and we are done.
\end{proof}

\smallskip

We learned this idea of using Young inequality for such proof from Bonthonneau and also Fermanian-Kammerer~\cite[Prop 3.2.1 p.~21]{FermKamlectures} -- we warmly thank them here. One can deduce from the above result that  $\Psi^m_{1,0}(M)$ sends $B^\alpha_{p,q}(M)$ to $B^{\alpha-m}_{p,q}(M)$ continuously in the manifold setting using charts and partitions of unity. Following~\cite{Horm3}, we represent pseudo-differential operators $A\in \Psi^m_{1,0}(M)$ over $M$ as
\begin{eqnarray*}
A=\sum_{i\in I}  \chi_i \kappa_i^{*}  A_i \kappa_{i*} \widetilde{\chi}_i+R\,,
\end{eqnarray*}
where $\chi_i$ is a partition of unity subordinated to
the cover $\cup_i U_i$, $\widetilde{\chi}_i\in C^\infty_c(U_i)$, $\widetilde{\chi}_i=1$ on the support of $\chi_i$,  $A_i\in \Psi^m_{1,0}(\mathbb{R}^d)$ and $R\in \Psi^{-\infty}(M)$ is a smoothing operator. 
Therefore if we are given some Besov distribution $u\in B^\alpha_{p,q}(M)$, then 
$$
Au=\sum_{i\in I}  \chi_i \kappa_i^{*}  A_i \kappa_{i*} \left(\widetilde{\chi}_iu\right) 
$$
and using the invariance of $B^\alpha_{p,q}(M)$ under diffeomorphisms and stability by multiplication by some smooth function one sees that $\kappa_{i*} \left(\widetilde{\chi}_iu\right)\in B^\alpha_{p,q}(\mathbb{R}^d)$, hence $A_i\kappa_{i*} \left(\widetilde{\chi}_iu\right)\in B^{\alpha-m}_{p,q}(\mathbb{R}^d) $, from Proposition \ref{l:simplifiedprooftwistedclass} and the diffeomorphism invariance and the stability by multiplication with smooth functions, we deduce that
$Au\in B^{\alpha-m}_{p,q}(M)$. 

\medskip

\subsection{Paraproduct decomposition on  manifolds$\boldmath{.}$ \hspace{0.1cm}}
\label{ss:definitionLPSprojectorsparaproduct}

The goal of the present paragraph is to present a simple method to decompose multilinear products of smooth functions as a sum of multilinear operations involving interactions of different frequencies. The ideas go back to Bony~\cite{Bony}, Coifman-Meyer~\cite{Meyer2}.
We start pedagogically by the simple case of a bilinear product $uv$ of smooth functions so that the reader can clearly see the mechanisms at work.

\subsubsection{Paraproduct and resonant operators on closed manifolds$\boldmath{.}$ \hspace{0.1cm}}
\label{sss:decompositionproduct}

Let $M$ be a compact manifold. Denote by $(U_i,\kappa_i)_i$  an open cover by charts of the manifold $M$ where $\kappa_i:U_i\subset M\mapsto \kappa_i(U_i)\subset \mathbb{R}^d$ and $\kappa_i$ is a smooth diffeomorphism. From the data of a smooth compact manifold and its open cover by charts, we may decompose as follows the product $uv$ of two smooth functions. We start from a partition of unity $\sum_i \chi_i=1$ with $\chi_i\in C^\infty_c(U_i)$, subordinated to $(U_i)_i$, and another family of smooth functions $ \widetilde \chi_i  \in C^\infty_c(U_i)$ with $\widetilde \chi_i=1$ on the support of $\chi_i$. We choose for every $i$ some function $\psi_i\in C_c^\infty(\kappa_i(U_i))$ such that $\psi_i|_{\text{supp}(\chi_i\circ \kappa_i^{-1})}=1$. {\bf From now on, we write  $\chi\ll \tilde{\chi}$ if $\tilde{\chi}=1$ on the support of $\chi$}. Recall the definitions \eqref{eq_flat_paraproduct} and \eqref{eq_flat_resonnat} of the paraproduct and resonant term on $\mathbb{R}^d$. Then we have the identities
\begin{eqnarray*} 
uv&=&\sum_{i\in I} (u\chi_i) (v\widetilde \chi_i)
=\sum_{i\in I} \kappa_i^*(\kappa_i)_*(u\chi_i)   \kappa_i^*(\kappa_i)_*(v\widetilde \chi_i)\\
&=&\sum_{i\in I} \kappa_i^*\big((\kappa_i)_*(u\chi_i)  (\kappa_i)_*(v\widetilde\chi_i)\big) = \sum_{i\in I} \kappa_i^*\big( \psi_i \times (\kappa_i)_*(u\chi_i) \times (\kappa_i)_*(v\widetilde\chi_i)\big)  \\
&=&\sum_{i\in I} \kappa_i^*\Big( \psi_i\big( (\kappa_i)_*(u\chi_i) \odot (\kappa_i)_*(v\widetilde\chi_i)\big)\Big)
+  
\sum_{i\in I} \kappa_i^*\Big( \psi_i\big( (\kappa_i)_*(u\chi_i) \prec (\kappa_i)_*(v\widetilde\chi_i)\big)\Big)   \\
&&+  
\sum_{i\in I} \kappa_i^*\Big( \psi_i\big( (\kappa_i)_*(u\chi_i) \succ (\kappa_i)_*(v\psi_i)\big)\Big)   \\
&\eqdef& u\odot v+ u\prec v+ u\succ v\,.
\end{eqnarray*}
This decomposition motivates the following more general definition of paraproduct and resonant operators which depend on the data of a smooth compact manifold, its open cover by charts and a collection of cut-off functions satisfying suitable compatibility conditions.

\smallskip

\begin{defi*} [Generalized paraproducts and resonant products]
We choose a family $\chi_i\in C^\infty_c(U_i)$  and another family of smooth functions $ \widetilde \chi_i  \in C^\infty_c(U_i)$ with   $\widetilde \chi_i=1$ on ${\rm supp} (\chi_i)$. Then for every $i$, we also choose some function $\psi_i\in C_c^\infty(\kappa_i(U_i))$ such that $\psi_i|_{\text{supp}(\chi_i\circ \kappa_i^{-1})}=1$. We define some generalized paraproduct and resonant operators setting
\begin{eqnarray*}
u \odot v &=&\sum_i \kappa_i^*\Big(\psi_i  \big(  \kappa_{i*}\left( \chi_{i}u \right)    \odot\kappa_{i*}\left( \widetilde{\chi}_{i}v \right)\big) \Big),\\
u \prec v &=&\sum_i \kappa_i^*\Big(\psi_i  \big(  \kappa_{i*}\left( \chi_{i}u \right)    \prec\kappa_{i*}\left( \widetilde{\chi}_{i}v \right)\big) \Big),\\
u\succ v &=&\sum_i \kappa_i^*\Big(\psi_i  \big(  \kappa_{i*}\left( \chi_{i}u \right)    \succ\kappa_{i*}\left( \widetilde{\chi}_{i}v \right)\big) \Big),
\end{eqnarray*}
where the operators $\prec,\succ,\odot$ on the right hand side are defined on $\mathbb{R}^d$. 
\end{defi*}

\smallskip

We do not impose that $\sum \chi_i=1$ so we do not necessarily have a decomposition of the product as $uv=u\prec v+u\succ v+u\odot v$. However when $ (\chi_i)_i$ is a partition of unity subordinated to 
$(U_i)_i$, i.e. $\sum_i\chi_i=1$, then the above definition yields a decomposition of the usual product of smooth functions as 
$$
uv=u \odot  v+ u \prec v+u\succ v.
$$
{\bf We will use the notation with numbers, e.g. $\prec_1, \prec_2, \odot_2,\succ_3 \ldots$,  to distinguish  paraproduct/resonant operators built from different cut-off functions.} 

\smallskip
 
Note the following subtle fact: our definition of $\prec,\succ,\odot$ is asymmetric in the choice of the cut-off functions, the collection $\widetilde{\chi}_i,i\in I$ does not form a partition of unity, so 
$$
u\prec v \neq v\succ u
$$ 
since we use different cut-off functions on the right or on the left of the paraproducts. The paraproduct operators $\prec,\succ$ and the resonant operator $\odot$ are \textbf{non-commutative} although their sum $\odot\;+\prec+\succ$ is the Young product of distributions, which is commutative. These products are not associative either. However, we can still justify that the two maps $\prec$ and $\succ$ have the expected analytical properties when acting on Besov spaces. The following estimates follow from the same estimates on $\mathbb{R}^d$ and from the diffeomorphism invariance of the H\"older-Besov spaces \cite{BCD}. One has
\begin{eqnarray*}
\Vert g\succ f \Vert_{\alpha+\beta}+\Vert f\prec g \Vert_{\alpha+\beta} \lesssim \Vert  f\Vert_{\alpha} \Vert g\Vert_{\beta},\qquad (\alpha< 0),
\end{eqnarray*}
and
\begin{eqnarray*}
\Vert f\odot  g \Vert_{\alpha+\beta} \lesssim \Vert  f\Vert_{\alpha} \Vert g\Vert_{\beta}, \qquad (\alpha+\beta>0).
\end{eqnarray*}
The proof of these estimates boils down to comparing $ u\chi_i \prec v\widetilde{\chi}_i $ and $ u\widetilde{\chi}_i \prec v\chi_i $ where we exchanged the cut-off functions and where $\prec$ is the paraproduct operator on $\mathbb{R}^d$. A first observation is that if $u,v$ belong to some Besov spaces of given regularity, their product with any smooth functions with compact support will belong to the Besov space of the exact same regularity. Therefore the position of the test functions do not affect the continuity properties of our paraproducts acting on Besov spaces. The same argument also applies to the resonant part.  

\smallskip

\begin{lemm} \label{lem:smt_comm}  
Let $M$ be a closed manifold. Let $\eta, \eta_1, \eta_2 \in \mathcal{C}^\infty(M),  f\in\mathcal{C}^\alpha(M), g \in \mathcal{C}^\beta(M)$ with $\alpha>0$ and $ \beta<0$. One has
\begin{equation*} \begin{split}
&\big\| \eta(f\prec g) -f\prec(\eta g)\big\|_{\alpha+\beta} \lesssim C(\eta) \|f\|_{\alpha}\|g\|_\beta,   \\
&\big\|\eta(f\prec g) -(\eta f)\prec g \big\|_{\alpha+\beta} \lesssim C(\eta) \|f\|_{\alpha}\|g\|_{{\beta}}, 
\end{split} \end{equation*}
and
$$
\big\|f\prec(\eta g) -(\eta f)\prec g \big\|_{\alpha+\beta} \lesssim C(\eta) \|f\|_{\alpha}\|g\|_{{\beta}}. 
$$
and
$$
\big\| \eta_1\eta_2(f\prec g)- (\eta_1 f)\prec(\eta_2 g) \big\|_{\alpha+\beta} \lesssim C(\eta_1, \eta_2)\|f\|_\alpha\|\|g\|_\beta.
$$
The same estimates also hold on $\mathbb{R}^d$ assuming the condition of compact support for all functions. The same estimates also hold  when we replace the paraproduct operator $ \prec $ by the resonant operator $\odot$ provided $\alpha+\beta>0$. 
\end{lemm} 

\smallskip

\begin{proof} 
It suffices to prove the first two estimates. 
We choose a family $\chi_i\in C^\infty_c(U_i)$  and another family $ \widetilde \chi_i  \in C^\infty_c(U_i)$ with   $\widetilde \chi_i=1$ on ${\rm supp} \chi_i$. For every $i$, we also choose some function $\psi_i\in C_c^\infty(\kappa_i(U_i))$ such that $\psi_i|_{\text{supp}(\chi_i\circ \kappa_i^{-1})}=1$, and  $ \widetilde{\psi}_i \in C_c^\infty(\kappa_i(U_i))$ such that   $ \widetilde \psi_{i}=1$  on $ {\rm supp}\psi_{i} \cup {\rm supp}(\widetilde{\chi}_i\circ \kappa_i^{-1}) $. We have
\begin{eqnarray*}
\eta (f\prec g)&=&\sum_{j} \eta \Big\{ \psi_{j} \big[ (f\chi_{j})\circ\kappa_{j}^{-1} \prec  (g\widetilde\chi_{j})\circ \kappa_j^{-1}\big] \Big\} \circ \kappa_j   \\
&=&\sum_{j} \eta \widetilde{\psi}_j\circ \kappa_j \Big\{ \psi_{j} \big[(f\chi_{j})\circ\kappa_{j}^{-1} \prec  (g\widetilde\chi_{j})\circ \kappa_j^{-1}\big] \Big\} \circ \kappa_j   \\
&=&\sum_{j}\left\lbrace \psi_{j} ( \eta \widetilde{\psi}_j\circ \kappa_j )\circ \kappa_j^{-1} \Big[(f\chi_{j})\circ\kappa_{j}^{-1} \prec  (g\widetilde\chi_{j})\circ \kappa_j^{-1}\Big]\right\rbrace \circ \kappa_j.
\end{eqnarray*}
Set, for $u\in \mcC^\alpha_{c}$, 
$$
P_u(v) \defeq u\prec v.
$$
It follows from the $\bbR^d$ version of Propositions~\ref{l:simplifiedprooftwistedclass} and \ref{lemm:composition} that $P_u$ has the following property. For any $\rho\in C^\infty_c(\mathbb{R}^n)$ the operator 
$$
\rho P_u - P_u\rho: \mcC^{\beta}\rightarrow \mcC^{\alpha+\beta} 
$$ 
has a norm dominated by $C(\rho)\|f\|_\alpha$. Therefore 
\begin{equation*} \begin{split}
\Big\| ( \eta \widetilde{\psi}_j\circ \kappa_j )\circ \kappa_j^{-1} \Big[ (f\chi_{j})\circ\kappa_{j}^{-1} \prec  (g\widetilde\chi_{j})\circ \kappa_j^{-1}\Big]  &-  (f\chi_{j})\circ\kappa_{j}^{-1} \prec  (g \eta \widetilde{\psi}_j\circ \kappa_j \widetilde\chi_{j})\circ \kappa_j^{-1} \Big\|_{\alpha+\beta}   \\
&\lesssim  C(\eta)  \|f\|_\alpha\|g\|_{\beta}. 
\end{split} \end{equation*}
Since $\widetilde{\psi}_j\circ\kappa_j=1 $ on ${\rm supp}(\widetilde \chi_j)$, the second term in the previous estimate is  the one  in  the definition of $f\prec (\eta g)$, therefore we get the first estimate
\begin{eqnarray*}
\big\|\eta (f\prec g) - f\prec(\eta g) \big\|_{\alpha+\beta} \lesssim C(\eta) \|f\|_{\alpha}\|g\|_{\beta}. 
\end{eqnarray*}
For the second estimate, we remark that we have the flat decomposition
\begin{equation*} \begin{split}
( \eta \widetilde{\psi}_j\circ \kappa_j )\circ \kappa_j^{-1}  \big[ (f\chi_{j})\circ\kappa_{j}^{-1} &\prec  (g\widetilde\chi_{j})\circ \kappa_j^{-1} \big]   \\
& \quad \quad = ( \eta \widetilde{\psi}_j\circ \kappa_j )\circ \kappa_j^{-1} \prec \big[ (f\chi_{j})\circ\kappa_{j}^{-1} \prec  (g\widetilde\chi_{j})\circ \kappa_j^{-1} \big]   \\
& \quad \quad \quad + ( \eta \widetilde{\psi}_j\circ \kappa_j )\circ \kappa_j^{-1} \succ  \big[ (f\chi_{j})\circ\kappa_{j}^{-1} \prec  (g\widetilde\chi_{j})\circ \kappa_j^{-1}\big]   \\
& \quad \quad \quad + ( \eta \widetilde{\psi}_j\circ \kappa_j )\circ \kappa_j^{-1} \odot  \big[ (f\chi_{j})\circ\kappa_{j}^{-1} \prec  (g\widetilde\chi_{j})\circ \kappa_j^{-1}\big]. 
\end{split} \end{equation*}
The two last terms in the right hand side are  in $\mcC^{\alpha+\beta}$   since $ \eta \in C^\infty$,  with their norm dominated by $ C \|\eta\|_{(\alpha-\beta)\vee (\varepsilon-\beta)} \|f\|_\alpha\| g \|_\beta$ for any $\varepsilon>0$. For the first term in the right hand side we use the $\bbR^d$ version of the paramultiplication estimate of Lemma \ref{lemm:comm_paramultiplication_flat} to get
\begin{eqnarray} \label{eq_lem_comm} \begin{split}
( \eta \widetilde{\psi}_j\circ \kappa_j )\circ \kappa_j^{-1} \prec &\; \big[ (f\chi_{j})\circ\kappa_{j}^{-1} \prec  (g\widetilde\chi_{j})\circ \kappa_j^{-1}\big]   \\
&- \big[ ( \eta \widetilde{\psi}_j\circ \kappa_j )\circ \kappa_j^{-1}   (f\chi_{j})\circ\kappa_{j}^{-1} ) \big] \prec  (g\widetilde\chi_{j})\circ \kappa_j^{-1} 
\end{split}  \end{eqnarray}
is in $\mcC^{\alpha+\beta}$ with its norm dominated by $ C \|\eta\|_{\alpha}\| f\|_\alpha \|g\|_\beta$.  As above, the second term in \eqref{eq_lem_comm} is  the  one in the definition of $(\eta f)\prec  g$, hence we get the second estimate as required.

When $\alpha+\beta>0$, the product $\eta fg$  is well-defined, hence we can decompose the product $\eta fg$ in three ways of paraproduct with respect to $\eta(fg), (\eta f)g, f(\eta g)$.   Since $\beta<0$, we have $ f\succ g, f\succ (\eta g), (\eta f)\succ g \in \mcC^{\alpha+\beta}$. Combining with the estimates above for paraproducts,  we imply the same estimates for $\odot$.
\end{proof} 

\smallskip

Finally, we have the following proposition which summarizes the results proved so far:  

\smallskip

\begin{prop} \label{p:youngproduct} 
 Let $\alpha,\beta$ be real numbers such that 
 $$
 \alpha<0, \quad \alpha+\beta > 0.
 $$
Then for all $(u,v)\in \mathcal{C}^\alpha(M)\times \mathcal{C}^\beta(M)$, the Young product $uv$ is well-defined in $\mathcal{D}^\prime(M)$ and can be decomposed as
$$
uv = u\odot v+u\prec v+u\succ v
$$
where the maps
\begin{align*}
   (u,v)&\in \mathcal{C}^\alpha(M)\times \mathcal{C}^\beta(M)\mapsto u\odot v+u\prec v \in \mathcal{C}^{\alpha+\beta}(M)\\
(u,v)&\in \mathcal{C}^\alpha(M)\times \mathcal{C}^\beta(M)\mapsto u\succ v \in \mathcal{C}^{\alpha}(M) 
\end{align*}
are bilinear continuous.
\end{prop}

\smallskip

\subsubsection{Generalized Littlewood-Paley-Stein projectors, paraproduct and resonant operators$\boldmath{.}$ \hspace{0.1cm}}
\label{sss:definitionLPSproj}

\begin{defi} [Generalized Littlewood-Paley-Stein projectors] \label{def:generalizedLPSproj}
From the above data, we can write explicitely a new formula for the manifold analog of the Littlewood-Paley-Stein projector as follows:
\begin{eqnarray*} \label{eq:LPSprojdefi}
P_{k}^i(\cdot)= \kappa_i^*\Big[\psi_i\Delta_k\big(\kappa_{i*}(\chi_i\cdot)\big)\Big], \qquad (k\geqslant 0).
\end{eqnarray*}
Note that our projectors are indexed by both a frequency $k$ and chart index $i\in I$. We introduce another auxiliary projector $\widetilde{P}^i_k$ in terms of the test function $\widetilde{\chi}_i\in C^\infty_c(U_i)$ which reads
\begin{eqnarray*} \label{eq:LPSprojdefi2}
\widetilde{P}_\ell^i(\cdot)= \kappa_i^*\Big[\widetilde{\psi}_i\Delta_\ell\big(\kappa_{i*}(\widetilde{\chi}_i\cdot)\big)\Big], \qquad (\ell\geqslant 0).
\end{eqnarray*}
where $\widetilde{\psi}_i\in C^\infty_c(\kappa_i(U_i))$ is any function which equals $1$ on the support of $\psi_i$ and $\widetilde{\chi}_i=1$ on the support of $\chi_i$. 
\end{defi}

\smallskip

If $\sum_{i\in I}\chi_i=1$ then we have the resolution of the identity $Id=\sum_{i\in I}\sum_{k=0}^\infty P^i_k$. Beware that the projectors $\widetilde{P}_k^i$ do not satisfy the resolution of the identity because we do not assume that the $(\widetilde{\chi}_i)_{i\in I}$ form a partition of unity. Their introduction is necessary since our paraproduct and resonant operators are asymmetrical. Then we can use the pair of Littlewood-Paley-Stein projectors $P^i_k,\widetilde{P}^i_\ell$ to write the definition of the resonant term as
\begin{eqnarray*} \label{eq:resonant}
u\odot v=\sum_{i\in I} \sum_{\vert k - \ell \vert\leqslant 1} P^i_k(u)\widetilde{P}^i_\ell(v)
\end{eqnarray*}
and the paraproduct term as
\begin{eqnarray*} \label{eq:paraprod}
u\prec v=\sum_{i\in I} \sum_{k\leqslant \ell - 2} P^i_k(u)\widetilde{P}^i_\ell(v).
\end{eqnarray*}
$u\succ v$ is defined accordingly, with the condition $l\leqslant k-2$. In the sequel, we shall also need a notion of localized paraproduct and resonant operators.

\smallskip

\begin{defi} [Localized paraproducts and resonant products] \label{def:localpararesonant}
 They are indexed by the chart index $i$ and take value in distributions supported in $U_i$. They are defined following the simple rule, for every chart index $i$
\begin{eqnarray*}\label{eq:resonantlocal}
u\odot_i v= \sum_{\vert k-\ell \vert\leqslant 1} P^i_k(u)\widetilde{P}^i_\ell(v) = \kappa_i^*\Big[ \psi_i\Big( \kappa_{i*} \left(\chi_iu\right)  \odot \kappa_{i*} \left(\widetilde{\chi}_iv\right)\Big)\Big],
\end{eqnarray*}
and 
\begin{eqnarray*} \label{eq:paraprodlocal}
u\prec_i v= \sum_{k\leqslant \ell - 2} P^i_k(u)\widetilde{P}^i_\ell(v) = \kappa_i^*\Big[\psi_i\Big( \kappa_{i*} \left(\chi_iu\right)  \prec \kappa_{i*} \left(\widetilde{\chi}_iv\right)\Big)\Big].
\end{eqnarray*}
We define $u\succ_i v$ accordingly, with the condition $\ell\leqslant k-2$.
\end{defi}

\smallskip

This localization is useful since we only need these localized resonant products for our stochastic estimates from the companion work~\cite[section 4]{BDFT}.

\subsection{Triple product and  a commutator involving paraproducts$\boldmath{.}$ \hspace{0.1cm}}
\label{SubsectionTripleParamultiplication}
 
In this section, we shall decompose a triple product of the form $F(f)gh$
where $(f,g,h)\in C^\infty(M)^3$ and $F\in C^\infty(\mathbb{R})$ as a sum of quadrilinear operations involving both the paramultiplications, paralinearizations to deal with the composite term $F(f)$ and the usual product on manifolds.

\subsubsection{Decomposition of triple products in terms of generalized paraproducts and resonant products$\boldmath{.}$ \hspace{0.1cm}}

Let us localize on manifolds a triple product following the philosophy that we used for usual double products
\begin{eqnarray*}
uvw=\sum_i \left(\chi_{i1}u \right)\left( \chi_{i2}v \right)\left( \chi_{i3}w \right)  
\end{eqnarray*} 
where only $\sum_i\chi_{i1}=1$ is a partition of unity subordinated to the cover $( U_i)_{i\in I}$, and  $ \chi_{i2}=1$ (resp. $ \chi_{i3}=1 $) on the support of $\chi_{i1}$ (resp. of $\chi_{i2}$), the three functions $\chi_{i1},\chi_{i2},\chi_{i3}$ belong to $C^\infty_c(U_i)$. Now we pull-back on open domains of $\mathbb{R}^d$ using some charts
\begin{eqnarray*}
uvw&=&\sum_i \kappa_i^*\Big[  \kappa_{i*} \left(\chi_{i1}u \right)   \kappa_{i*}\left( \chi_{i2}v \right)    \kappa_{i*}\left( \chi_{i3}w \right)  \Big]\\\
&=&\sum_i \kappa_i^*\Big[ \psi_i \kappa_{i*} \left(\chi_{i1}u \right)   \kappa_{i*}\left( \chi_{i2}v \right)    \kappa_{i*}\left( \chi_{i3}w \right)  \Big]
\end{eqnarray*}  
where everything in parenthesis is happening in $\mathbb{R}^d$. Then we decompose the products inside the parenthesis using the flat paraproduct and resonant operators. So, if we choose $\psi_i=1$ on $\kappa_i(\text{supp}(\chi_{i3}))$, we get the following decomposition of the triple product: 
\begin{eqnarray*}
uvw
&=&\sum_i \kappa_i^*\Big[ \psi_i  \Big(\kappa_{i*} \left(\chi_{i1}u \right)\Big) (\prec+ \succ+\odot)   \Big(  \kappa_{i*}\left( \chi_{i2}v \right)     (\prec+ \succ+\odot) \kappa_{i*}\left( \chi_{i3}w \right)\Big)  \Big]
\\
&=&u\prec_1 \left( v\prec_2w \right)+ u\prec_1 \left( v\succ_2w \right)+ u\prec_1 \left( v\odot_2w \right) +\cdots
\end{eqnarray*} 
where the $\cdots$ means that we replaced $\prec_1$ with either $\succ_1,\odot_1$ and sum over all possibilities, we get nine terms in total. Thus we decomposed a triple product into a sum of $9$ trilinear operations involving paraproduct and resonant operators. The trilinear operation we are interested in reads
\begin{equation*} \label{eq:firsttrilinearidentity}
u\prec_1 \left( v\prec_2w \right)\defeq
\sum_i \kappa_i^*\Big[ \psi_i \Big( \kappa_{i*} \left(\chi_{i1}u \right) \Big)\prec   \Big(  \kappa_{i*}\left( \chi_{i2}v \right)    \prec\kappa_{i*}\left( \chi_{i3}w \right)\Big)  \Big].
\end{equation*}
The important fact is that the generalized paraproduct and resonant operators have the same continuity properties as in the flat case and can be decomposed in terms of generalized Littlewood-Paley-Stein projectors $P^i_k$ of general form in such a way that one can repeat for them word by word the stochastic estimates of the paper \cite{BDFT}.  

 For later use in the proof of Theorem \ref{thm_A_2_mfd}, we introduce now a trilinear operation defined by
$$
f\cdot_1 (g\odot_2 h) \defeq \sum_i \kappa_i^*\Big[ \psi_i\kappa_{i*} \left(\chi_{i1} f \right) \Big(   \kappa_{i*}\left( \chi_{i2} g \right) \odot\kappa_{i*}\left( \chi_{i3} h \right) \Big) \Big].
$$
Our use of a numbering subscript even for the multiplication of functions is deliberate, since we want to keep track of the partitions of unity and cut-off functions we are using.

\subsubsection{ Paralinearization of some trilinear operator$\boldmath{.}$ \hspace{0.1cm}}

Now for $F\in C^\infty(\mathbb{R})$,  $f\in \mcC^{\alpha}(M),g\in \mcC^\beta(M),h\in \mcC^{\gamma}(M)$  with $\alpha+ \beta+\gamma>0$, $\beta+\gamma<0$ and  $2\alpha+\gamma>0$,  we define the trilinear operation:
\begin{eqnarray*} \label{eq:trilinearsecond}
  F(f)\odot_1 (g \prec_2 h)= \sum_i    \kappa_i^*\Big[ \psi_i \Big( \kappa_{i*} \big(\chi_{i1} F(f) \big) \Big)\odot   \Big(  \kappa_{i*}\left( \chi_{i2}g  \right)    \prec\kappa_{i*}\left( \chi_{i3} h \right) \Big)\Big],
\end{eqnarray*} 
 where $\chi_{i1}, \chi_{i2}, \chi_{i3}\in C^{\infty}_c(U_i)$ and $\psi\in C^{\infty}_{c}(\kappa_i(U_i))$ with $\chi_{i1}\ll \chi_{i2}\ll\chi_{i3}\ll \kappa_i^*\psi_i$.
\smallskip

In the following lemma and in the sequel, for every distributions $f,g$, whenever $f-g\in \mcC^\alpha$ for some $\alpha\in\mathbb R$, we write
\begin{align*}
    f=g+\mcC^{\alpha}\,.
\end{align*}

\begin{lemm} \label{lemThmCorrector}
Let us consider $F\in C^\infty(\mathbb{R})$,  $f\in \mcC^\alpha(M)$, $g\in \mcC^\beta(M)$ and $h\in \mcC^\gamma(M)$ for $\alpha+\beta+\gamma>0$, $ \beta+\gamma<0 $ and $2\alpha+\gamma>0$.  Then for any $\chi_{i4}\in C^\infty_c(U_i)$ with $\chi_{i1}\ll \chi_{i4}$, we have the regularity estimate
\begin{eqnarray*}
&& \Big( \kappa_{i*} \big(\chi_{i1} F(f) \big) \Big)\odot   \Big(  \kappa_{i*}\left( \chi_{i2}g  \right)    \prec\kappa_{i*}\left( \chi_{i3} h \right) \Big)\\
&&\quad=  \kappa_{i*} \left(\chi_{i1} F'(f) \right)   \kappa_{i*}\left( \chi_{i2}g  \right)  \Big(  \kappa_{i*}\left( \chi_{i3} h  \right) \odot \kappa_{i*} \left(\chi_{i4} f  \right) \Big)+\mcC^{(2\alpha+\gamma)\wedge (\alpha+\beta+\gamma)}.
\end{eqnarray*} 
As consequence we have
\begin{eqnarray*} \label{ThmCorrector}
F(f)\odot_1\left(g\prec_2 h \right) = F'(f)\cdot_1(g\cdot_2(h\odot _3 f))+\mcC^{(2\alpha+\gamma)\wedge (\alpha+\beta+\gamma)},
\end{eqnarray*}
where 
\begin{eqnarray*}
F'(f)\cdot_1(g\cdot_2(h\odot _3 f))\defeq\sum_i    \kappa_i^*\Big[ \psi_i  \kappa_{i*} \left(\chi_{i1} F'(f) \right)   \kappa_{i*}\left( \chi_{i2}g  \right)  \Big(  \kappa_{i*}\left( \chi_{i3} h  \right) \odot \kappa_{i*} \left(\chi_{i4} f  \right) \Big)\Big].
\end{eqnarray*}

\end{lemm}

\smallskip

\begin{proof}
The key observation  is that in flat space, for $F_i(x,y)= \left( \kappa_{i*} \chi_{i1} \right)(x)F(y)$, for any cut-off $\chi_{i4}$ such that $\chi_{i4}=1$ on
the support of $\chi_{i1}$, we have the two identities: 
\begin{eqnarray*}
&\kappa_{i*} \big(\chi_{i1} F(f) \big)(x) = \left( \kappa_{i*} \chi_{i1} \right)(x)F(\kappa_{i*} \left(\chi_{i4} f\right)(x)) = F_i\big(x,\kappa_{i*} \left(\chi_{i4} f\right)(x)\big),   \\ 
&\left(\kappa_{i*} \chi_{i1} \right)(x) F^\prime(\kappa_{i*} \big(\chi_{i4} f\big)(x) \big) = (\partial_yF_i)\big(x,\kappa_{i*} \left(\chi_{i4} f\right)(x)\big),
\end{eqnarray*}
since for every $x$ such that $\chi_{i4}(x)=1$ the term $\kappa_{i*} \chi_{i1}(x)$ in factor must vanish.
Hence, applying to $F_i\big(x,\kappa_{i*} \left(\chi_{i4} f\right)(x)\big)$ the estimates of Bony~\cite[Prop 4.4 p.~230]{Bony}, Meyer~\cite[Thm 2 p.~281]{Meyer2} proved on $\mathbb{R}^d$ implies that:
\begin{eqnarray*}
F_i\big(\cdot,\kappa_{i*} \left(\chi_{i4} f\right)(\cdot)\big) - (\partial_yF_i)\big(\cdot,\kappa_{i*} \left(\chi_{i4} f\right)(\cdot)\big)\prec (\kappa_{i*} \left(\chi_{i4} f\right))\in  \mathcal{C}^{2\alpha}(\mathbb{R}^{d}).
\end{eqnarray*} 
By the two above identities, we get
$$
\kappa_{i*} \left(\chi_{i1} F(f) \right) - \kappa_{i*} \left(\chi_{i1} F'(f) \right)\prec  \kappa_{i*} \left(\chi_{i4} f\right) \in \mcC^{2\alpha}. 
$$ 
Hence combining with the estimate $\|u\odot v\|_{\alpha_1+\alpha_2}\lesssim \|u\|_{\alpha_1}\|v\|_{\alpha_2} $ for $\alpha_1+\alpha_2>0$ yields 
\begin{equation*}  \begin{split}
& \Big( \kappa_{i*} \big(\chi_{i1} F(f) \big) \Big)\odot   \Big(  \kappa_{i*}\left( \chi_{i2}g  \right)    \prec\kappa_{i*}\left( \chi_{i3} h \right) \Big)  \\
&=  \Big( \kappa_{i*} \left(\chi_{i1} F'(f) \right)\prec  \kappa_{i*} \left(\chi_{i4} f\right) \Big)\odot    \Big(  \kappa_{i*}\left( \chi_{i2}g  \right)    \prec\kappa_{i*}\left( \chi_{i3} h \right)\Big)   +\mcC^{2\alpha+ \gamma}.
\end{split} \end{equation*}
Denote by 
$$
{\bf C}(f, g, h)= (f\prec g)\odot h-f(g\odot h)
$$  
the (flat) commutator in $\mathbb{R}^d$. Then we  apply  two times the flat commutator estimate as follows:
\begin{eqnarray} \nonumber 
  && \Big( \kappa_{i*} \big(\chi_{i1} F(f) \big) \Big)\odot   \Big(  \kappa_{i*}\left( \chi_{i2}g  \right)    \prec\kappa_{i*}\left( \chi_{i3} h \right) \Big)   \\ \nonumber
  &&\quad = \Big( \kappa_{i*} \left(\chi_{i1} F'(f) \right)\prec  \kappa_{i*} \left(\chi_{i4} f\right)\Big) \odot    \Big(  \kappa_{i*}\left( \chi_{i2}g  \right)    \prec\kappa_{i*}\left( \chi_{i3} h \right)\Big)  +\mcC^{2\alpha+ \gamma}   \\ \nonumber 
  &&\quad =    {\bf C}\Big(  \kappa_{i*} \left(\chi_{i1} F'(f) \right), \kappa_{i*} \left(\chi_{i4} f\right) ,     \kappa_{i*}\left( \chi_{i2}g  \right)    \prec\kappa_{i*}\left( \chi_{i3} h \right)  \Big) +\mcC^{2\alpha+ \gamma} \\ \nonumber
  &&\quad \quad  +   \kappa_{i*} \left(\chi_{i1} F'(f) \right) \Big(  \kappa_{i*} \left(\chi_{i4} f\right) \Big) \odot \Big(  \kappa_{i*}\left( \chi_{i2}g  \right)    \prec\kappa_{i*}\left( \chi_{i3} h \right) \Big)\\ \nonumber
  &&\quad=  \kappa_{i*} \left(\chi_{i1} F'(f) \right) \Big(  \kappa_{i*} \left(\chi_{i4} f\right)\Big) \odot  \Big(  \kappa_{i*}\left( \chi_{i2}g  \right)    \prec\kappa_{i*}\left( \chi_{i3} h \right) \Big)+\mcC^{2\alpha+\gamma}\\ \nonumber
  &&\quad =  \kappa_{i*} \left(\chi_{i1} F'(f) \right)    \kappa_{i*}\left( \chi_{i2}g  \right)  \Big(  \kappa_{i*}\left( \chi_{i3} h  \right)\Big) \odot \Big(\kappa_{i*} \left(\chi_{i4} f\right)  \Big) \\ \nonumber
  &&\quad\quad  + {\bf C}\Big(\kappa_{i*}\left( \chi_{i2}g  \right) ,   \kappa_{i*}\left( \chi_{i3} h \right), \kappa_{i*} \left(\chi_{i4} f\right)\Big)+\mcC^{2\alpha+\gamma}\\ \nonumber
&&\quad =  \kappa_{i*} \left(\chi_{i1} F'(f) \right) \left(   \kappa_{i*}\left( \chi_{i2}g  \right)  \left(  \kappa_{i*}\left( \chi_{i3} h  \right) \odot \kappa_{i*} \left(\chi_{i4} f\right)  \right) \right)  +\mcC^{ (2\alpha+ \gamma)\wedge(\alpha+\beta+\gamma)} \label{eq:use_Bony_comm}
\end{eqnarray}
Then we get the desired estimate.
\end{proof}
This estimate will be used in the proof of Theorem \ref{thm_A_2_mfd} in the next section.

\subsection{A local-to-global principle$\boldmath{.}$ \hspace{0.15cm}}
\label{SectionLocalizationLemma}

For the purpose of doing the stochastic estimates in our companion work \cite{BDFT}, we state here a key localization Lemma which allows us to isolate the singularities of a paraproduct $f\prec g$ in terms of the singularities of $g$. The reader can skip this section at first reading.

\smallskip

\begin{lemm}[Key localization Lemma]\label{l:keylocalization}
Let $f,g$ be two tempered distributions on $\mathbb{R}^d$ such that $g$ has compact support. Then for every $\chi\in C^\infty_c(\mathbb{R}^d)$ such that $\chi=1$ on the support of $g$, the differences
$$ 
f\prec g-\chi\left(f\prec g \right), \quad f\prec g-\left(\chi f \right) \prec g, \quad  f\odot g - \chi \left( f \odot g \right),  \quad  f\odot g-\left(\chi f \right) \odot g 
$$
all lie in $C^\infty(\mathbb{R}^d)$, with no regularity assumptions on $f$ and $g$.
\end{lemm}

\smallskip

Note that the two quantities $f\prec g-\chi\left(f\prec g \right), f\prec g-\left(\chi f \right) \prec g $ always exist. However if $f,g\in \mcC^\alpha\times \mcC^\beta$ but $\alpha+\beta\leqslant 0$, then the two difference terms
$f\odot g - \chi \left( f \odot g \right),  f\odot g-\left(\chi f \right) \odot g  $ are only defined by a mollification and limiting procedure but the limiting term is smooth as claimed in the statement of the Lemma. The proof uses a composition theorem for pseudofferential operators different from Proposition \ref{lemm:commutator}.

\smallskip

\begin{proof}
The distribution $f$ belongs to some H\"older $\mcC^\alpha(\bbR^d)$ for some $\alpha\in \mathbb{R}$. We will later see that the regularity of $g$ is almost irrelevant in the arguments that follow.
The key idea is to make appear the Littlewood-Paley projectors
\begin{eqnarray*}
f\prec g=f\prec \left(\chi \widetilde{\chi} g\right)=\sum_{i\geqslant 2} S_{i-2}(f) \Delta_i(\chi \widetilde{\chi} f)
\end{eqnarray*}
for some function $\widetilde{\chi}$ which equals $1$ on the support of $g$ and such that $\chi=1$ in the support of $\widetilde{\chi}$. Hence $f\prec g-\chi\left(f\prec g \right)$ can be decomposed as
\begin{eqnarray*}
f\prec g-\chi\left(f\prec g \right)=\sum_{i\geqslant 2} S_{i-2}(f) \big(\Delta_i(\chi \widetilde{\chi} g)-\chi \Delta_i( \widetilde{ \chi} g)\big) = \sum_{i\geqslant 2} S_{i-2}(f) [\Delta_i,M_\chi]M_{\widetilde{\chi}} \left( g\right).
\end{eqnarray*}
Denote by $M_\chi$, $M_{\widetilde{\chi}}$ the multiplication operators by $\chi$ and ${\widetilde{\chi}}$. We prove that the commutators $[\Delta_i,M_\chi]M_{\widetilde{\chi}}$ form a family of smoothing operators in the semiclassical sense. Let us explain in more detail. Recall that the Littlewood-Paley projector $\Delta_i=\psi(2^{-i}\vert D\vert)$ should be considered as a semiclassical pseudo-differential operator where $\hbar=2^{-i}$ whose symbol lies in the class $S(1)$~\cite[p.~72]{Zworski}. By the composition Theorem~\cite[Thm 4.14 and 4.18]{Zworski}, the respective symbols $c_1(x;\xi),c_2(x;\xi)$ of the composite operators $\Delta_i\circ M_\chi, M_\chi\circ \Delta_i $ equal $\psi(2^{-i}\vert \xi\vert)$ mod $\hbar^\infty\left\langle \xi\right\rangle^{-\infty}$ for all $x$ such that $\chi(x)=1$. (Beware that the multipliers $M_\chi,M_{\widetilde{\chi}}$ and $\Delta_i$ are semiclassical quantizations of symbols in the class $S(1)$~\cite[p.~72]{Zworski}.) Therefore the commutator $[ \Delta_i,M_\chi ]=Op_{2^{-i}}\left( c_1-c_2 \right)$ is the semiclassical quantization of a symbol $(c_1-c_2)(x;\xi)$ that is smoothing semiclassically exactly when $x\in \{\chi=1\}$. Finally set
$$
[ \Delta_i,M_\chi ]\circ M_{\widetilde{\chi}}=Op_{2^{-i}}(c),
$$ 
and note, again by the composition Theorem~\cite[Thm 4.14 and 4.18]{Zworski}, and from the fact that $\chi=1$ on the support of $\widetilde{\chi}$, that 
$$
c(x;\xi)=\mathcal{O}(\hbar^\infty \left\langle \xi\right\rangle^{-\infty})
$$ 
for all $x$. The operator $[ \Delta_i,M_\chi ]\circ M_{\widetilde{\chi}}$ is thus a semiclassical smoothing operator, and we have
$$
\big\Vert [\Delta_i,M_\chi]M_{\widetilde{\chi}} \left( g\right) \big\Vert_{C^N(\mathbb{R}^d)} \lesssim 2^{-mNi} 
$$ 
for all integers $N$ and $m$. We have as a consequence the estimate
\begin{equation*} \begin{split}
\big\Vert f\prec g-\chi\left(f\prec g \right) \big\Vert_{C^N(\bbR^d)} &\leqslant \sum_{i\geqslant 2}  \big\Vert S_{i-2}(f) [\Delta_i,M_\chi]M_{\widetilde{\chi}} \left( g\right) \big\Vert_{C^N(\bbR^d)}   \\
&\lesssim \sum_{i\geqslant 2}  \Vert S_{i-2}(f) \Vert_{C^N(\bbR^d)} \big\Vert [\Delta_i,M_\chi]M_{\widetilde{\chi}} \left( g\right)\big\Vert_{C^N(\bbR^d)}   \\
&\lesssim \sum_{i=2}^\infty 2^{i(N-\alpha)} 2^{-2iN} \lesssim \sum_{i=2}^\infty 2^{-iN}<+\infty.
\end{split} \end{equation*}
The proof for the difference $f\odot g-\chi\left(f\odot g \right)$ is identical: Replace $S_{i-2}(f)$ by $\Delta_j(f)$ for $\vert i-j\vert\leqslant 1$. To control $f\prec g-(\chi f)\prec g$, write
\begin{equation*} \begin{split}
f\prec g-(\chi f)&\prec g =    \\
&\underset{\in C^\infty}{\underbrace{f\prec g - \widetilde{\chi}\chi(f\prec g)}} + \underset{\in C^\infty \text{ since } (\chi f)\prec g \,=\,(\chi f)\prec (\widetilde{\chi}g) }{\underbrace{\widetilde{\chi}\big( (\chi f)\prec g\big)- (\chi f)\prec g  }}  + \widetilde{\chi}\chi(f\prec g) - \widetilde{\chi}\big( (\chi f)\prec g\big)
\end{split}\end{equation*}
where we used twice the previous result, then
\begin{eqnarray*}
f\prec g-(\chi f)\prec g
=\sum M_{\widetilde{\chi}} [ M_\chi, S_{i-2}](f)  \Delta_i(g)  +C^\infty
\end{eqnarray*}
and we repeat the previous commutator arguments using both that
$M_{\widetilde{\chi}},M_\chi,S_{i-2}=\beta(2^{-i+2}\vert D\vert) $, $\psi=\beta(2^{-1}.)-\beta(.)$ are semiclassical operators obtained by quantizing symbols in the class $S(1)$ and the support properties of $\chi,\widetilde{\chi}$. A similar argument also yields that $f\odot g- (\chi f)\odot g $ is smooth.
\end{proof}

\medskip
\noindent {\bf From local to global principle:} Recall that $\left(\kappa_i,U_i\right)_i$ forms a collection of open charts and cover of the closed, compact manifold $M$. Consider the bilinear and trilinear operations 
$$
f\prec_1 g\defeq\sum_i \kappa_i^{*}\Big[\psi_i \Big( \kappa_{i*}\left(f\chi_{i1}\right)\prec \kappa_{i*}\left(g\chi_{i2}\right)  \Big) \Big]
$$
and 
$$ 
f\prec_1(g\prec_2 h)\defeq\sum_i \kappa_i^*\Big[ \psi_i \Big( \kappa_{i*} \left(\chi_{i1}f \right)\Big) \prec   \Big(  \kappa_{i*}\left( \chi_{i2}g \right)    \prec\kappa_{i*}\left( \chi_{i3}h \right)\Big) \Big] 
$$ 
where $\chi_{i1}, \chi_{i2}, \chi_{i3}$ are arbitrary cut-off functions supported in $U_i\subset M$ such that $\chi_{2i}=\chi_{i3}= 1 $ on ${\rm supp} (\chi_{i1})$. Assume we have a \textbf{local form of regularity} which means for every chart index $i$, the functions 
$$
\Big( \kappa_{i*}\left( \chi_{i1} f  \right) \prec \kappa_{i*}\left( \chi_{i2} g \right)\Big) 
$$
and
$$  
\Big( \kappa_{i*} \left(\chi_{i1}f \right)\Big) \prec   \Big(  \kappa_{i*}\left( \chi_{i2}g \right)    \prec\kappa_{i*}\left( \chi_{i3}h \right)\Big) 
$$
are $\mcC^\alpha$ (resp. $C_T\mcC^\alpha$) in some neighborhood of $\kappa_i({\rm supp}(\chi_{i2}))$ and $\kappa_i({\rm supp}(\chi_{i3}))$ respectively. Then $f\prec_1 g$ and $ f\prec_1(g\prec_2 h) $ are both $\mcC^\alpha(M)$ (resp. $C_T\mcC^\alpha(M)$) \textbf{globally and the result does not depend on the choice of cut-off functions} $\psi_i,\chi_{i1},\chi_{i2},\chi_{i3}$ provided they satisfy the compatibility condition on supports previously stated. We can show a similar property for $ \odot, \succ$ instead of $\prec$.

\smallskip

\begin{proof}
This is a trivial consequence of the localization Lemma~\ref{l:keylocalization} since both $\kappa_{i*}\left( \chi_{i1} f  \right) \prec \kappa_{i*}\left( \chi_{i2} g \right)$ and $\big( \kappa_{i*} \left(\chi_{i1}f \right)\big) \prec   \big(  \kappa_{i*}\left( \chi_{i2}g \right)    \prec\kappa_{i*}\left( \chi_{i3}h \right)\big) $ are smooth outside $\kappa_i({\rm supp}(\chi_{i2}))$ and $\kappa_i({\rm supp}(\chi_{i3}))$ respectively. Indeed, for any function $\psi$ which equals $1$ on the support of $\chi_{i2}$ (resp. $\chi_{i3}$), the difference $(1-\psi)\kappa_{i*}\left( \chi_{i1} f  \right) \prec \kappa_{i*}\left( \chi_{i2} g \right)$ (resp. $(1-\psi)\big( \kappa_{i*} \left(\chi_{i1}f \right)\big) \prec   \big(  \kappa_{i*}\left( \chi_{i2}g \right)    \prec\kappa_{i*}\left( \chi_{i3}h \right)\big) $) is smooth by Lemma~\ref{l:keylocalization}.  
\end{proof}
\smallskip

\subsection{Littlewood-Paley-Stein projectors and pseudo-differential operators$\boldmath{.}$ \hspace{0.1cm}}
\label{sect_LPS_meet_parabolic_cal}

We study in this section some commutator lemma involving both pseudo-differential operators and the generalized Littlewood-Paley projectors.

\smallskip

\begin{prop}\label{coro:boundedpsidosLP}
Let $(P^i_k,\widetilde{P}^i_\ell)$ be a pair of generalized Littlewood-Paley-Stein projectors in the sense of Definition \ref{def:generalizedLPSproj} where $i$ is a chart index and $k,\ell$ represents the frequencies $2^k,2^\ell$. For every pseudo-differential operator $A\in \Psi^m_{1,0}(M)$ the series of commutators  
\begin{eqnarray*}
\sum_{\vert k-\ell\vert\leqslant 1} \big( P^i_kA\widetilde{P}^i_\ell-AP^i_k\widetilde{P}^i_\ell \big)
\end{eqnarray*}
converges absolutely in $\Psi^{m-1-\varepsilon}_{1,0}(M)$.
\end{prop}

\smallskip

\begin{proof}
We work in the same chart $\kappa_i:U_i\mapsto \kappa_i(U_i)$ used to define the pair $P,\widetilde{P}$ of generalized Littlewood-Paley-Stein projectors whose representation reads:
$$
P^i_k=\kappa_i^*\psi \Delta_k\kappa_{i*}\chi, \widetilde{P}^i_k=\kappa_i^*\widetilde{\psi} \Delta_k\kappa_{i*}\widetilde{\chi}.
$$

First we use the fact that the sequence $(\sum_{k=\ell-1}^{\ell+1} P_k^i)_\ell$ is bounded in $\Psi^0_{1,0}(M)$. The proof follows
from considering the symbol of $P_k^i$ in the chart $\kappa_i$ which reads
$$
p_k(x;\xi) = \psi(2^{-k}\xi)(\psi\kappa_{i*}\chi)(x).
$$
Since we have the estimate
\begin{equation*} \begin{split}
\big\vert\partial_x^\alpha\partial_\xi^\beta \psi(2^{-k}\xi)(\psi\kappa_{i*}\chi)(x) \big\vert &= \big\vert\partial_\xi^\beta \psi(2^{-k}\xi)  \partial_x^\alpha(\psi\kappa_{i*}\chi)(x)\big\vert   \\
&\lesssim 2^{-k\vert\beta\vert}\Vert \partial_x^\alpha(\psi\kappa_{i*}\chi)\Vert_{L^\infty} \Vert \partial^\beta_\xi\psi \Vert_{L^\infty}\lesssim \vert \xi\vert^{-\vert\beta\vert}
\end{split} \end{equation*}
we deduce that the sequence of symbols $(p_k)_k$ is bounded in $S^0_{1,0}(\mathbb{R}^d)$, hence the sequence $(P_k^i)_k$ and $(\sum_{k=\ell-1}^{\ell+1} P_k^i)_\ell$ are bounded in $\Psi^0_{1,0}(M)$.

Setting the sequence of commutators 
$$
\bigg(B_\ell = \sum_{k=\ell-1}^{\ell+1}\big[P_k^i,A\big]\bigg)_{\ell\geqslant 1},
$$ 
we deduce by the usual commutator estimates in the pseudo-differential calculus that the sequence $(B_\ell)_\ell$ is bounded in $\Psi^{m-1}_{1,0}(M)$. 

Secondly, the series $\sum_{\vert k-\ell\vert\leqslant 1} P^i_kA\widetilde{P}^i_\ell-AP^i_k\widetilde{P}^i_\ell$ rewrites as
\begin{eqnarray*}
\sum_\ell B_\ell \widetilde{P}_\ell^i
\end{eqnarray*}
where the sequence $(B_\ell)_\ell$ is bounded in $\Psi^{m-1}_{1,0}(M)$.
We consider the symbol of each composite operator $B_\ell \widetilde{P}_\ell^i$ in the same chart $\kappa_i:U_i\mapsto \kappa_i(U_i)$ used to define the pair $P,\widetilde{P}$ of generalized Littlewood-Paley-Stein projectors. Recall that
each $\widetilde{P}_\ell^i$ is supported in $U_i\times U_i$. 
We choose some functions $\chi_i,\widetilde{\chi}_i\in C^\infty_c(U_i)^2$ which both equal $1$ on the support of $\kappa_i^{*}\widetilde{\psi}$, $\chi_i=1$ on the support of $\widetilde{\chi}_i$.  
Then using the pair of cut-off functions, we can decompose the previous series into two pieces of different natures: 
\begin{eqnarray*}
\sum_\ell B_\ell \widetilde{P}_\ell^i=\sum_\ell \chi_i B_\ell \widetilde{\chi}_i \widetilde{P}_\ell^i+\sum_\ell (1-\chi_i) B_\ell \widetilde{\chi}_i \widetilde{P}_\ell^i
\end{eqnarray*}
where the operator $(1-\chi_i) B_\ell \widetilde{\chi}_i$ is supported outside the diagonal therefore it is a smoothing operator.
We study the two pieces separately. To study the first piece $\sum_\ell \chi_i B_\ell \widetilde{\chi}_i \widetilde{P}_\ell^i$ precisely, we need to consider the operator $\chi_i B_\ell \widetilde{\chi}_i \widetilde{P}_\ell^i$. We first
conjugate it by $\kappa_i$ to reduce to compactly supported pseudo-differential operators on $\mathbb{R}^d$. This yields:
\begin{eqnarray*}
\kappa_{i*}\left( \chi_i B_\ell \widetilde{\chi}_i \widetilde{P}_\ell^i \right)\kappa_i^* = \left(\kappa_{i*} \chi_i B_\ell \widetilde{\chi}_i \kappa_i^*\right)\left( \kappa_{i*}\widetilde{P}_\ell^i \kappa_i^*\right)
=\left(\kappa_{i*} \chi_i B_\ell \widetilde{\chi}_i \kappa_i^*\right)\widetilde{\psi}\Delta_\ell \kappa_i^*\widetilde{\chi}
\end{eqnarray*}
where $\left(\kappa_{i*} \chi_i B_\ell \widetilde{\chi}_i \kappa_i^*\right)\widetilde{\psi}\in \Psi^{m-1}_{1,0}(\mathbb{R}^d)$ is a bounded sequence of compactly supported pseudo-differential operators on $\mathbb{R}^d$, we used the fact that multiplication by smooth function are pseudo-differential operators of order $0$ and the composition for pseudo-differential operators is bounded.
Therefore there exists a bounded sequence $b_\ell$ of symbols in $S^{m-1}_{1,0}(\mathbb{R}^d)$ such that for all $\ell$:
$\left(\kappa_{i*} \chi_i B_\ell \widetilde{\chi}_i \kappa_i^*\right)\widetilde{\psi}=Op\left( b_\ell\right)$,
\begin{eqnarray*}
\vert \partial_x^\alpha\partial_\xi^\beta b_\ell(x;\xi) \vert\leqslant C_{\alpha,\beta}(1+\vert \xi\vert)^{m-\vert\beta\vert}
\end{eqnarray*}
where the constant $C_{\alpha,\beta}$ does not depend on $\ell$.
Note that the composition $ \left(\kappa_{i*} \chi_i B_\ell \widetilde{\chi}_i \kappa_i^*\right)\widetilde{\psi}\Delta_\ell $ also reads:
\begin{eqnarray*}
Op(b_\ell)\Delta_\ell=Op\left(b_\ell\psi(2^{-\ell}\cdot) \right)=\frac{1}{(2\pi)^d}\int_{\mathbb{R}^d} e^{i\xi \cdot(x-y)} b_\ell(x,\xi)\psi(2^{-\ell}\xi) \rmd\xi\,.
\end{eqnarray*}
Therefore, to prove the convergence of the series
$\sum_\ell \left(\kappa_{i*} \chi_i B_\ell \widetilde{\chi}_i \kappa_i^*\right)\widetilde{\psi}\Delta_\ell$, it suffices to show that
the partial sums 
$\sum_{\ell\leqslant N} b_\ell(x,\xi)\psi(2^{-\ell}\xi)$ are bounded in the space $S^{m-1}_{1,0}(\mathbb{R}^d)$ of symbols of order $m-1$ but the series converges in the space $S^{m-1+\varepsilon}_{1,0}(\mathbb{R}^d)$ for all $\varepsilon>0$.
This is a consequence of the partition of unity identity
$1=\psi_0(\xi)+\sum_\ell \psi(2^{-\ell}\xi)$,
\begin{eqnarray*}
\big\vert \partial_x^\alpha\partial_\xi^\beta \sum_\ell b_\ell(x,\xi)\psi(2^{-\ell}\xi) \big\vert \leqslant \sum_{\ell, \vert \xi\vert\simeq 2^\ell} \big\vert \partial_\xi^\beta (\partial_x^\alpha b_\ell)(x,\xi)\psi(2^{-\ell}\xi) \big\vert \lesssim \sum_{\ell, \vert \xi\vert\simeq 2^\ell} 2^{j(m-\vert\beta \vert)}\lesssim (1+\vert\xi\vert)^{m-\vert\beta\vert}
\end{eqnarray*}
where we used the Leibniz rule and also the fact that given $\xi$, the series $\sum_\ell b_\ell(x,\xi)\psi(2^{-\ell}\xi)$ reduces to a finite sum
$\sum_{\ell, \vert \xi\vert\simeq 2^\ell} b_\ell(x,\xi)\psi(2^{-\ell}\xi)$.
Therefore the series
\begin{eqnarray*}
\sum_\ell \left(\kappa_{i*} \chi_i B_\ell \widetilde{\chi}_i \kappa_i^*\right)\widetilde{\psi}\Delta_\ell=\sum_\ell Op(b_\ell)\Delta_\ell=
\frac{1}{(2\pi)^d}\int_{\mathbb{R}^d} e^{i\xi\cdot(x-y)} \sum_\ell\left( b_\ell(x,\xi)\psi(2^{-\ell}\xi)\right) \rmd\xi
\end{eqnarray*}
defines a pseudo-differential operators in $\Psi^{m-1}_{1,0}(\mathbb{R}^d)$.
Again by the invariance of pseudo-differential operators under diffeomorphisms, we get  
$\sum_\ell \chi_i B_\ell \widetilde{\chi}_i \widetilde{P}_\ell^i\in \Psi^{m-1}_{1,0}(M)$.

It remains to deal with the term $\sum_\ell (1-\chi_i) B_\ell \widetilde{\chi}_i \widetilde{P}_\ell^i $. First note that the sequence $\big((1-\chi_i) B_\ell \widetilde{\chi}_i\big)_\ell$ is bounded in $\Psi^{-\infty}(M)$. It suffices to prove that for any smooth function with sufficiently small support, the operator $\chi_2 \left(\sum_\ell (1-\chi_i) B_\ell \widetilde{\chi}_i \widetilde{P}_\ell^i\right)$ is smoothing. For any chart $\Psi: V\mapsto \Psi(V)$, choose any function $\chi_2\in C^\infty_c(V)$, then we reduce the study of $\chi_2 \left(\sum_\ell (1-\chi_i) B_\ell \widetilde{\chi}_i \widetilde{P}_\ell^i\right)$ to
\begin{equation*} \begin{split}
\Psi_{*}\left(\chi_2(1- \chi_i) B_\ell \widetilde{\chi}_i \widetilde{P}_\ell^i \right)\kappa_i^* &= \left(\Psi_{*} \chi_2(1- \chi_i) B_\ell \widetilde{\chi}_i \kappa_i^*\right)\left( \kappa_{i*}\widetilde{P}_\ell^i \kappa_i^*\right)   \\
&=\left(\Psi_{*} \chi_2(1- \chi_i) B_\ell \widetilde{\chi}_i \kappa_i^*\right)\widetilde{\psi}\Delta_\ell \kappa_i^*\widetilde{\chi}.
\end{split} \end{equation*}
Now, it is an immediate consequence of the composition theorem for pseudo-differential operator that the operator $\left(\Psi_{*} \chi_2(1- \chi_i) B_\ell \widetilde{\chi}_i \kappa_i^*\right)\widetilde{\psi}$ is smoothing on $\mathbb{R}^d$, so arguing as above, we can conclude that the series
$\sum_\ell \left(\Psi_{*} \chi_2(1- \chi_i) B_\ell \widetilde{\chi}_i \kappa_i^*\right)\widetilde{\psi}\Delta_\ell$ converges in $\Psi^{-\infty}(\mathbb{R}^d)$ which concludes the proof of Proposition \ref{coro:boundedpsidosLP}.
\end{proof}

\smallskip

\section{Commutator estimates for paradifferential operators}
\label{s:paradiffcommutators}

The goal of this section is to recall the strict minimum material in paradifferential calculus to control in the following section the commutator $[e^{-tP},P_u]$ where $P_u$ is the paramultiplication operator $u\prec$ for some H\"older function $u$. A simple idea for a simple goal: If we are able to see $P_u$ as an operator in some well-behaved class with a good composition theorem the control of the commutator $[e^{-tP},P_u]$ will be a direct consequence of this composition theorem.

The paraproduct operators are examples of paradifferential operators. After some recollection on this class of operators in Section \ref{SubsectionRecollectionParadiff} we introduce in Section \ref{SubsectionSimplePararegularization} a useful regularization procedure and prove in Section \ref{CompositionParadiffOperators} a composition result for some paradifferential operators. With end this section with a key localization lemma that somehow allows to isolate the singularities of a paraproduct $f\prec g$ in terms of the singularities of $g$ -- see Section \ref{SectionLocalizationLemma}.

\smallskip

\subsection{Recollection on paradifferential operators on $\mathbb{R}^d$$\boldmath{.}$ \hspace{0.1cm}}
\label{SubsectionRecollectionParadiff}

We mostly follow the notations and terminology of Meyer \cite{Meyer2}. To illustrate the notion of paradifferential operator we take a new look at the paraproduct operator. For $u\in \mcC^\alpha(\mathbb{R}^d)$ with $\alpha>0$ we define the linear operator
$$
P_u:v\in \mathcal{S}^\prime(\mathbb{R}^d)\mapsto u\prec v  
$$
where $u\prec v=\sum_{i\geqslant 5} S_{ i-5}(u) \Delta_i(v)$. The operator $P_u$ has symbol
\begin{eqnarray*}
\sigma(x;\xi)=\sum_{i=5}^\infty S_{ i-5}(u)(x)\psi(2^{-i}|\xi|) 
\end{eqnarray*} 
where $\psi$ generates the Littlewood-Paley-Stein partition of unity. From now on we assume without loss of generality that $\psi$ vanishes outside the corona $ \frac{1}{2}\leqslant \vert \xi\vert \leqslant 4$ and equals $1$ on the smaller corona $1\leqslant \vert\xi \vert\leqslant 2$. For a function $u$ of positive H\"older regularity it is proved in \cite[p.~292]{Meyer2} that the above symbol $\sigma$ belongs to the class $A^0_\alpha$ that we define following \cite[Definition 1 p.~286]{Meyer2}.

\smallskip

\begin{defi*} 
 A symbol $\sigma
\in C^\infty(\mathbb{R}^d\times \mathbb{R}^d)$ belongs to the space $A^m_\alpha\left( \mathbb{R}^d\right)$ for $m\in \mathbb{R}$ and $\alpha>0$ if:
\begin{enumerate}
\item 
 For every multiindex $\gamma$, there exists a constant $C_{\gamma}$ such that
\begin{eqnarray} \label{ineq1}
\Vert\partial_\xi^\gamma \sigma(x;\xi)\Vert_{\mcC^\alpha_x\left(\mathbb{R}^d\right)} \leqslant C_{\gamma} \left(1+\vert\xi\vert\right)^{m-\vert \gamma\vert}.
\end{eqnarray}
\item  For every multiindices $(\beta,\gamma)$, there exists a constant $C_{\beta, \gamma}$ such that if $\vert\beta\vert >\alpha$ then
\begin{eqnarray} \label{ineq2}
\Vert\partial_x^\beta\partial_\xi^\gamma \sigma(x;\xi)\Vert_{\mcC^\alpha_x\left(\mathbb{R}^d\right)}
\leqslant C_{\beta,\gamma }  \left(1+\vert\xi\vert\right)^{m-\vert \gamma\vert+\vert\beta\vert-\alpha}.
\end{eqnarray} 
\end{enumerate}
\end{defi*}

\smallskip

We have 
$$
\bigcap_{\alpha\in \mathbb{R}_{\geqslant 0}} A^m_\alpha(\mathbb{R}^d)=S^m_{1,0}(\mathbb{R}^d)
$$ 
and the inclusion 
$$
S^m_{1,0}(\mathbb{R}^d)\subset A^m_\alpha(\mathbb{R}^d).
$$ 
Then \cite{Meyer2} introduces a second class denoted by $B^m_\alpha(\mathbb{R}^d)$ as follows.

\smallskip

\begin{defi} [The class $B^m_\alpha $] \label{def:classB}
A symbol $\sigma$ is said to be in the class $B^m_\alpha(\mathbb{R}^d)$ if \eqref{ineq1} holds and there exists $0<K<1$ such that for each fixed $\xi$ the partial Fourier transform $\widehat{\sigma}(\eta,\xi)$ in $x$ of the symbol $\sigma$ is supported in the set $\big\{\vert\eta\vert\leqslant  K \vert \xi\vert\big\}$. 
\end{defi}

\smallskip

Then it is claimed that~\cite[bottom p.~286]{Meyer2} (see also~\cite[p.~292]{Meyer2}):

\smallskip

\begin{lemm}\label{l:BinA}
We have the inclusion $B^m_\alpha(\mathbb{R}^d)\subset A^m_\alpha(\mathbb{R}^d)$ and $P_u\in B^0_\alpha(\mathbb{R}^d)$.
\end{lemm}

\smallskip

We check for pedagogical purposes that, for $u\in \mcC^\alpha(\mathbb{R}^d)$, the paramultiplication operator $P_u$ belongs to the class $B^0_\alpha(\mathbb{R}^d)$. Recall that its symbol reads $\sigma(x,\xi)=\sum_{i=5}^\infty S_{ i-5}(u)(x)\psi(2^{-i}|\xi|) $, hence the Fourier transform with respect to  the variable $x$ reads 
$$
\widehat{\sigma}(\eta,\xi)=\sum_{i=5}^\infty \widehat{S_{ i-5}(u)}(\eta)\psi(2^{-i}|\xi|).
$$ 
Note that by definition of our dyadic decomposition, $\widehat{S_{ i-5}(u)}$ is supported on a ball of radius $\leqslant 2^{i-3}$ and $\psi(2^{-i}.) $ is supported in the corona $\big\{ 2^{i-1}\leqslant\vert \xi\vert\leqslant 2^{i+2}\big\}$
so the Fourier vanishing condition is satisfied.

\subsection{A simple pararegularization$\boldmath{.}$ \hspace{0.1cm}}
\label{SubsectionSimplePararegularization}

Despite its usefulness in several nonlinear problems, it is well-known since the work of Bourdeau, Stein~\cite[Chapter IX]{Horm97} that the class $S^m_{1,1}(\bbR^d)$ is ill-defined when acting on Sobolev or H\"older spaces of negative regularity. Since in the study of SPDEs the operators act on Besov spaces of negative regularity we need to modify the symbols in the class $S^m_{1,1}(\bbR^d)$ by some cut-off function to make them well behaved on Besov spaces of singular distributions. For this, we first define a specific class of cut-off functions.

\smallskip

\begin{defi} [cut-off function] \label{def:parareg}
\textbf{In the sequel, given $0<K_1<K_2<1$, we choose some  bounded cut-off function $\chi\in C^\infty(\mathbb{R}^d\times\mathbb{R}^d)$ such that
{
$$
|\partial_{\eta, \xi}^\alpha \chi(\eta, \xi)|\leqslant C_\alpha(1+|\xi|+ |\eta|)^{-|\alpha|} \, ,$$
}
$\chi=0$ near 
$(\eta,0)$, $\chi=0$ when $\vert \eta\vert>K_2\vert\xi\vert$ and $\chi=1 $ on $ \vert\eta\vert\leqslant K_1\vert \xi\vert$.}
\end{defi}

\smallskip

We next define a kind of smoothing procedure for symbols called \emph{pararegularization} which is a simplified version of what can be found in Section 10.2 of H\"ormander's book \cite{Horm97}.

\smallskip

\begin{defi}[Pararegularization] \label{def:parareg2}
With this choice of cut-off functions, starting from any $\sigma\in C^\infty(\mathbb{R}^d\times \mathbb{R}^d)$ satisfying \eqref{ineq1}, our pararegularized symbols $\sigma_\chi$ 
is defined from the condition 
$$
\widehat{\sigma_\chi}(\eta,\xi)=\widehat{\sigma}(\eta,\xi)\chi(\eta,\xi).
$$
\end{defi}

\smallskip

This operation of Fourier cut-off will always produce some symbol $\sigma_\chi$ which belongs to the class $B^m_\alpha(\mathbb{R}^d)$ of Definition \ref{def:classB}. It is obvious by construction that our paramultiplication operator $P_u$ is exactly a pararegularized operator of the form $Op(\sigma_\chi)$ for some cut-off $\chi\in C^\infty(\mathbb{R}^d\times \mathbb{R}^d)$ since its symbol 
$$
\sigma(x,\xi) = \sum_{i=5}^\infty S_{ i-5}(u)(x) \, \psi(2^{-i}|\xi|)
$$ 
vanishes near $\xi=0$, by the support property of the Littlewood-Paley-Stein partition of unity function $\varphi(.)$; and its Fourier transform  $ \widehat{\sigma}(\eta,\xi)=\sum_{i=5}^\infty \widehat{S_{ i-5}(u)}(\eta) \, \psi(2^{-i}|\xi|) $ also vanishes near the  twisted diagonal $\{(-\xi,\xi)\}\subset \mathbb{R}^d\times \mathbb{R}^d$
since $ \widehat{\sigma}(\eta,\xi) $
vanishes when $\vert\eta \vert> \frac{1}{4} \vert \xi\vert$. (Indeed $\vert\eta \vert\leqslant 2^{i-3}$ and $\vert \xi\vert\geqslant 2^{i-1}$ imply that $ \frac{\vert \eta\vert}{\vert \xi\vert}\leqslant \frac{2^{i-3}}{2^{i-1}}=2^{-2}$.) Now we shall use the fact that the paradifferential regularization of a classical pseudo-differential operator preserves its properties~\cite[p.~236]{Horm97}: 

\smallskip

\begin{lemm} [Paradifferential regularization of classical symbols] \label{lem:pararegclassicalsymbols}
Let $a\in C^\infty(\mathbb{R}^d\times \mathbb{R}^d)$ be a classical symbol in $S^m_{1,0}(\mathbb{R}^d)$, $\chi$ is a cut-off function from Definition \ref{def:parareg} and $a_\chi$ the cut-off symbol as defined in Definition \ref{def:parareg2}. Then the difference
$a-a_\chi\in S^{-\infty}_{1,0}(\mathbb{R}^d)$.
\end{lemm}

\smallskip

This means that $a=a_\chi$ modulo smoothing operators.

\smallskip

\begin{proof}
We need to prove that $\mathcal{F}^{-1}_{\eta,\xi} (\widehat{a}(1-\chi))(x,x-y) $ is a smooth function in the $x$ variable and in the Schwartz class in $x-y$. This means that for a $K\in C^\infty(\mathbb{R}^d_x\times \mathbb{R}^d_h)$:
$$ \vert\left( \partial_x^\alpha \partial_h^\beta  K \right) (x,x-y) \vert \leqslant C_{\alpha,\beta, N} (1+\vert x-y\vert)^{-N}  $$
for all multiindices $\alpha,\beta$ and $N\geqslant 0$.
Since this is the Schwartz kernel of $Op\left(a-a_\chi \right)$, this implies that 
the difference $a-a_\chi$ is a smoothing symbol in $S^{-\infty}_{1,0}(\mathbb{R}^d)$.
Because of the support properties of $\chi$, the symbol $\widehat{a}(\eta,\xi)(1-\chi)(\eta,\xi)$ is non-vanishing only when $\vert \eta\vert> K\vert \xi\vert$ for some $K>0$. On this subset, we have an inequality of the form $$(1+\vert \eta\vert)^{-1}\leqslant (1+K\vert\xi\vert)^{-1}\lesssim (1+\vert \xi\vert)^{-1}, $$ therefore $\widehat{a}(\eta,\xi)(1-\chi)(\eta,\xi)$ satisfies, for all $N$, the estimate
\begin{eqnarray*}
\big\vert \widehat{a}(1-\chi)(\eta,\xi) \big\vert &\lesssim & \sup_\xi \big\Vert(1+\vert\xi\vert)^{-m} a(\cdot,\xi) \big\Vert_{C^{2N+m}_x}
(1+\vert \eta\vert)^{-2N-m}(1+\vert \xi\vert)^m (1-\chi)(\eta,\xi) \\
&\lesssim &(1+\vert\xi\vert)^{-N}(1+\vert \eta\vert)^{-N}.
\end{eqnarray*}
(For the second inequality, we used the fact that for any $U\in C^\infty_c(\mathbb{R}^d)$, one has $ \vert\widehat{U}(\xi) \vert  \lesssim \Vert U \Vert_{C^m(\mathbb{R}^d)} (1+\vert \xi\vert)^{-m}$, which follows from  integration by parts.) 
{
Similarly, we have
\begin{align*}
&\big\vert |\partial^\alpha_\xi  \widehat{a} (\eta,\xi)   \partial^\beta_\xi  (1-\chi)(\eta,\xi) \big\vert \\
 &\quad \lesssim  \sup_\xi \big\Vert(1+\vert\xi\vert)^{-m +|\alpha|} \partial^\alpha_\xi a(\cdot,\xi) \big\Vert_{C^{2N+m}_x}
(1+\vert \eta\vert)^{-2N-m}(1+\vert \xi\vert)^{m-|\alpha|} C_\beta (1+|\xi| + |\eta|)^{-|\beta|} \\
& \quad \lesssim (1+\vert\xi\vert)^{-N}(1+\vert \eta\vert)^{-N}.
\end{align*}
This second estimate ensures that we have sufficient decay in $x-y$ once we invert the Fourier transform to get the Schwartz kernel.
} 
Therefore the inverse Fourier transform in $(\xi,\eta)$ of $\widehat{a}(1-\chi)(\eta,\xi)$ yields a smooth kernel with the desired property. 
\end{proof}

\smallskip

We obtain the following result as a consequence of Lemma \ref{lem:pararegclassicalsymbols} and the Schauder estimate from Proposition \ref{l:simplifiedprooftwistedclass}.

\smallskip

\begin{coro}
Each operator in the class $\Psi^m_{1,0}(\mathbb{R}^d)$ sends continuously $B_{p,q}^\alpha(\mathbb{R}^d)$ into $B_{p,q}^{\alpha-m}(\mathbb{R}^d)$.
\end{coro}

\smallskip

The same result holds for operators and Besov spaces on $M$.

\smallskip

\subsection{Composition of paradifferential operators$\boldmath{.}$ \hspace{0.1cm}}
\label{CompositionParadiffOperators}

The following commutator result is useful.

\smallskip

\begin{prop}[Commutator Lemma]\label{lemm:commutator}
Let $\alpha\in (0,1)$, $(m_1,m_2)\in \mathbb{R}^2$. If $a\in S^{m_1}_{1,0}(\mathbb{R}^d)$ and $b\in B^{m_2}_\alpha(\mathbb{R}^d)$ then the commutator 
$$
\big[Op(a) , Op(b)\big] = Op(c)+R
$$ 
where $R\in \Psi^{-\infty}(\mathbb{R}^d)$ is smoothing and $c$ lies in $S^{m_1+m_2-\alpha}_{1,1}(\bbR^d)$ and has a \textbf{pararegularized symbol} $\widehat{c}(\eta,\xi)$ supported on $\vert\eta\vert\leqslant K\vert\xi\vert$ for some $K<1$.
\end{prop}

\smallskip

The fact that the commutator $Op(c)$ has pararegularized symbol $c$ has central importance for us since it will allow $Op(c)$ to act on some Besov spaces of non-positive regularities. Proposition \ref{lemm:commutator} follows from the following more general composition result.

\smallskip

\begin{prop}[Composition Lemma] \label{lemm:composition}
Let $\alpha\in (0,1)$, $(m_1,m_2)\in \mathbb{R}^2$. If $a\in B^{m_1}_{\alpha}(\mathbb{R}^d)$ and $b\in B^{m_2}_\alpha(\mathbb{R}^d)$ with the constant $K$ appearing in Definition \ref{def:classB} satisfies $K\leqslant \frac{1}{4}$. 
Then 
$$
Op(a)\circ Op(b)=  Op(ab)+Op(c)
$$ 
where $c\in S^{m_1+m_2-\alpha}_{1,1}(\mathbb{R}^d)$ and there is some $0<\widetilde{K}<1$ such that $\widehat{c}(\eta,\xi)$ is supported on $\vert\eta\vert\leqslant \widetilde{K}\vert\xi\vert$.
\end{prop}

\smallskip

We deduce Proposition \ref{lemm:commutator} from the composition result applied to the pararegularized $a_\chi$ instead of $a$ -- they differ from a smoothing operator, from Lemma \ref{lem:pararegclassicalsymbols}, and since 
$$
[Op(a),Op(b)]= [Op(a_\chi),Op(b)] \text{ mod} \left(\Psi^{-\infty}\right) = Op(c)  \text{ mod}\left( \Psi^{-\infty} \right)
$$ 
for some $c\in S^{m_1+m_2-\alpha}_{1,1}(\mathbb{R}^d)$ and $\widehat{c}(\eta,\xi)$ is supported on $\vert\eta\vert\leqslant \widetilde{K}\vert\xi\vert$ for some $0<\widetilde{K}<1$.

\smallskip

\begin{proof}
We give here a self-contained proof of Proposition \ref{lemm:composition} essentially following Meyer's exposition in \cite[Theorem 4]{Meyer2}. Therein the remainder term belongs to $S^{-\alpha}_{1,1}(\bbR^d)$ since one of the symbols is only in $S^0_{1,1}(\bbR^d)$. In our case, we make the stronger assumption that the symbols are in $B^{m_1}_\alpha(\bbR^d)$ and $B^{m_2}_\alpha(\bbR^d)$. Hence we need to check that our symbol $c$ is in fact a pararegularized symbol, which means $\widehat{c}$ vanishes outside $\{ \vert \eta\vert\leqslant \widetilde{K}\vert\xi\vert \}$ for some $\widetilde{K}\in (0,1)$. This is sufficient for $Op(c)$ to act on some Besov distribution $v$ of negative regularity.

\smallskip

As usual, we start from the Fourier representation formula for $c$, which reads: 
\begin{eqnarray*}
c(x;\xi)=\frac{1}{(2\pi)^d}\int_{\mathbb{R}^d} \big(a(x,\xi+\eta)-a(x;\xi)\big) e^{i\eta\cdot x}\,\widehat{b}(\eta,\xi) \rmd\eta
\end{eqnarray*}
where $\widehat{b}$ is supported on $\vert \eta\vert\leqslant\frac{\vert\xi\vert}{10}$.
Now we rewrite the representation formula for $c$ but inserting a dyadic decomposition in the integration variable $\eta$
\begin{eqnarray*}
c(x;\xi) = \frac{1}{(2\pi)^d}\sum_{j;2^{j-1}\leqslant\frac{\vert\xi\vert}{10}} 
\int_{\mathbb{R}^d} \big(a(x,\xi+\eta)-a(x;\xi)\big) \psi(2^{-j}\eta) e^{i\eta\cdot x} \, \widehat{b}(\eta,\xi)\rmd\eta,
\end{eqnarray*}
the summation is over some finite number of $j$ since for fixed $\xi$ the integrand vanishes when $\vert\eta\vert > \frac{\vert\xi\vert}{10}$.
Choose some cut-off function $\widetilde{\chi}\in C^\infty_c(\mathbb{R}^d\setminus\{0\})$ such that $\widetilde{\chi}=1$ on the support of $\psi$.
\begin{eqnarray*}
\vert c(x;\xi)\vert&\leqslant&\frac{1}{(2\pi)^d}\sum_{j;2^{j-1}\leqslant\frac{\vert\xi\vert}{10}} 
\bigg\vert \int_{\mathbb{R}^d} \big(a(x,\xi+\eta)-a(x;\xi)\big)\widetilde{\chi}(2^{-j}\eta) \psi(2^{-j}\eta) e^{i\eta\cdot x}\,\widehat{b}(\eta,\xi)\rmd\eta \bigg\vert   \\
&\leqslant&
\frac{1}{(2\pi)^d} \sum_{j;2^{j-1}\leqslant\frac{\vert\xi\vert}{10}} \Vert A_j\Vert_{L^1(\mathbb{R}^d)}\Vert B_j \Vert_{L^\infty(\mathbb{R}^d)} 
\end{eqnarray*}
where 
$$
A_j=\mathcal{F}^{-1}_\eta\left(\left(a(x,\xi+\eta)-a(x;\xi)\right)\widetilde{\chi}(2^{-j}\eta) \right), \quad B_j=\mathcal{F}^{-1}_\eta\left( \psi(2^{-j}\eta)e^{i\eta\cdot x} \widehat{b}(\eta,\xi) \right)
$$ 
and the variables $(x;\xi)$ are treated like parameters. By the definition of $b\in B_\alpha^{m_2}(\bbR^d)$, we have the bound
\begin{eqnarray*}
\Vert B_j\Vert_{L^\infty(\mathbb{R}^d)}\lesssim 2^{-j\alpha} (1+\vert \xi\vert)^{m_2}.
\end{eqnarray*}
For the control of $\Vert A_j\Vert_{L^1(\mathbb{R}^d)}$, we use some ideas from the Wiener algebra. Here is the key observation: By scale invariance, the $L^1$ norm of $A_j$ is equal to the $L^1$ norm of the rescaled function 
$$
\mathcal{F}^{-1}_\eta\Big(\left(a(x,\xi+2^j\eta)-a(x;\xi)\right)\widetilde{\chi}(\eta) \Big).
$$
Therefore, we have
$$
\Vert A_j\Vert_{L^1(\mathbb{R}^d)} = \big\Vert\mathcal{F}^{-1}_\eta\left(\left(a(x,\xi+2^j\eta)-a(x;\xi)\right)\widetilde{\chi}(\eta) \right)\big\Vert_{L^1(\mathbb{R}^d)}\lesssim \big\Vert
\left(a(x,\xi+2^j\eta)-a(x;\xi)\right)\widetilde{\chi}(\eta) \big\Vert_{H^s_\eta}  
$$
for all $s>\frac{d}{2}$. We used the following fundamental fact, for any function $U$:
\begin{eqnarray*}
&&\Vert \widehat{U}\Vert_{L^1(\mathbb{R}^d)} = \int_{\mathbb{R}^d}  
\big\vert (1+\vert\xi\vert)^{-s} (1+\vert\xi\vert)^s\widehat{U}(\xi) \big\vert \rmd\xi \\
&& \leqslant \left(\int_{\mathbb{R}^d} (1+\vert\xi\vert)^{-2s} \rmd\xi\right)^{\frac{1}{2}} \left( \int_{\mathbb{R}^d}  
\vert  (1+\vert\xi\vert)^s\widehat{U}(\xi) \vert^2 \rmd\xi \right)^{\frac{1}{2}} \lesssim \Vert U \Vert_{H^s} 
\end{eqnarray*}
where we used Cauchy--Schwartz in the second estimate and 
the definition of Sobolev norms, the right hand side is finite as soon as $s>\frac{d}{2}$. So we need to estimate the Sobolev regularity in $\eta$ of
$$
g_j(\eta)= \left(a(x,\xi+2^j\eta)-a(x;\xi)\right)\widetilde{\chi}(\eta)
$$ 
by just bounding the derivatives. We remind the reader that $(x,j,\xi)$ appearing in the definition of $g_j$ are treated as parameters. To control the difference $ \left(a(x,\xi+2^j\eta)-a(x;\xi)\right)$, we will use the fundamental Theorem of calculus. Recall that $2^j\leqslant \frac{\vert\xi\vert}{5}$ if $2^{j-1}\leqslant \frac{\vert\xi\vert}{10}$, therefore on the support of $\widetilde{\chi}$ that we assume is contained in $ \{  \vert\eta\vert\leqslant 9/2 \} $, we always have inequalities of the form
\begin{eqnarray*}
\vert 2^j\eta\vert\leqslant 92^j/2 \leqslant \frac{9}{10}\vert\xi\vert \implies  \frac{1}{10} \vert \xi\vert \leqslant \vert \xi+2^j\eta\vert\leqslant \frac{19}{10}\vert \xi\vert.
\end{eqnarray*}
We use the fundamental Theorem of calculus and the regularity of $a$ in $\xi$
\begin{eqnarray*}
\vert a(x;\xi)-a(x;2^j\eta+\xi) \vert\lesssim \underset{\lesssim (1+\vert \xi\vert)^{m_1-1}}{ \underbrace{ \left( \sup_{ \frac{1}{10}\vert\xi\vert\leqslant \vert\widetilde{\xi}\vert \leqslant \frac{19}{10}\vert\xi\vert}\big\vert \partial_{\widetilde{\xi}} a(x;\widetilde{\xi})\big\vert \right)}}  2^j\vert\eta\vert \lesssim 2^j(1+\vert \xi\vert)^{m_1-1}
\end{eqnarray*}
where we used estimate (\ref{ineq1}) to control the derivative $\partial_{\widetilde{\xi}} a(x;\widetilde{\xi})$. For all multiindices $\vert\beta \vert\geqslant 1$, we again use the Fundamental Theorem of calculus to obtain
\begin{equation*} \begin{split}
\big\vert \partial_\eta^\beta \left(a(x;\xi)-a(x;2^j\eta+\xi)\right)\big\vert &= \bigg\vert \partial_\eta^\beta\left( 2^j\int_0^1 d_\xi a(x;\xi+u2^j\eta) (\eta)  \rmd u \right) \bigg\vert   \\
&\lesssim (1+\vert \xi\vert)^{m_1-\vert \beta\vert-1}2^{j(1+\vert \beta\vert)}+(1+\vert \xi\vert)^{m_1-\vert \beta\vert}2^{j\vert \beta\vert}
\end{split} \end{equation*}
from a careful application of the Leibniz rule and where again we used \eqref{ineq1} to control the derivative of $a$ in the second variable. The above estimate is uniform in $x$. We use the crucial fact that $2^j\lesssim \vert\xi\vert$ hence for all multiindex $\vert\beta\vert\geqslant 1$, we have 
$$
(1+\vert \xi\vert)^{m_1-\vert \beta\vert}2^{j\vert \beta\vert}\lesssim (1+\vert \xi\vert)^{m_1-1}2^{j}
$$ 
which allows us to simplify the previous bound as 
$$
\big\vert \partial_\eta^\beta \left(a(x;\xi)-a(x;2^j\eta+\xi)\right) \big\vert \lesssim (1+\vert \xi\vert)^{m_1-1}2^{j}.
$$
Therefore, the decay we can get for the $H_\eta^{[\frac{d}{2}]+1}(\mathbb{R}^d)$ norm w.r.t. $\eta$ of $g_j$ would have the simple form:
\begin{eqnarray*}
\Vert g_j\Vert_{H^{[\frac{d}{2}]+1}_\eta(\mathbb{R}^d)}\lesssim 2^j(1+\vert \xi\vert)^{m_1-1}
\end{eqnarray*}
since the support of $\widetilde{\chi}$ in $\eta$ is compact.
Going back to our initial goal of bounding the symbol $c$, we get:
\begin{eqnarray*}
\vert c(x;\xi)\vert&\leqslant &\sum_{j:2^j\leqslant\frac{\vert\xi\vert}{5}} \Vert A_j\Vert_{L^1}\Vert B_j \Vert_{L^\infty}\lesssim (1+\vert \xi\vert)^{m_2+m_1}  \sum_{j:2^j\leqslant\frac{\vert\xi\vert}{5}}  2^{-j\alpha}(1+\vert\xi\vert)^{-1}2^j\\
&\lesssim &(1+\vert \xi\vert)^{m_1+m_2-1} \sum_{j:2^j\leqslant\frac{\vert\xi\vert}{5}}2^{j(1-\alpha)} \lesssim  (1+\vert \xi\vert)^{m_1+m_2-\vert\alpha\vert}
\end{eqnarray*}
since $\alpha\in (0,1)$.
Bounding the derivatives of $c$ in $\xi$ is similar and left to the reader. For the moment, we just proved our symbol $c$ belongs to the class $S^{m_1+m_2-\alpha}_{1,1}(\bbR^d)$. It remains to check that our symbol has the correct vanishing properties of pararegularized operators.

\smallskip

It remains to check the vanishing properties of the symbol $\widehat{c}(\eta_1,\xi)$ when the norms of $\xi$ and $\eta_1$ get close to each other. We start with the explicit formula
\begin{eqnarray*}
\widehat{c}(\eta_1,\xi)=\frac{1}{(2\pi)^d}\int_{\mathbb{R}^d} \big(\widehat{a}(\eta_1-\eta,\xi+\eta)-\widehat{a}(\eta_1-\eta,\xi) \big) \widehat{b}(\eta,\xi)\rmd\eta.
\end{eqnarray*}
Let us assume, without loss of generality, that one can take the constant $K$ from Definition \ref{def:classB} equal to $1/4$. Now observe that the integrand in 
$$
I_1 = \int_{\mathbb{R}^d} \big(\widehat{a}(\eta_1-\eta,\xi+\eta)\big) \widehat{b}(\eta,\xi)\rmd\eta
$$ 
is non-zero provided
$ \vert \eta_1-\eta\vert\geqslant \frac{1}{4}\vert \xi+\eta \vert $ (by the constraint from $a$) and $\vert \eta\vert\leqslant \frac{1}{4}\vert \xi\vert$ (by the constraint from $b$). In particular, 
$$
 \vert \eta_1\vert\leqslant |\eta_1-\eta|+|\eta|\leqslant\frac{1}{4}|\xi+\eta|+|\eta|\leqslant\frac{1}{4}|\xi|+\frac{5}{4}|\eta|\leqslant (1/4+5/16)|\xi|\leqslant\frac{3}{4}|\xi|\,.
 $$ 
The integrand 
$$
I_2=\int_{\mathbb{R}^d}-\widehat{a}(\eta_1-\eta,\xi) \widehat{b}(\eta,\xi)\rmd\eta=\int_{\mathbb{R}^d}-\widehat{a}(-\eta,\xi) \widehat{b}(\eta+\eta_1,\xi)\rmd\eta
$$ 
is non-zero provided $ \vert\eta+\eta_1\vert\geqslant \frac{1}{4}\vert \xi\vert $ and $\vert \eta\vert\leqslant \frac{1}{4}\vert \xi\vert$. In particular, 
$$ 
\vert \eta_1\vert\leqslant  |\eta_1+\eta|+|\eta|\leqslant\frac{1}{4}|\xi|+|\eta|\leqslant\frac{1}{2}|\xi|\leqslant\frac{3}{4}|\xi|\,.
$$
So $\widehat{c}$ is supported on $\vert \eta_1\vert\leqslant \frac{3}{4}\vert\xi\vert$.
\end{proof}

\medskip

\subsection{Commuting the heat operator with a paraproduct. \hspace{0.1cm}}
\label{SectionCommutationParamultiplication}

Recall $P=1-\Delta$ stands for the negative massive Laplace-Beltrami operator on functions on $M$. Our goal in this section is to control analytically some commutator $[e^{-tP},f\prec_i]$ of the heat operator $e^{-tP}$ with some paramultiplication operator $f\prec_i$. We use the tools from Section \ref{s:paradiffcommutators} developed on $\mathbb{R}^d$ to achieve our goal. This naturally leads to a continuity result on the commutator of a paraproduct operator with the resolvent operator of $P$.

\medskip

One key idea we need to control commutators of the form $[e^{-tP},P_u]$ is to think of the heat kernel $e^{-tP}$ as a parameter-dependent pseudo-differential operator of order $m\leqslant 0$, but $e^{-tP}$ grows like $t^{\frac{m}{2}}$ in $\Psi^m_{1,0}(M)$. We need to pay some price under the form of the exploding weight $t^{\frac{m}{2}}$ if we require more smoothing properties.

\smallskip

\begin{lemm}\label{l:heatpsido}
\textbf{(The heat kernel viewed as a parameter-dependent pseudo-differential operator)} Pick $m\leqslant 0$. Then $(t^{\frac{-m}{2}}e^{-tP})_{t\in [0,1]}$ is a bounded family of pseudo-differential operators in $\Psi^m_{1,0}(M)$.
\end{lemm}

\smallskip

\begin{proof}
The proof is obvious using the local representation of the heat kernel 
in charts that we shall use several times in the present work, we refer to Theorem~\ref{thm:heatcalculus} for a precise statement:
\begin{eqnarray*}
\kappa \circ K_t\circ\kappa^{-1}(x,y)=t^{-\frac{d}{2}}\widetilde{A}(t,\frac{x-y}{\sqrt{t}},x)
\end{eqnarray*}
and $\widetilde{A}\in C^\infty([0,+\infty)_{\frac{1}{2}}\times \mathbb{R}^d\times U)$ satisfies the estimate
\begin{eqnarray*}
\sup_{ (t,X,x)\in [0,a]\times \mathbb{R}^d \times U} \Big\vert \left( D^{\alpha}_{\sqrt{t},X,x}\widetilde{A}\right)(t,X,x) \Big\vert \leqslant C_{N,\alpha,\kappa(U)}\left(1+\Vert X\Vert\right)^{-N}.
\end{eqnarray*}
Then it suffices to Fourier transform $\widetilde{A}(t,X,x)$ in the middle variable $X$ to get  $\widehat{A}(t,\xi,x)$ which is Schwartz in the middle variable $\xi$ uniformly in $(t,x)$ in compact sets and
$ t^{-\frac{d}{2}}\widetilde{A}(t,\frac{x-y}{\sqrt{t}}, x)$ is the kernel of $Op\left( \widehat{A}(t,\sqrt{t}\xi,x) \right)=\widehat{A}(t,\sqrt{t}D,x) $.

Now it remains to check that 
$t^{-\frac{m}{2}}\widetilde{A}(t,\sqrt{t}\xi ,x)$ is a symbol in $S^m_{1,0}(U\times \mathbb{R}^d)$ uniformly in $t\in [0,1]$, $\vert \xi\vert\geqslant 1$:
\begin{eqnarray*}
\big\vert t^{-\frac{m}{2}}\widetilde{A}(t,\sqrt{t}\xi ,x)\big\vert \lesssim t^{-\frac{m}{2}}(1+ t^{\frac{1}{2}} \vert \xi\vert  )^{-N}\lesssim (1+\vert \xi\vert)^{-m}.
\end{eqnarray*}
Furthermore, for the derivatives one checks that
\begin{equation*} \begin{split}
\big\vert\partial_x^\alpha\partial_\xi^\beta t^{-\frac{m}{2}}\widehat{A}(t,\sqrt{t}\xi ,x) \big\vert &= \big\vert t^{\frac{-m+\vert \beta\vert}{2}} \left(\partial_x^\alpha\partial_\xi^\beta\widehat{A}\right)(t,\sqrt{t}\xi ,x)\big\vert   \\
&\lesssim t^{\frac{-m+\vert \beta\vert}{2}} (1+ t^{\frac{1}{2}} \vert \xi\vert)^{-N} \lesssim (1+\vert \xi\vert)^{-m-\vert\beta\vert}.
\end{split} \end{equation*}
We have controlled all the seminorms of $S^m_{1,0}$ uniformly in $t\in [0,1]$ in local charts and we are done.
\end{proof}

\smallskip

\begin{coro} \label{lemma2}
Let $\alpha\in\mathbb{R}$, $u\in \mathcal{C}^{\alpha}(M)$ and $t>0$. For all $\varepsilon>0$, we have:
$$
    \Vert e^{-tP}u\Vert_{\mathcal{C}^{\alpha+\varepsilon}(M)}\lesssim e^{-t}t^{-\frac{\varepsilon}{2}}\Vert u\Vert_{\mathcal{C}^\alpha(M)}\,.
$$
\end{coro}

\smallskip

\begin{proof}
The main difficulty in proving the statement is the small time $t\in (0,1]$ since we can deal with the large times using the spectral gap of the massive operator $P$. We know from Schauder estimates, Proposition \ref{l:simplifiedprooftwistedclass}, that a bounded family of elements $(A_u)_u\in \Psi^m(M)$ is bounded in $L\big(\mathcal{C}^{\alpha}(M), \mathcal{C}^{\alpha-m}(M)\big)$ (recall that $\mcC^\alpha=B^\alpha_{\infty,\infty}$). Since the family $(t^{\varepsilon/2}e^{-tP})_{0<t\leqslant1}$ is bounded in $\Psi^{-\varepsilon}_{1,0}(M)$, from Lemma \ref{l:heatpsido}, the conclusion follows by writing
\begin{eqnarray*}
\Vert e^{-tP}u\Vert_{\mathcal{C}^{\alpha+\varepsilon}}\lesssim  t^{-\frac{\varepsilon}{2}} \Vert t^{\frac{\varepsilon}{2}} e^{-tP}u\Vert_{L(\mathcal{C}^{\alpha}(M) , \mathcal{C}^{\alpha+\varepsilon}(M))}.
\end{eqnarray*}
\end{proof}

\smallskip

Next, we establish some manifold version of~\cite[Lemma 5.3.20]{JP}. For an arbitrary $i\in I$, denote by $\prec_i$ the localized paraproduct from Definition \ref{def:localpararesonant}.

\smallskip

\begin{lemm}\label{lemma3} 
Let $\alpha<1$, $\beta\in\mathbb{R}$, $u\in\mathcal{C}^{\alpha}(M), v\in\mathcal{C}^{\beta}(M)$ and $t>0$ be given. For all $\varepsilon>0$, we have
$$
\big\Vert  e^{-tP}(u\prec_i v)-u\prec_i( e^{-tP}v) \big\Vert_{\mathcal{C}^{\alpha+\beta+\varepsilon}(M)} \lesssim e^{-t} t^{-\varepsilon/2} \Vert u\Vert_{\mathcal{C}^\alpha(M)}  \Vert v\Vert_{\mathcal{C}^\beta(M)} 
$$
and
$$
\big\Vert  e^{-tP}(u\prec v)-u\prec( e^{-tP}v) \big\Vert_{\mathcal{C}^{\alpha+\beta+\varepsilon}(M)} \lesssim e^{-t} t^{-\varepsilon/2} \Vert u\Vert_{\mathcal{C}^\alpha}  \Vert v\Vert_{\mathcal{C}^\beta}.
$$
\end{lemm} 

\smallskip

\begin{proof}
The key conceptual idea is to think of the operator $P_u:v\mapsto u\prec v$ as a pararegularized operator where the threshold regularity is imposed by the H\"older 
regularity of $u$.

\smallskip

{\it 1.} Let us do the proof on $\mathbb{R}^d$ with some operator $H_t$ in the heat calculus in the sense of Theorem \ref{thm:heatcalculus}; it has the same analytic properties as the heat kernel and Lemma \ref{l:heatpsido} and Corollary \ref{lemma2} hold for $H_t$. Then we shall use charts and localization to make the proof global and on manifolds. Now note that $P_u\in B_\alpha^0(\bbR^d)$. We treat the heat operator $H_t$ as an element of $\Psi^{-\varepsilon}_{1,0}(\mathbb{R}^d)$ and we need to measure its growth in the Fréchet space $\Psi^{-\varepsilon}(\mathbb{R}^d)$ when $t\rightarrow 0^+$. In fact $t^{\frac{\varepsilon}{2}}H_t$ is bounded in $\Psi^{-\varepsilon}(\mathbb{R}^d)$ uniformly in $t\in [0,1]$ by Lemma \ref{l:heatpsido}. Therefore, Proposition \ref{lemm:commutator} on commutators shows that the commutator 
$$
[t^{\frac{\varepsilon}{2}}H_t,P_u]\in Op\left(\widetilde{S}^{-\alpha-\varepsilon}_{1,1}\right)
$$ 
is regularizing of order $\alpha+\varepsilon$, uniformly in $t\in [0,1]$. We thus have
$$
[H_t,P_u]=\mathcal{O}(t^{-\frac{\varepsilon}{2}})\in \widetilde{\Psi}^{-\alpha-\varepsilon}_{1,1}. 
$$
and 
$$
[H_t,P_u]v =\mathcal{O}(t^{-\frac{\varepsilon}{2}}) \in \mcC^{\alpha+\beta+\varepsilon}
$$ 
since $\widetilde{\Psi}^{-\alpha-\varepsilon}_{1,1}$ maps $\mcC^s(\bbR^d)$ into $\mcC^{s+\alpha+\varepsilon}(\bbR^d)$, by Proposition \ref{l:simplifiedprooftwistedclass}.

\smallskip

{\it 2.} The next step is to localize, then globalize, the proof on $\mathbb{R}^d$ to extend it to $M$. As usual, denote by $\left(\kappa_i,U_i,\chi_{i,1}\right)_{i\in I}$ the local charts plus the subordinated partition of unity.
We choose some functions $ \chi_{i,2}\in C^\infty_c(U_i),\, \chi_{i,2}=1$ on support of $\chi_{i,1}$ and $\psi_i\in C^\infty_c(\kappa_i(U_i))$, $\psi_i=1$ on support of $\chi_{i,2}$. 

Recall that
\begin{eqnarray*}
\left(u\prec_i v \right)=
 \kappa_i^{*}\big[ \psi_{i}\big(\kappa_{i*}\left(\chi_{i,1}u\right)\prec  \kappa_{i*}\left( \chi_{i,2}  v\right)\big)\big].
\end{eqnarray*}
The key idea to put the estimate on $M$ is to use commutator estimates plus cut-off functions to localize. We write carefully the first term we are studying
\begin{eqnarray*}
e^{-tP}\left(u\prec_i v \right)=
e^{-tP}\big( \kappa_i^{*}\big[ \psi_i\left(\kappa_{i*}\left(\chi_{i,1}u\right)\prec  \kappa_{i*}\left( \chi_{i,2}  v\right) \right)\big]\big)
\end{eqnarray*}
Choose $\chi_{i,3}\in C^\infty_c(U_i)$ which equals $1$ on support of $\kappa_i^*\psi_i$. One has
\begin{equation*} \begin{split}
e^{-tP}\big( \kappa_i^{*}\big[ \psi_i\big(\kappa_{i*}\left(\chi_{i,1}u\right)\prec &\kappa_{i*}\left( \chi_{i,2}  v\right) \big)\big]\big)   \\
&= e^{-tP}\chi^2_{i,3} \big( \kappa_i^{*}\big[ \psi_i\left(\kappa_{i*}\left(\chi_{i,1}u\right)\prec  \kappa_{i*}\left( \chi_{i,2}  v\right) \right)\big]\big)   \\
&= \chi_{i,3}e^{-tP}\chi_{i,3}\big( \kappa_i^{*}\big[ \psi_i\left(\kappa_{i*}\left(\chi_{i,1}u\right)\prec  \kappa_{i*}\left( \chi_{i,2}  v\right) \right)\big]\big) + \mathcal{O}_{\mathcal{C}^{1+\beta+\varepsilon}}(t^{-\frac{\varepsilon}{2}})   \\
&\sim \kappa_i^{*}\big( \kappa_{i*}  \left(\chi_{i,3}e^{-tP}\chi_{i,3}\right) \kappa_i^* \left[ \psi_i \left( \kappa_{i*}\left(\chi_iu\right)\prec  \kappa_{i*}\left( \chi_{i,2}  v\right)\right) \right] \big)\big)   \\
&= \kappa_i^{*}\big( \kappa_{i*}  \left(\chi_{i,3}e^{-tP} \right) \kappa_i^* \left[ \psi_i \left( \kappa_{i*}\left(\chi_iu\right)\prec  \kappa_{i*}\left( \chi_{i,2}  v\right)\right) \right] \big)\big).
\end{split} \end{equation*} 
We used the commutator estimate $t^{\frac{\varepsilon}{2}}[e^{-tP},\chi_{i,3}]\in \Psi^{-1-\varepsilon}$ uniformly in $t\in [0,1]$ ( $t^{\frac{\varepsilon}{2}}e^{-tP}\in \Psi^{-\varepsilon}_{1,0}$ and $\chi_{i,3}\in \Psi^0_{1,0} $) and the Schauder estimates for pseudo-differential operators. We also used the support property of $\chi_{i,3}$ to identify $\kappa_{i*}\chi_{i,3}\psi_i=\psi_i$. Now we also write in detail the term $u\prec_i (e^{-tP}v)$
\begin{eqnarray*}
u\prec_i (e^{-tP}v) &=& \kappa_i^{*}\left( \psi_{i}[\kappa_{i*}\left(\chi_{i,1}u\right)\prec  \kappa_{i*}\left( \chi_{i,2} e^{-tP} v\right)]\right)\\
& =&
 \kappa_i^{*}\left( \psi_{i} [\kappa_{i*}\left(\chi_{i,1}u\right)\prec  \kappa_{i*}\left( \chi_{i,2}\chi_{i,3} e^{-tP} v\right)]\right)\\
&\sim& \kappa_i^{*}\left( \psi_{i}[\kappa_{i*}\left(\chi_{i,1}u\right)\prec  \kappa_{i*}\left(\chi_{i,3} e^{-tP} \chi_{i,2} v\right)]\right)\\
&=&
 \kappa_i^{*}\left( \psi_{i}[\kappa_{i*}\left(\chi_{i,1}u\right)\prec    \left(\kappa_{i*}\chi_{i,3} e^{-tP} \kappa_i^*\right) \kappa_{i*}\left(\chi_{i,2} v\right)]\right)\\ 
&=&
 \kappa_i^{*}\left( \psi_{i}[\kappa_{i*}\left(\chi_{i,1}u\right)\prec    \left(\kappa_{i*}\chi_{i,3} e^{-tP} \kappa_i^*\right)\left(\psi_i \kappa_{i*}\left(\chi_{i,2} v\right)\right)]\right)
\end{eqnarray*} 
where we used again the commutator estimate and the support properties of all the cut-off functions. Now we recognize a commutator in $\mathbb{R}^d$ between an element in the heat calculus and some paradifferential operator
\begin{eqnarray*}
e^{-tP}\left(u\prec_i v\right)- u\prec_i (e^{-tP}v)\sim 
 \kappa_i^{*}\left( \psi_{i} [\left(\kappa_{i*}\chi_{i,3} e^{-tP} \kappa_i^*\right)\psi_i ,M_{\kappa_{i*}\left(\chi_{i,1}u\right)}]     \kappa_{i*}\left(\chi_{i,2} v\right)\right)
\end{eqnarray*} 
and, using the first part, we obtain
$$
\big\Vert e^{-tP}(u\prec_i v)-u\prec_i( e^{-tP}v) \big\Vert_{\mathcal{C}^{\alpha+\beta+\varepsilon}(M)} \lesssim e^{-t} t^{-\varepsilon/2} \Vert u\Vert_{\mathcal{C}^\alpha}  \Vert v\Vert_{\mathcal{C}^\beta(M)}
$$
as required. 
\end{proof}

\smallskip

We now prove an analogue of Lemma A3 in \cite{JP} describing the commutator of a paraproduct operator and the resolvent operator of $P$.
 
\smallskip
 
\begin{prop}\label{p:commutatorheat}
Let $0<\alpha<1$, $\beta\in\mathbb{R}$, $f\in C_T\mathcal{C}^\alpha(M)\cap C_T^{\frac{\alpha}{2}}L_\infty(M)$ and $g\in C_T\mathcal{C}^\beta(M)$. For all $0<\delta<2$, for all chart index $i$ we have
\begin{equation*} \begin{split}
\Big\Vert t\mapsto\int_0^{t}\big(e^{-(t-s)P}(f_s\prec_i g_s) - f_t\prec_i(e^{-(t-s)P}&g_s)\big)\rmd s \Big\Vert_{C_T\mathcal{C}^{\alpha+\beta+\delta}(M)}   \\
&\leqslant C \big(\Vert f\Vert_{C_T\mathcal{C}^\alpha(M)} +\Vert f\Vert_{C_T^{\frac{\alpha}{2}}L_\infty(M)}\big)\Vert g \Vert_{C_T\mathcal{C}^\beta(M)}
\end{split} \end{equation*}
for a positive constant $C$ independent of $T$. The same estimate holds for the global paraproduct by summing over $i\in I$.
\end{prop}

\smallskip

\begin{proof}
To prove this proposition, we rely on Corollary \ref{lemma2} and Lemma \ref{lemma3}. Let $\left(\kappa_i,U_i,\chi_{i,1}\right)_{i\in I}$ the local charts plus the subordinated partition of unity and  $ \chi_{i,2}\in C^\infty_c(U_i),\, \chi_{i,2}=1$ on support of $\chi_{i,1}$  and $\psi_i\in C^\infty_c(\kappa_i(U_i))$, $\psi_i=1$ on support of $\chi_{i,2}$. Recall that for fixed $i\in I$
\begin{eqnarray*}
\left(u\prec_i v \right)=
 \kappa_i^{*}\big[ \psi_{i}\big(\kappa_{i*}\left(\chi_{i,1}u\right)\prec  \kappa_{i*}\left( \chi_{i,2}  v\right)\big)\big].
\end{eqnarray*}
 
We start by rewriting
\begin{align*}
&\big\Vert  e^{-(t-s)P}(f_s\prec_i g_s)-f_t\prec_i(e^{-(t-s)P}g_s) \big\Vert_{\mathcal{C}^{\alpha+\beta+\delta}}\\
= &\big\Vert  e^{-(t-s)P}(f_s\prec_i g_s)-f_s\prec_i (e^{-(t-s)P}g_s) - (f_t-f_s)\prec_i (e^{-(t-s)P}g_s) \big\Vert_{\mathcal{C}^{\alpha+\beta+\delta}}\\
\leq&\underbrace{\big\Vert  e^{-(t-s)P}(f_s\prec_i g_s)-f_s\prec_i(e^{-(t-s)P}g_s) \big\Vert_{\mathcal{C}^{\alpha+\beta+\delta}}}_A + \underbrace{\big\Vert (f_t-f_s)\prec_i(e^{-(t-s)P}g_s) \big\Vert_{\mathcal{C}^{\alpha+\beta+\delta}}}_B\,.
\end{align*}
To bound the term A, let us use Lemma~\ref{lemma3} at time $t-s$ with $u=f_s$, $v=g_s$ and $\varepsilon=\delta$. This yields
\begin{align*}
    A\leqslant e^{-(t-s)} (t-s)^{-\frac{\delta}{2}}\Vert f_s\Vert_{\mathcal{C}^\alpha} \Vert g_s\Vert_{\mathcal{C}^\beta}\leqslant e^{-(t-s)} (t-s)^{-\frac{\delta}{2}}\Vert f\Vert_{C_T\mathcal{C}^\alpha} \Vert g\Vert_{C_T\mathcal{C}^\beta}\,.
\end{align*}
To bound the term B, let us first use a paraproduct estimate, that gives
\begin{align*}
    B\leqslant\Vert f_t-f_s \Vert_{L_\infty} \Vert (e^{-(t-s)P}g_s) \Vert_{\mathcal{C}^{\alpha+\beta+\delta}}\leqslant(t-s)^{\frac\alpha2}\Vert f\Vert_{C_T^{\frac\alpha2}L_\infty} \Vert (e^{-(t-s)P}g_s) \Vert_{\mathcal{C}^{\alpha+\beta+\delta}}\,.
\end{align*}
Then, let us use Corollary~\ref{lemma2} at times $t-s$ with $u=g_s$ and $\varepsilon=\alpha+\delta$. This yields
\begin{align*}
    B\leqslant e^{-(t-s)} (t-s)^{\frac\alpha2-\frac{\alpha+\delta}{2}}\Vert f\Vert_{C_T^{\frac\alpha2}L_\infty}\Vert g_s \Vert_{\mathcal{C}^\beta}\leqslant e^{-(t-s)}(t-s)^{-\frac{\delta}{2}}\Vert f\Vert_{C_T^{\frac\alpha2}L_\infty}\Vert g \Vert_{C_T\mathcal{C}^\beta}\,.
\end{align*}
We can now conclude since 
\begin{equation*} \begin{split}
\Big\Vert \int_0^t\big(e^{-(t-s)P}(f_s&\prec_i g_s) - f_t\prec_i (e^{-(t-s)P}g_s)\big)\rmd s \Big\Vert_{\mathcal{C}^{\alpha+\beta+\delta}}\\ \leqslant&\int_0^t \big\Vert  e^{-(t-s)P}(f_s\prec_i g_s)-f_t\prec_i(e^{-(t-s)P}g_s)  \big\Vert_{\mathcal{C}^{\alpha+\beta+\delta}}\rmd s   \\
   \leqslant&\int_0^t (A+B)\rmd s=\int_0^te^{-(t-s)}(t-s)^{-\frac\delta2}\rmd s\, \Big(\Vert f\Vert_{C_T\mathcal{C}^\alpha} +\Vert f\Vert_{C_T^{\frac{\alpha}{2}}L_\infty}\Big) \Vert g \Vert_{C_T\mathcal{C}^\beta}\,,
\end{split} \end{equation*}
and the integral over $s$ is convergent since $\delta<2$ by hypothesis. Moreover, the sup over $t\in [0,T]$ can then be bounded independently of $T$ thanks to the exponential decay of the massive Laplacian. 
\end{proof}

\smallskip

\subsection{Proof of a regularity estimate. \hspace{0.1cm}}
\label{SectionProofThmA2}

Recall the notations 
$$
\X_r \defeq \underline{\mathcal{L}}^{-1}(\xi_r), \quad\Xtwo_r \defeq :\X^2\hspace{-0.08cm}:_r,\quad \IXtwo_r \defeq \underline{\mathcal{L}}^{-1} (\Xtwo_r), \quad \IXthree_r \defeq \underline{\mathcal{L}}^{-1}( :\X^3\hspace{-0.08cm}:_r ).
$$ 
from the introduction. We prove in this section the following result.

\smallskip

\begin{thm} \label{thm_A_2_mfd}
Fix $\varepsilon>0$. As $r\searrow0$, we have that 
$v_{\textrm{ref},r}$ converges in probability in $C_T\mcC^{1-\varepsilon}(M)$, and that  
\begin{equation} \label{EqFormerThm1.1}
\nabla \IXtwo_r \cdot \nabla v_{\textrm{\emph{ref}},r} - b_r(e^{3\IXtwo_r}\IXthree_r)
\end{equation}
converges in probability in $C_T\mcC^{-2\varepsilon}(M)$.
\end{thm}

\smallskip

Theorem \ref{thm_A_2_mfd} is used in \cite{BDFT} to make sense of Jagannath \& Perkowski's formulation \eqref{EqJPFOrmulation} of the parabolic $\Phi^4$ equation \eqref{EqRenormalizedSPDE} in the limit where $r>0$ goes to $0$. We prove Theorem \ref{thm_A_2_mfd} in Section \ref{SubsectionConstructionScalarAnsatz} following the reasoning used by Jagannath \& Perkowski in their proof of Lemma A.2 in \cite{JP}. 

\subsubsection{Proof of Theorem \ref{thm_A_2_mfd}$\boldmath{.}$ \hspace{0.1cm}}
\label{SubsectionConstructionScalarAnsatz}

{\it 1. Control of the regularity of $v_{\textrm{\emph{ref}},r}$.} Recall we set
\begin{eqnarray*} 
f\cdot_1 (g\odot_2 h  - k) &\defeq& \sum_i \kappa_i^*\Big[ \psi_i  \kappa_{i*} \left(\chi_{i1} f \right) \Big(  \kappa_{i*}\left( \chi_{i2} g \right)    \odot\kappa_{i*}\left( \chi_{i3} h \right)- \kappa_{i*}\left( \chi_{i2} k \right)\Big)  \Big],
\end{eqnarray*}
and use a similar definition of the terms that appear in the right hand side of \eqref{EqDecomposition} below. We start from the triple product 
$$
3e^{3\IXtwo_r} \left(\IXthree_r \Xtwo_r - b_r(\X_r + \IXthree_r)\right)
$$ 
and we decompose it as a sum of trilinear operations  defined in  Section \ref{SubsectionTripleParamultiplication}
\begin{equation} \label{EqDecomposition} \begin{split}
3e^{3\IXtwo_r} \Big(\IXthree_r \Xtwo_r - b_r(\X_r &+ \IXthree_r )\Big)   \\
&= 3e^{3\IXtwo_r}\cdot_1   ( \IXthree_r \succ_2 \Xtwo_r ) + 3e^{3\IXtwo_r} \cdot_1 \big( \IXthree_r \odot_2 \Xtwo_r - 3e^{3\IXtwo_r} b\X_r\big)   \\
&\quad+ 3e^{3\IXtwo_r} \prec_1 ( \IXthree_r  \prec_2 \Xtwo_r) + \big[ 3e^{3\IXtwo_r} \odot_1 ( \IXthree_r \prec_2 \Xtwo_r) - 3e^{3\IXtwo_r} b \IXthree_r \big]   \\
&\quad+ 3e^{3\IXtwo_r} \succ_1 ( \IXthree_r \prec_2 \Xtwo_r).  
\end{split} \end{equation}
We know from Section 4.3 of \cite{BDFT} that $\IXtwo_r\in C_T\mcC^{1-\varepsilon}(M)$, uniformly in $r>0$ in any $L^p(\mathbb{P})$ space, and for every chart index $i$ and every $\psi_i\in C^\infty_c(\kappa_i(U_i))$ such that $\psi_i=1$ on the support of $\chi_{i3}$, we have
$$
\psi_i\left(\kappa_{i*} \left(\chi_{i2} \IXthree_r \right)\odot \kappa_{i*} \left(\chi_{i3} \Xtwo_r\right) - \kappa_{i*} \left(\chi_{i2}  b_r\X_r\right)\right) \in C_T\mcC^{-1/2-\varepsilon}(\kappa_i(U_i))
$$
uniformly in $r>0$ in any $L^p(\mathbb{P})$ space. The estimate is formulated on $\mathbb{R}^d$. Formulating the estimate on $M$ is equivalent to proving that
$$ 
\IXthree_r \odot_i \Xtwo_r - \chi_{i2}b_r\X_r\defeq\kappa_i^*\Big[ \psi_i\left(\kappa_{i*} \left(\chi_{i2} \IXthree_r \right)\odot \kappa_{i*} \left(\chi_{i3} \Xtwo_r\right) - \kappa_{i*} \left(\chi_{i2}  b_r\X_r\right)\right)\Big] \in C_T\mcC^{-1/2-\varepsilon}(U_i),
$$ 
a result that was proved in Section 4 of \cite{BDFT}. By the local-to-global principle, we deduce that $3e^{3\IXtwo_r} \cdot_1( \IXthree_r  \odot_2 \Xtwo_r - b_r\,\X_r)$ is well-defined and in $C_T\mcC^{-1/2-\varepsilon}(M)$, uniformly in $r\in[0,1]$, in $\mathbb{P}$-probability.  (We do not have stronger estimate in $L^p(\mathbb{P})$ spaces as the exponential term $e^{3\IXtwo_r} \in C_T\mcC^{1-\varepsilon}(M)$ is not known to be $\mathbb{P}$-integrable.)

Similarly $\IXthree_r \in C_T\mcC^{1/2-\varepsilon/2}(M)$ and  $\Xtwo_r\in C_T\mcC^{-1-\varepsilon/2}(M)$ the paraproduct estimates imply that $\IXthree_r \succ_{i,2} \Xtwo_r\in C_T\mcC^{-1/2-\varepsilon}(M)$ for some local paraproduct $\succ_{i,2}$ for every chart index $i$, hence  $3e^{3\IXtwo_r} \cdot_1  ( \IXthree_r  \succ_2 \Xtwo_r ) \in C_T\mcC^{-1/2-\varepsilon}(M) $ globally.  Again, it follows from the flat paraproduct estimates that $3e^{3\IXtwo_r} \succ_1 ( \IXthree_r \prec_2 \Xtwo_r) \in  C_T\mcC^{-1/2-\varepsilon}(M)$, with estimates that holds uniformly in $0<r\leqslant1$ in $\mathbb{P}$-probability.

\smallskip

The most complicated term is 
$$
3e^{3\IXtwo_r} \odot_1 ( \IXthree_r \prec_2 \Xtwo_r) - 3e^{3\IXtwo_r} b_r \IXthree_r.
$$ 
We use the identity \eqref{eq:use_Bony_comm}, with $f= \IXtwo_r, \alpha=1-\varepsilon, g=\IXthree_r , \beta= 1/2-\varepsilon, h=\Xtwo_r, \gamma= -1-\varepsilon$, to infer that 
\begin{eqnarray*}
3e^{3\IXtwo_r} \odot_1 ( \IXthree_r \prec_2 \Xtwo_r) - 3e^{3\IXtwo_r}  b_r \IXthree_r  = 9e^{\IXtwo_r} \cdot_1 \left(  \IXthree_r \cdot_2 (\Xtwo_r \odot_{3} \IXtwo_r - \frac{b_r}{3}) \ \right) +\mcC^{1/2-3\varepsilon}(M),
\end{eqnarray*}
It follows from \cite[Section 4.2]{BDFT} that for every chart index $i$ and every $\psi_i\in C^\infty_c(\kappa_i(U_i))$ such that $\psi_i=1$ on the support of $\chi_{i4}$, 
\begin{eqnarray*}
\psi_i \left( \kappa_{i*} \left(\chi_{i3} \Xtwo_r \right)\odot \kappa_{i*} \left(\chi_{i4} \IXtwo_r \right) - \kappa_{i*} \left(\chi_{i3}  \frac{b_r}{3} \right) \right) \in C_T\mcC^{-\varepsilon}(\bbR^d),
\end{eqnarray*}
hence $3e^{3\IXtwo_r} \odot_1 ( \IXthree_r \prec_2 \Xtwo_r) - 3e^{3\IXtwo_r}  b_r \IXthree_r \in C_T\mcC^{-1/2-\varepsilon}(M)$. The estimates are uniform in $0<r\leqslant 1$ in $\mathbb{P}$-probability.

{\it 2. Control of the gradient term.} We aim at controlling the regularity of 
$$
\nabla\IXtwo_r \cdot \nabla v_{\textrm{ref},r} - b_r(e^{3\IXtwo_r}\IXthree_r),
$$ 
where
$$
v_{\textrm{ref},r}\simeq\mathcal{L}^{-1}\left( 3e^{3\IXtwo_r}\Xtwo_r - b_r(\X_r+\IXthree_r ) \right).
$$ 
The proof is simple but the fact we are writing huge products of functions makes it look a bit combinatorial. First using the fact that the inverse heat operator $\mathcal{L}^{-1}$ is smoothing off-diagonal which implies that for any compactly supported distribution $U\in \mathcal{D}^\prime_c(U_i)$ and $\widetilde{\chi}_i\in C^\infty_c(U_i)$ such that $\widetilde{\chi}_i=1$ on support of $U$
\begin{eqnarray*}
\mathcal{L}^{-1}(U) - \widetilde{\chi}_i \, \mathcal{L}^{-1}(U) \in C^\infty(M)
\end{eqnarray*}
and using a trivialization of the tangent bundle $TM$ over each open chart $U_i$, we get the expression
\begin{equation*} \begin{split}
&\nabla\IXtwo_r \cdot \nabla v_{\textrm{ref},r} - b_r(e^{3\IXtwo_r}\IXthree_r ) =   \\
&\sum_i \kappa_i^*\bigg[\psi_i (\kappa_{i*}g)^{\mu\nu} \partial_\mu \kappa_{i*}\bigg(\widetilde{\chi}_i \mathcal{L}^{-1}\kappa_i^*   \\
&\hspace{2cm}\times \underbrace{\left( \kappa_{i*}(\chi_{i1}3e^{3\IXtwo_r}) \left[ \kappa_{i*}(\chi_{i2}\IXthree_r )\kappa_{i*}(\chi_{i3}\Xtwo_r) - b_r\kappa_{i*} \left( \chi_{i1}b(\X_r+\IXthree_r ) \right) \right] \right)} \bigg)   \\
&\hspace{3.5cm}\times\partial_\nu\kappa_{i*}(\chi_{i4}\IXtwo_r) \bigg]    \\
&- b_r (e^{3\IXtwo_r}\IXthree_r ) +\mcC^{-2\varepsilon}(M)   
\end{split} \end{equation*}
where the localization of the heat kernel avoided a double sum over the partition of unity indices. The error term in $\mcC^{-2\varepsilon}(M)$ comes from the irregularity of $\partial_\nu\kappa_{i*}(\chi_{i4}\IXtwo_r)$.
We need that $\chi_{i1}\ll\chi_{i2}\ll\chi_{i3}\ll\chi_{i4}\ll\widetilde{\chi}_i\ll\kappa_{i*}\psi_i$, where $\sum_i \chi_{i1}=1$ and $s_\mu(x)$ denotes the $\mu$ component of $s(x)\in T_xM$. The term underbraced was already defined in the first part and equals
\begin{eqnarray*}
\kappa_{i*}(\chi_{i1}3e^{3\IXtwo_r}) \prec\left( \kappa_{i*}(\chi_{i2}\IXthree_r ) \prec \kappa_{i*}(\chi_{i3}\Xtwo_r)\right) +\mcC^{-1/2-5\varepsilon}(M)
\end{eqnarray*}
where we singled out the most singular term which has regularity $\mcC^{-1-2\varepsilon}$; it is the term with the lowest regularity of the list of terms that contribute to the singularities of the scalar product. Applying $\mathcal{L}^{-1}$ to the  $\mcC^{-1/2-5\varepsilon}$ error term yields a term of regularity $\mcC^{\frac{3}{2}-5\varepsilon}$ and differentiating with respect to $\partial_\mu$ yields a term of regularity  $\mcC^{\frac{1}{2}-5\varepsilon}$ and multiplying with $\partial_\nu\kappa_{i*}(\chi_{i5}\IXtwo_r)\in \mcC^{-2\varepsilon}$ yields a well-defined term of regularity $\mcC^{-2\varepsilon}$. By the result of Proposition \ref{lemm:comm_paramultiplication_flat}, we rewrite the underbraced term as
\begin{eqnarray*}
\kappa_{i*}\big(\chi_{i1}3e^{3\IXtwo}\IXthree_r \big) \prec \kappa_{i*}(\chi_{i3}\Xtwo_r) +\mcC^{-1/2-5\varepsilon}(M).
\end{eqnarray*}
The next step is to commute the heat inverse and the paramultiplication operator 
$$
P_{\chi_{i1}3e^{3\IXtwo_r}\IXthree_r }(f) \defeq \chi_{i1}3e^{3\IXtwo_r}\IXthree_r \prec f.
$$ 
Since $\chi_{i1}3e^{3\IXtwo_r}\IXthree_r \in \mcC^{\frac{1}{2}-5\varepsilon}$, we have
\begin{equation*} \begin{split}
&\nabla\IXtwo_r \cdot \nabla v_{\textrm{ref},r} - b_r(e^{3\IXtwo_r}\IXthree_r )   \\
&= \sum_i \kappa_i^*\left[\psi_i (\kappa_{i*}g)^{\mu\nu}  \partial_\mu\kappa_{i*}\left(\kappa_i^* \underbrace{ \left( \kappa_{i*}(\chi_{i1}3e^{3\IXtwo_r}\IXthree_r )\right) \prec \kappa_{i*}(\widetilde{\chi}_i \mathcal{L}^{-1}\chi_{i3}\Xtwo_r) }  \right)  \partial_\nu\kappa_{i*}(\chi_{i4}\IXtwo_r) \right]   \\
&\quad- b_r(e^{3\IXtwo_r}\IXthree_r ) +\mcC^{-2\varepsilon}
\end{split} \end{equation*}
where we use the commutator estimate from Proposition \ref{p:commutatorheat}. Its use requires some information on the regularity of $3e^{3\IXtwo_r} \IXthree_r$. However, we proved in \cite[Section 4]{BDFT} that $\IXtwo_r\in \mcC^{1-2\varepsilon}([0,T]\times M), \IXthree_r \in \mcC^{\frac{1}{2}-3\varepsilon}([0,T]\times M)$, where the regularity is measured in space-time parabolic H\"older-Zygmund spaces, with estimates that are uniform in $0<r\leqslant1$ in $\mathbb{P}$-probability. We thus have $3e^{3\IXtwo_r} \IXthree_r \in \mcC^{\frac{1}{2}-3\varepsilon}([0,T]\times M)$, and this implies that
\begin{eqnarray*}
\kappa_{i*}\widetilde{\chi}_i \mathcal{L}^{-1}\kappa_i^*\left( \left( \kappa_{i*}(\chi_{i1}3e^{3\IXtwo_r}\IXthree_r )\right)\prec \kappa_{i*}(\chi_{i3}\Xtwo_r)  \right)= \left( \kappa_{i*}(\chi_{i1}3e^{3\IXtwo}\IXthree_r )\right) \prec \kappa_{i*}(\widetilde{\chi}_i \mathcal{L}^{-1}\chi_{i4}\Xtwo_r)   \\
+ \Big[ \kappa_{i*} \widetilde{\chi}_i \mathcal{L}^{-1}\kappa_i^*, \kappa_{i*}(\chi_{i1}3e^{3\IXtwo_r}\IXthree_r )\prec\Big]\left( \kappa_{i*}(\chi_{i3}\Xtwo_r)  \right)
\end{eqnarray*}
where the term with the commutator 
$$ 
\Big[\kappa_{i*} \widetilde{\chi}_i \mathcal{L}^{-1}\kappa_i^*, \kappa_{i*}(\chi_{i1}3e^{3\IXtwo_r}\IXthree_r )\prec\Big] \left( \kappa_{i*}(\chi_{i3}\Xtwo_r)  \right) 
$$ 
belongs to $C_T\mcC^{\frac{3}{2}-3\varepsilon}(M)$. So for the moment, the term we need to study simplifies as
\begin{equation*} \begin{split}
&\nabla\IXtwo_r \cdot \nabla v_{\textrm{ref},r} - b_r(e^{3\IXtwo_r}\IXthree_r )   \\
&= \sum_i \kappa_i^*\left[\psi_i (\kappa_{i*}g)^{\mu\nu}  \partial_\mu \kappa_{i*}\left(\kappa_i^* \widetilde{\psi}_i \underbrace{ \left( \kappa_{i*}(\chi_{i1}3e^{3\IXtwo_r}\IXthree_r )\right) \prec \kappa_{i*}(\widetilde{\chi}_i \mathcal{L}^{-1}\Xtwo_r) }  \right)  \partial_\nu\kappa_{i*}(\chi_{i4}\IXtwo_r) \right]   \\
&\quad- b(e^{3\IXtwo_r}\IXthree_r ) +\mcC^{-2\varepsilon}(M)
\end{split} \end{equation*}
where we inserted a cut-off function $\widetilde{\psi}_i$ and removed the $\chi_{i3}$ in front of $\Xtwo_r$ which does not affect regularities thanks to Lemma \ref{l:keylocalization} and the localization Lemma for the heat operator.

\smallskip

Our next goal will be to commute the partial derivative $\partial_\mu$ with the paramutiplication operator, we can already commute $\partial_\mu$ with $\widetilde{\psi}_i$ which yields a first-order differential operator $L$ with smooth coefficients and compactly supported in $\kappa_i(U_i)$. So everything boils down to studying the regularizing properties of some commutator on $\mathbb{R}^d$ of the form
$$
[L,P_U]  
$$ 
where 
$$
U=\kappa_{i*}(\chi_{i1}3e^{3\IXtwo}\IXthree ), \quad P_U = \kappa_{i*}(\chi_{i1}3e^{3\IXtwo}\IXthree )\prec
$$ 
is a paramultiplication operator on $\mathbb{R}^d$. By the results of Lemma \ref{lem:pararegclassicalsymbols}, the paramultiplication operator $M_{U}$ is an element in the class $B^0_{\frac{1}{2}-3\varepsilon}(\mathbb{R}^d)$ and so we have
$$
[L,P_U] = A + \Psi^{-\infty}
$$ 
where $A\in S^{1-(\frac{1}{2}-3\varepsilon)}_{1,1}(\bbR^d)$ is a pararegularized operator. This implies that 
$$
[L,M_{U}] \kappa_{i*}(\widetilde{\chi}_i \mathcal{L}^{-1}V) \in \mcC^{\frac{1}{2}-5\varepsilon}(\bbR^d)
$$ 
 which is under control. We are thus reduced to the study of  
\begin{equation*} \begin{split}
\sum_i \kappa_i^* &\left[\psi_i (\kappa_{i*}g)^{\mu\nu}  \kappa_{i*}\left(\kappa_i^* \widetilde{\psi}_i \underbrace{ \left( \kappa_{i*}(\chi_{i1}3e^{3\IXtwo}\IXthree_r )\right) \prec \partial_\mu \kappa_{i*}(\widetilde{\chi}_i \mathcal{L}^{-1}\Xtwo_r) }  \right)  \partial_\nu\kappa_{i*}(\chi_{i4}\IXtwo_r) \right]   \\ 
&\simeq \sum_i \kappa_i^*\left[\psi_i (\kappa_{i*}g)^{\mu\nu} \kappa_{i*}\left(\kappa_i^* \widetilde{\psi}_i \underbrace{ \left( \kappa_{i*}(\chi_{i1}3e^{3\IXtwo_r}\IXthree_r )\right) \prec \partial_\mu \kappa_{i*}(\widetilde{\chi}_i \IXtwo_r) }  \right) \odot \partial_\nu\kappa_{i*}(\chi_{i4}\IXtwo_r) \right],
\end{split} \end{equation*}
with equality up to an element in $\mcC^{-4\varepsilon}(M)$. We can use the Lemma \ref{lemThmCorrector} on Gubinelli, Imkeller \& Perkowski's corrector to control in $\mcC^{\frac{1}{2}-3\varepsilon-2\varepsilon-2\varepsilon}(\bbR^d)$ the difference
\begin{equation*} \begin{split}
\left( \left( \kappa_{i*}(\chi_{i1}3e^{3\IXtwo_r}\IXthree_r )\right) \prec \partial_\mu \kappa_{i*}(\widetilde{\chi}_i \IXtwo_r)   \right) &\odot \partial_\nu\kappa_{i*}(\chi_{i4}\IXtwo_r)   \\
&- \kappa_{i*}(\chi_{i1}3e^{3\IXtwo_r}\IXthree_r )\left( \partial_\mu \kappa_{i*}(\widetilde{\chi}_i\IXtwo_r)\odot \partial_\nu\kappa_{i*}(\chi_{i4}\IXtwo_r)\right) 
\end{split} \end{equation*}
So, finally, the expression we need to simplify reads
\begin{equation*} \begin{split}
&\nabla\IXtwo_r \cdot \nabla v_{\textrm{ref},r} - b_r(e^{3\IXtwo_r}\IXthree_r )   \\
&\simeq \sum_i \kappa_i^*\left[\psi_i (\kappa_{i*}g)^{\mu\nu} \kappa_{i*}(\chi_{i1}3e^{3\IXtwo_r}\IXthree_r ) \left(  \partial_\mu \kappa_{i*}(\widetilde{\chi}_{i} \IXtwo_r)   \odot \partial_\nu\kappa_{i*}(\chi_{i4}\IXtwo_r) \right)\right] - b_r(e^{3\IXtwo_r}\IXthree_r )
\end{split} \end{equation*}
up to a term in $\mcC^{-4\varepsilon}(M)$ -- where we could remove $\widetilde{\psi}_i$ again thanks to Lemma~\ref{l:keylocalization}. Now we can conclude thanks to the results of \cite[Section 4]{BDFT} which tell us that $\left\langle \nabla \IXtwo_r\odot_i \nabla\IXtwo_r\right\rangle$ has the same regularity as $\IXtwo\odot_i \Delta\IXtwo $ and must be renormalized with the same counterterm. More precisely, for every chart index $i$ the term
\begin{eqnarray*}
(\kappa_{i*}g)^{\mu\nu} \left(  \partial_\mu \kappa_{i*}(\widetilde{\chi}_{i} \IXtwo_r)   \odot \partial_\nu\kappa_{i*}(\chi_{i4}\IXtwo_r)\right) - \chi_{i4}b_r
\end{eqnarray*}
has a limit in $\mcC^{-4\varepsilon}(M)$ in $\mathbb{P}$-probability.
\smallskip


\section{Microlocal estimates on generalized propagators}
\label{sect_microlocal_est}

 In this second part of the work we study the fine properties of all the two-point functions appearing in the Feynman amplitudes from the stochastic estimates of our paper \cite{BDFT}. These functions are usually called {\sl propagators} in the physics literature. Recall $P=1-\Delta_g$  is the massive Laplacian and the fundamental operators we are concerned with are $e^{-tP}$, $\Delta e^{-tP}$ and $ e^{-\vert t-s\vert P}P^{-1}$. They have kernels with parabolic singularities that we shall describe precisely.  

\medskip

Before we enter the subject of that section we recall the following notations from \cite{BDFT}. Assume we are given a Riemannian manifold $\mathcal{X}$ and an open subset $U\subset \mathcal{X}$. In applications below $\mathcal{X}$ will be $M^k$ for some $k$ or a submanifold of that space, endowed with the induced Riemannian metric. For a closed conic set $\Gamma$ in $T^*U\backslash\{0\}$, we denote by $\mathcal{D}'_\Gamma(U)$ the space of distributions on $U$ whose wave front set is contained in $\Gamma$. This is a locally convex topological vector space endowed with a natural \emph{normal} topology associated with the seminorms
\begin{eqnarray*}
\Vert \Lambda\Vert_{N,V,\chi,\kappa} = \sup_{\xi\in V} \big\vert (1+\vert \xi\vert)^N\widehat{ (\kappa_*\Lambda)\chi}(\xi)\big\vert
\end{eqnarray*}
for all  chart $\kappa:\Omega\subset U\mapsto \mathbb{R}^{\dim(\mathcal{X})}$, integer $N$, $\chi\in C^\infty_c(\kappa(\Omega))$, cone $V\subset \mathbb{R}^{n*}$ such that
$$
\text{supp}(\chi)\times V\cap \kappa^{-1*}\Gamma=\emptyset, \text{ where } \kappa^{-1*}\Gamma = \big\{ (\kappa(x); (^t\rmd\kappa)^{-1}(\xi) ); (x;\xi)\in \Gamma\big\}.
$$
And we also need the seminorms of the strong topology of distributions
\begin{eqnarray*}
\sup_{\chi\in B}\vert\langle \Lambda,\chi\rangle\vert
\end{eqnarray*}
where $B$ is a bounded set of $C^\infty_c(\mathcal{X})$ which means that there is some compact $C$ such that $\text{supp}(B)\subset C$ and for any differential operator $P$, $\sup_{\chi\in B}\Vert P\chi\Vert_{L^\infty(C)}<\infty$. To be bounded in $\mathcal{D}^\prime_\Gamma(U)$ will always mean that all the above seminorms are bounded. Recall from Section 3.1 of \cite{BDFT} the notion of (parabolic) scaling field $\rho$ for a submanifold $\mathcal{Y}\subset\mathcal{X}$. We assume that $(e^{-s\rho})^*\Gamma\subset\Gamma$, for all $s\geqslant 0$. Denote by $K_t^{\mathcal{X}}$ the heat kernel of $\mathcal{X}$. 

\smallskip

\begin{defi*}
For $\alpha<0$ and $a\in\bbR$ we define the space $\mathcal{S}^{\alpha,(a,\rho)}_\Gamma(U)$ of distributions $\Lambda\in\mathcal{D}'(U)$ with the following property. For all pseudo-differential operators $Q$ with Schwartz kernel compactly supported in $U\times U$ and whose symbol vanishes on $\Gamma$, for each compact set $C\subset U$, there is a finite positive constant $m_{C,Q}$ such that
$$
\underset{s\geqslant 1}{\sup}\,\underset{x\in C}{\sup}\,\underset{0<t\leqslant1}{\sup}\, e^{as} t^{-\alpha/2} \big\vert \big\langle (e^{-s\rho})^*\Lambda , (I+Q)K_t^{\mathcal{X}}(x,\cdot) \big\rangle\big\vert \leqslant m_{C,Q} < \infty.
$$
\end{defi*}

\smallskip

We define $\mathcal{S}^a_\Gamma(U)$  as the union over $\alpha$ of all the spaces $\mathcal{S}^{\alpha,(a,\rho)}_\Gamma(U)$, for $a\in\bbR$ fixed and $\rho$ a scaling field for the inclusion $\mathcal{Y}\subset\mathcal{X}$ whose backward semiflows leave $\Gamma$ fixed. The letter `$\mathcal{S}$' is chosen for {\sl scaling}. The exponent $a$ retains the scaling property and $\Gamma$ information on the wavefront set. Note that the space $\mathcal{S}^a_\Gamma(U)$ is a priori larger than conormal distributions with wavefront set in $N^*\left(\mathcal{Y}\subset U \right)$ since elements in $\mathcal{S}^a_\Gamma(U)$ might have some wavefront set contained in the cone $\Gamma$ which is not necessarily included in $N^*\left(\mathcal{Y}\subset U \right)$.

\medskip

\subsection{Parabolic kernels and the class $\Psi_P^a$$\boldmath{.}$\hspace{0.1cm}}
\label{ss:parabolickernels}

$\;$ We define in this section the parabolic calculus which describes the singularities of the propagators that appear in the analysis of the dynamical $\Phi^4$ equation on a $3$-dimensional closed manifold.

\subsubsection{The elements in $\Psi_P(\mathbb{R}\times \mathbb{R}^d)$$\boldmath{.}$\hspace{0.1cm}}

Following a practical approach in microlocal analysis, we start by defining the parabolic calculus on $\mathbb{R}\times \mathbb{R}^d$. Then we prove a change of variables formula for the parabolic operators. A way to give an intrinsic definition of the class of parabolic operators is to work in position space and test the growth of the Schwartz kernel near diagonals by vector fields in the module of vector fields tangent to the diagonal. This is inspired by results of Beals~\cite{Beals}, Bony~\cite{Bony2,Bony3}, H\"ormander~\cite[p100-104]{Horm3}, Joshi~\cite{JoshiPAMS,Joshi}, Melrose--Ritter~\cite{MelRitt}, and also Taylor~\cite[Prop 2.2 p.~6]{Taylor2} that define pseudo-differential kernels by their diagonal behaviour and under testing with vector fields. Let $\mathcal{M}$ be the $C^\infty(\mathbb{R}_t\times \mathbb{R}_x^d\times \mathbb{R}_h^d)$--module of vector fields tangent to $\{t=0,h=0\}$; it is generated by the vector fields
\begin{eqnarray*}
 \quad \partial_{x^i}\,,\quad t\partial_t\,, \quad t\partial_{h^i} \,, \quad  h^i\partial_t\, ,\quad h^i\partial_{h^k}\,,\qquad (1\leqslant (i,k)\leqslant d).
\end{eqnarray*}

\smallskip

\begin{defi} [$\Psi_P(\mathbb{R}\times \mathbb{R}^d)$] \label{defi:flatparabolic}
Pick a negative real number $a$. An operator $A: C^\infty(\mathbb{R}\times \mathbb{R}^d)\mapsto \mathcal{D}^\prime(\mathbb{R}\times \mathbb{R}^d)$ belongs to $ \Psi_P^a(\mathbb{R}\times \mathbb{R}^d)$ if its Schwartz kernel $K\in \mathcal{D}^\prime( \mathbb{R}^{1+d}\times \mathbb{R}^{1+d} )$ satisfies the following properties.
\begin{enumerate}
	\item There is a function $A\in C^\infty([0,+\infty)\times \left(\mathbb{R}^d\setminus\{0\}\right)\times \mathbb{R}^d)$ such that either $K(t,s,x,y)=A(\vert t-s\vert,x-y,x)$ or $K(t,s,x,y)= {\bf 1}_{[0,\infty)}(t-s) A( t-s,x-y,x)$.   \vspace{0.1cm}
	
	\item There is $R>0$ such that for all $(t,h)$ with $t > 0$ and $\vert  t\vert+\Vert h\Vert^2\leqslant R$, for all vector fields $L_1,\dots,L_k\in \mathcal{M}$
\begin{eqnarray*}
\sup_{x\in \mathbb{R}^d} \big\vert \left(L_1\dots L_k A\right)(t,h,x) \big\vert \lesssim_{L_1,\dots,L_k} \left\{\begin{array}{ll} \left( \vert t\vert+\Vert h\Vert^2\right)^{-\frac{2+2a+d}{2}} & \textrm{if } d+2+2a > 0,   \\
\vert\hspace{-0.05cm}\log \left( \vert t\vert+\Vert h\Vert^2\right) \hspace{-0.05cm}\vert & \textrm{if } d+2+2a = 0.
\end{array} \right.
\end{eqnarray*}

	\item There exists $\delta>0$ such that for all $t\geqslant \frac{R}{2}$, we have the decay estimate
\begin{eqnarray*}
\sup_{x\in \mathbb{R}^d} \big\vert\hspace{-0.07cm} \left(\partial_{t,x,h}^\alpha A\right)\hspace{-0.05cm}(t,h,x) \big\vert \leqslant C_{\alpha}e^{-\delta t}.
\end{eqnarray*}
\end{enumerate} 
\end{defi}

\smallskip

If in the above definition, we add the extra assumption that $A$ is compactly supported in the last variable $x$, then we get a proper operator. This might be necessary to compose elements in the parabolic calculus. However, since we only work on closed manifolds, we can take without loss of generality proper kernels as models for the parabolic kernels on manifolds.  The following example illustrates the potential difficulties of a Fourier transform approach to parabolic calculus.

\smallskip

\begin{ex}
We work on flat space $\mathbb{R}^d$. 
The inverse heat operator
$\mathcal{L}^{-1}$ is a  Fourier multiplier by $\big(i\tau+1+\vert \xi\vert^2\big)^{-1}$ which is given by the well-defined symbol if we are allowing parabolic scalings. However the operator 
$$ 
\varphi\mapsto \int_{-\infty}^t e^{-\vert t-s\vert P}P^{-1}\varphi(s,\cdot)\rmd s
$$ 
is a Fourier multiplier by
$$ 
\big(i\tau+1+\vert \xi\vert^2\big)^{-1} (1+\vert \xi\vert^2)^{-1} 
$$
which is \textbf{not a smooth symbol in the usual H\"ormander classes} even viewed as a parabolic symbol. The problem of $\big(i\tau+1+\vert \xi\vert^2\big)^{-1} (1+\vert \xi\vert^2)^{-1} $ lies in the order of the symbol. One considers the parabolic compactification of $\mathbb{R}^{1+d}$ by the parabolic sphere at infinity then the parabolic order has a jump on the $\xi=0$ hypersurface at the parabolic sphere at infinity. 
\end{ex}

\smallskip

It is elementary to note the following facts. First, for any kernel $K$ as in Definition \ref{defi:flatparabolic}, for any function $\chi\in C^\infty(\mathbb{R}^d\times \mathbb{R}^d)$, the product $\chi K$ also satisfies Definition~\ref{defi:flatparabolic}. Second, the module $\mathcal{M}$ is the Lie algebra of vector fields tangent to the submanifold $\{t=0,x=y\}$ in $[0,+\infty)\times \mathbb{R}^d\times \mathbb{R}^d$. This forms a finitely generated  module over $C^\infty([0,+\infty)\times \mathbb{R}^d\times \mathbb{R}^d)$ and can be interpreted as the smooth derivation of the algebra $C^\infty\big([0,+\infty)\times \mathbb{R}^d\times \mathbb{R}^d\big)$ leaving fixed the ideal 
$$
\mathcal{I} = t\,C^\infty([0,+\infty)\times \mathbb{R}^d\times \mathbb{R}^d) + \sum_{i=1}^d (x^i-y^i) \, C^\infty\big([0,+\infty)\times \mathbb{R}^d\times \mathbb{R}^d\big)
$$ 
of smooth functions vanishing over $\{t=0,x=y\}$. More generally, given some submanifold $Y\subset X$ in some ambient manifold $X$, let $\mathcal{I}_Y$ denote the ideal of smooth functions vanishing over $Y$. Then the module $\mathcal{M}_Y\subset C^\infty(TX)$ of vector fields tangent to $Y$ is defined as
\begin{eqnarray*}
L\in \mathcal{M}_Y \text{ if }\big(\forall f\in \mathcal{I}_Y, Lf\in \mathcal{I}_Y\big),
\end{eqnarray*}
where $L\in C^\infty(TX)$ acts as a Lie derivative. (We refer to H\"ormander's treatment \cite[Lemma 18.2.5 p.~100]{Horm3} for a careful definition in coordinates. Now we need to verify some form of diffeomorphism invariance.) 

\smallskip

Let $\Phi:\mathbb{R}^d\mapsto \mathbb{R}^d$ denote a diffeomorphism which is the identity outside some compact subset. The lifted diffeomorphism 
$$
\widetilde{\Phi}:(t,x,y)\in \mathbb{R}\times \mathbb{R}^d\times \mathbb{R}^d\mapsto \big(t,\Phi(x),\Phi(y)\big)\in \mathbb{R}\times \mathbb{R}^d\times \mathbb{R}^d
$$ 
leaves the submanifold $\{t=0,x=y\}$ invariant. Hence the ideal $\mathcal{I}$ of smooth functions vanishing over $\{t=0,x=y\}$ is invariant by pull-back: $\widetilde{\Phi}^*\mathcal{I}=\mathcal{I}$, and for all vector field $L\in \mathcal{M}$ one has $\widetilde{\Phi}_*L\in \mathcal{M}$. As a consequence, the elements in $\Psi_P^a(\mathbb{R}\times \mathbb{R}^d)$ enjoy the following invariance property: for any pair of test functions $\chi_1,\chi_2\in C^\infty_c(\mathbb{R}^d)$, for any kernel $K$ satisfying Definition \ref{defi:flatparabolic} the new kernel
\begin{eqnarray*}
K\big(t,s,\Phi(x),\Phi(y)\big) \chi_1(x) \chi_2(y)
\end{eqnarray*}
also satisfies Definition \ref{defi:flatparabolic}. This invariance property immediately allows us to globalize Definition \ref{defi:flatparabolic} to the setting of $\mathbb{R}\times M$, where $M$ is a smooth closed manifold. In the sequel, we work on $\mathbb{R}^2\times M^2$, since operators such as $e^{-\vert t-s\vert P} P^{-1}$ depend on two time variables. We can give an intrinsic definition of parabolic kernels on $\bbR^\times M$ as follows. We denote by 
$$
\mathcal{M}_P\subset C^\infty\big(T(\mathbb{R}\times M^2)\big)
$$ 
the tangent Lie algebra of the submanifold $\{0\}\times \text{Diag} \subset \mathbb{R}\times M^2$. 

\smallskip

\begin{defi} [Kernels with parabolic singularities] \label{def:paraboliccalculus}
Pick a negative real number $a$. Elements of the parabolic calculus in $\Psi_P^a(\mathbb{R}^2\times M^2)$ are operators whose  kernels $K\in C^\infty(\mathbb{R}^2\times M^2\setminus \text{Space time diagonal })$ for which there exists a function $A\in C^\infty\left( (0,\infty)\times \left(M^2\setminus \text{Space diagonal}\right)\right)$ such that one has either  
$$
K(t,s,x,y) = A\big(\vert t-s\vert,x,y\big)
$$ 
or 
$$
K(t,s,x,y)={\bf 1}_{[0,\infty)}(t-s) A( t-s,x,y)
$$
and the following property hold. There exists $R>0$ such that for all $0<t<R$ and all $L_1,\dots,L_k\in \mathcal{M}_P$ one has
\begin{eqnarray} \label{ineq:Aderivatives}
\big\vert \left(L_1\dots L_k A\right)(t,x,y) \big\vert \lesssim_{L_1,\dots,L_k} \left\{ \begin{array}{ll} 
\left( \vert t\vert+d(x,y)^2\right)^{-\frac{2+2a+d}{2}} & \textrm{ if } 2+d+2a > 0,   \\
\big\vert\hspace{-0.05cm}\log\left( \vert t\vert+d(x,y)^2\right)\hspace{-0.07cm}\big\vert & \textrm{ if } 2+d+2a = 0.
\end{array} \right.
\end{eqnarray}
\end{defi}

\smallskip

We would like to clarify two things in the above definition. First, the fact that the function $A$ is smooth up to $\{0\}\times \left(M^2\setminus \text{Space diagonal}\right)$ is important for our application to stochastic estimates. This implies, for instance, that when space points $x\neq y$ are distinct, the kernels $K$ in $\Psi_P^a$ are smooth on the manifold with boundary $t\geqslant s$ and also when $t\leqslant a$. It allowed us to consider only smoothing of the white noise $\xi$ in \cite{BDFT} in space rather than in space and time. Second, if we want to put some locally convex topology on $\Psi_P^a(\mathbb{R}^2\times M^2)$ then this topology is the weakest topology defined from the seminorms of $C^\infty\left( [0,+\infty)\times \left(M^2\setminus \text{Space diagonal}\right)\right)$ and by taking the best constants in the estimate \eqref{ineq:Aderivatives}. Third, Definition \ref{def:paraboliccalculus} is intrinsic, and we see the central role played by the parabolic distance $\left( \vert t\vert+\text{dist}(x,y)^2\right)^{\frac{1}{2}}$. The kernels in $\Psi_P^a$ are stable by differentiation by elements of the module $\mathcal{M}$. 

\smallskip

A concrete version of the above estimates. Near every $p$, there is an open subset $U$ near $p$ and a coordinate system $(x^i)_{i=1}^d$ on $U$ in which the generator of parabolic scaling is the weighted Euler vector field $\rho=2t\partial_t+\sum_{i=1}^d(x^i-y^i)\partial_{x^i}$ and in which we identify the Lie algebra of vector fields tangent to $\{0\}\times \text{Diag}$ as the $C^\infty$-module generated by
\begin{eqnarray*}
t\partial_t, \quad t(\partial_{x^i}+\partial_{y^i}) , \quad  (x^i-y^i)\partial_t ,\partial_{x^i}+\partial_{y^i},  \quad (x^i-y^i)\partial_{x^k},(x^i-y^i)\partial_{y^k} \qquad (1\leqslant i,k\leqslant d).
\end{eqnarray*}
In these local coordinates, the kernel $K$ can be represented non-uniquely as $K(t,s,x,y)=A(\vert t-s \vert,x,x-y)$ or $K(t,s,x,y)={\bf 1}_{[0,\infty)}(t-s) A(t-s,x,x-y)$, where $A(t,x,X)\in C^\infty([0,+\infty)\times U\times \mathbb{R}^d)$ satisfies the estimates
\begin{eqnarray*}
\vert L_1\dots L_k A  \vert \lesssim_{L_1,\dots,L_k} \left( \vert t\vert+\vert X\vert^2\right)^{-\frac{2+2a+d}{2}}
\end{eqnarray*}
uniformly in $t$ in some compact set, any of the vector fields $L_1,\dots,L_k$ generating the tangent Lie algebra.

\smallskip

From the above definition, it is immediate that the Schwartz kernels of elements in $\Psi_P$ of the form
\begin{equation} \label{eq:formkernels}
A\big(\vert t-s\vert,x,y\big), A\big(t-s,x,y\big) {\bf 1}_{[0,\infty)}(t-s) ,
\end{equation}
have conormal singularities in the union of conormals 
$$
N^*\left(\{s= t\}\subset \mathbb{R}^2\times M^2\right)\cup N^*\left(\{s=t,x=y\} \subset \mathbb{R}^2\times M^2\right).
$$ 
Namely for any kernel $K\in \Psi_P^a$, we have the wave front set bound:
$$
WF(K)\subset N^*\left(\{s= t\}\subset \mathbb{R}^2\times M^2\right)\cup N^*\left(\{s=t,x=y\} \subset \mathbb{R}^2\times M^2\right) .
$$
In all applications, we use parabolic kernels of the above forms described by equations~\ref{eq:formkernels} where we allow for a discontinuity in the time variables but this discontinuity is controlled.

\smallskip

We next give a reformulation of the estimate \eqref{ineq:Aderivatives} appearing in the definition of kernel elements in $\Psi_P^a$. This examines how the kernel grows when we differentiate at arbitrarily high order and we do not necessarily differentiate in the tangent direction to the diagonal but in all directions.

\smallskip

\begin{lemm}\label{l:gainderivative}
We use the notations of Definition \ref{def:paraboliccalculus} and consider some parabolic kernel $K\in \Psi_P^a$. Every point $p\in M$ has a neighbourhood $U$ and a coordinate system $(x^i)_{i=1}^d$ on $U$ such that the kernel $K$ can be represented either as $
K(t,s,x,y) = A\big(\vert t-s\vert,x,y\big)
$ 
or $K(t,s,x,y)={\bf 1}_{[0,\infty)}(t-s) A(t-s,x,x-y)$, where the function $A(t,x,h)\in C^\infty((0,+\infty)\times U\times \mathbb{R}^d)$ satisfies the estimates
\begin{eqnarray*}
\big\vert \partial_{\sqrt{t}}^{\alpha_1}\partial_h^{\alpha_2}\partial_x^\beta A\big\vert \leqslant C \left(  \sqrt t+\Vert h\Vert\right)^{-{2-2a-d-\vert\alpha_1\vert-\vert\alpha_2\vert}{}}.
\end{eqnarray*}
\end{lemm}

\smallskip

\begin{proof}
In the proof, we write $\| (t,h)\|\defeq \sqrt{t}+\|h\|$. In $\mathbb{R}_t\times U_x\times \mathbb{R}_h^d$,
cover the complement of $ \{h=0,t=0\} $ by conic sets of the form $V^ i= \big\{\vert h^i\vert\geqslant \frac{1}{2+2d} \Vert (t,h)\Vert \big\}$, $V^0 = \big\{ \vert \sqrt{t}\vert\geqslant  \frac{1}{2+2d} \Vert (t,h) \Vert\big\}$, $1\leqslant i\leqslant d$~\footnote{it looks like covers of projective spaces in some sense}. We reduce the proof 
to $P=\partial_{h^i}$, the more general case is similar.
For any $(t,x,h)$ in one of these conical sets say $V^j$, we 
note that
\begin{equation*} \begin{split}
\vert \partial_{h^i}A\vert &\leqslant  \frac{1}{\vert h^j\vert} \vert h^j \partial_{h^i}A\vert\leqslant  (2+2d)\Vert h\Vert^{-1} \vert h^j \partial_{h^i}A\vert   \\
&\leqslant C (2d+2) \left( t+\Vert h \Vert^2\right)^{-\frac{d+2+2a}{2}}\Vert h\Vert^{-1}   \\
&\lesssim \left( t+\Vert h \Vert^2\right)^{-\frac{d+2+2a+1}{2}}, 
\end{split} \end{equation*}
where we used the crucial fact that $ h^j \partial_{h^i}\in \mathcal{M}$ is a vector field vanishing on $ \{h=0,t=0\}$, which allowed to apply the estimate \eqref{ineq:Aderivatives} to the tangent vector field $ (x^j-y^j) \partial_{x^i}$.
\end{proof}

\smallskip

\subsubsection{Singularities of $e^{-tP}, {\Delta e^{-tP}}, e^{-tP}P^{-1}$$\boldmath{.}$\hspace{0.1cm}}
\label{p:parabolicbounds}

The goal of the present subsection is to study in detail the singularities of above operators from the point of view of the parabolic calculus.

\medskip

\textsf{\textbf{ (a) Heat singularity$\boldmath{.}$\hspace{0.1cm}}} We start by describing precisely the parabolic singularities of $e^{-tP}$. We begin by recalling a statement which can be found under different but closely related forms, in the works of Melrose \cite{MelroseAPS}, Grieser \cite{Grieser} and Taylor \cite{Taylor2}. We denote below by $C^\infty([0,+\infty)_{\frac{1}{2}})$ the space of smooth functions of $\sqrt{t}$.

\smallskip

\begin{thm}\label{thm:heatcalculus}
We first describe 
kernels of the heat calculus $\Psi_H^\alpha(M)$ for some real number $\alpha\leqslant 0 $. 
\begin{itemize}
	\item The kernel $K$ is smooth in $(0,+\infty)\times M^2$.   \vspace{0.1cm}
	
	\item We have the following off-diagonal quantitative bounds: for any differential operator $P_{\sqrt{t},x,y}$ with smooth coefficients in $x$, $y$ and $\sqrt{t}$, and for all $N>0$, 
\begin{eqnarray*}
\big\Vert P_{\sqrt{t},x,y}K\big\Vert_{L^\infty(M\times M)}\leqslant C_{U,N,\alpha}\left(1+\frac{d(x,y)}{\sqrt{t}}\right)^{-N}.
\end{eqnarray*}

	\item For any $p$, for any open set $U$ endowed with a coordinate system near $p$, there is an element $\widetilde{A}\in C^\infty([0,+\infty)_{\frac{1}{2}}\times \mathbb{R}^d\times U)$ such that
\begin{eqnarray*}
K_t(x,y)=t^{-\frac{d}{2}-1-\alpha}\widetilde{A}\Big(t,\frac{x-y}{\sqrt{t}},x\Big)
\end{eqnarray*}
and $\widetilde{A}\in C^\infty([0,+\infty)_{\frac{1}{2}}\times \mathbb{R}^d\times U)$ satisfies the estimate
\begin{eqnarray*}
\big\Vert D^{\alpha}_{\sqrt{t},X,x}\widetilde{A} \big\Vert_{L^\infty([0,a]\times \mathbb{R}^d \times U)} \leqslant C_{N,\alpha,\kappa(U)}\left(1+\Vert X\Vert\right)^{-N}.
\end{eqnarray*}
The above bound holds only true for some compact time interval of the form $[0,a]$ for $0<a<+\infty$.   \vspace{0.1cm}

	\item All previous bounds hold true with an exponential factor $e^{-t\delta}$ for all $\delta \in [0,1)$ which shows exponential decay in time $t$. 
\end{itemize}

Let $K_t$ denotes the massive heat kernel $e^{-tP}$, then the kernel $K_\bullet$ on $ (0,+\infty)\times M^2$ satisfies the above properties with $\alpha=-1$.
\end{thm}

\smallskip

Moreover, it is proved in the work~\cite{Grieser} that the above heat calculus $\Psi_H^\bullet(M)$ has suitable invariance properties and behaves well under composition. This is used to prove the existence of the heat parametrix.  
From the previous representation, we immediately deduce the parabolic singularity of the heat kernel.

\smallskip

\begin{lemm}[Heat kernel with parabolic singularity]\label{l:heatvsparabolic}
Under the notations of the previous Theorem, we have for every $p$ and any coordinate system defined in $U$ near $p$
\begin{eqnarray*}
\vert e^{-tP}(x,y)\vert\leqslant C\big(\sqrt{t}+\vert x-y \vert \big)^{-d}.
\end{eqnarray*}
Moreover, the above estimate still holds with the same exponent on the right-hand side for the heat-kernel differentiated with vector fields tangent to the diagonal. This implies that $e^{-tP}$ defines an element in $\Psi_P^{-1}$ of the parabolic calculus.   
\end{lemm} 

\smallskip

\begin{proof}
Recall that
$$
e^{-tP}\equiv K_t(x,y) = t^{-\frac{d}{2}}\widetilde{A}\Big(t,\frac{x-y}{\sqrt{t}},x\Big)
$$
in the chart of Theorem \ref{thm:heatcalculus}. As a corollary of the bounds on $\widetilde{A}$ on the middle variable $X$, we find that
\begin{eqnarray*}
\Big\vert t^{-\frac{d}{2}} \widetilde{A}\Big(t,\frac{x-y}{\sqrt{t}},x\Big) \Big\vert \leqslant C_{d}t^{-\frac{d}{2}}\left(1+\frac{\vert x-y \vert}{\sqrt{t}} \right)^{-d}\lesssim\left(\sqrt{t}+\vert x-y\vert \right)^{-d}.
\end{eqnarray*}

This bound immediately shows that the heat kernel has a parabolic singularity. However, we need a bit more to prove that it is an element of the parabolic calculus.
Choose a tangent vector field of the form $\partial_x+\partial_y$ or 
of the 
form $M(x-y).\partial_x\defeq M^i_j (x^j-y^j)\partial_{x^i}$. Then observe that:
\begin{eqnarray*}
(\partial_x+\partial_y)\left( t^{-\frac{d}{2}}\widetilde{A}\Big(t,\frac{x-y}{\sqrt{t}},x\Big)\right) = t^{-\frac{d}{2}}(\partial_y\widetilde{A})\Big(t,\frac{x-y}{\sqrt{t}},x\Big) \lesssim \left(\sqrt{t}+\vert x-y\vert \right)^{-d}
\end{eqnarray*}
since $(\partial_y\widetilde{A}) $ satisfies the same estimates as $\widetilde{A}$. We also get that
\begin{eqnarray*}
\big(M(x-y)\cdot\partial_x\big) \left( t^{-\frac{d}{2}}\widetilde{A}\Big(t,\frac{x-y}{\sqrt{t}},x\Big)\right) = t^{-\frac{d}{2}} t^{-\frac{1}{2}} \big(M(x-y)\cdot \partial_X\widetilde{A}\big)\Big(t,\frac{x-y}{\sqrt{t}},x\Big)h   \\
\lesssim \left(\sqrt{t}+\vert x-y\vert \right)^{-d}\frac{\vert x-y\vert}{\sqrt{t}}\left(1+\frac{\vert x-y\vert}{\sqrt{t}}\right)^{-N}
\lesssim \left(\sqrt{t}+\vert x-y\vert \right)^{-d},
\end{eqnarray*}
where we used the fact that $\frac{\vert x-y\vert}{\sqrt{t}}\left(1+\frac{\vert x-y\vert}{\sqrt{t}}\right)^{-N}$ is bounded when $N$ is large enough. Repeating the above bounds for $L_1\dots L_p K$ where $L_1,\dots,L_p\in \mathcal{M}$ immediately shows that the heat kernel is also an element of $\Psi_P^{-1}$.
\end{proof} 


\medskip

\textbf{\textsf{(b) Representation of $e^{-tP} P^{-1}$ in the parabolic calculus$\boldmath{.}$\hspace{0.1cm}}} The second kernel we need to describe is the Schwartz kernel of $e^{-tP} P^{-1}$ . We shall state some preliminary simple lemma: 

\smallskip

\begin{lemm} [From heat decay to parabolic distance] \label{l:ineq2}
In dimension $d\geqslant3$, we have the inequality
\begin{eqnarray*}
\int_{t}^1 s^{-\frac{d}{2}} \left(1+\frac{\vert x-y\vert}{\sqrt{s}} \right)^{-N-d} \rmd s\lesssim \left(1+\vert x-y\vert\right)^{-N}\frac{2}{d-2}\left(\sqrt{t}+\vert x-y\vert\right)^{-d+2}
\end{eqnarray*}
when $\max(t,\vert x-y\vert)$ tends to $0$
\end{lemm}
\begin{proof}
Set $a=\vert x-y\vert$, then:
\begin{equation*} \begin{split}
\int_{t}^1 s^{-\frac{d}{2}} \left(1+\frac{a}{\sqrt{s}}\right)^{-N-d}\rmd s &= \int_{t}^1 s^{-\frac{d}{2}}\left(1+\frac{a}{\sqrt{s}}\right)^{-d} \left(1+\frac{a}{\sqrt{s}}\right)^{-N}\rmd s   \\
&\leqslant (1+a)^{-N} \int_{t}^1  (\sqrt{s}+a)^{-d} \rmd s   \\
&\leqslant2(1+a)^{-N} \int_{\sqrt{t}}^1  (r+a)^{-d}  r\rmd r \leqslant 2(1+a)^{-N}\int_{\sqrt{t}}^1  (r+a)^{-d+1}  \rmd r   \\
&\leqslant(1+a)^{-N} \, \frac{2(1+a)^{-d+2}-2(\sqrt{t}+a)^{-d+2}}{-d+2}   \\
&\leqslant(1+a)^{-N} \, \frac{2}{d-2}(\sqrt{t}+a)^{-d+2}. 
\end{split} \end{equation*}
\end{proof}

\smallskip

Using the above Lemma we deduce information on the analytic structure of the Schwartz kernel of $e^{-tP} P^{-1}$.

\smallskip

\begin{prop}\label{p:propagatorinparabolic}
With the notations of Theorem \ref{thm:heatcalculus}, the kernel of $e^{-tP}P^{-1}$ belongs to $\Psi_P^{-2}$.
\end{prop}

\smallskip

\begin{proof}
We start by noticing that when $x\neq y$, the kernel $e^{-tP}P^{-1}(x,y), t\geqslant 0$ is smooth in $(t,x,y)$ since 
$$\partial_t^\alpha \partial_{x,y}^\beta e^{-tP}P^{-1}(x,y)=(-1)^{\alpha} \partial_{x,y}^\beta e^{-tP} P^{\alpha-1}(x,y)$$
which is smooth outside the diagonal as a bounded family 
of pseudo-differential operators uniformly in $t\in [0,1)$. For $t\geqslant 1$, $(-1)^{\alpha} \partial_{x,y}^\beta e^{-tP} P^{\alpha-1}(x,y)$ has smooth kernel by smoothing properties of the heat operator. We need to show that in some chart of the type $U\times \big\{\vert h\vert\leqslant R, h\in \mathbb{R}^d\big\}$ as in Theorem \ref{thm:heatcalculus}
\begin{eqnarray*}
e^{-tP} P^{-1}(x,y) = B\big(\sqrt{t},x-y,x\big)
\end{eqnarray*}
where the kernel $B(\sqrt{t},X,y)$ has the following decay properties
\begin{eqnarray*}
\vert  B(\sqrt{t},X,y)\vert\leqslant C_{N+d} \big(\sqrt{t}+\vert x-y\vert\big)^{-d+2} \big(1+\vert x-y\vert\big)^{-N}
\end{eqnarray*}
\textbf{uniformly} in $t$ in some compact interval. The positive mass ensures exponential decay of the integrand in the following integral formula
$$
e^{-tP} P^{-1}=\int_t^{+\infty} e^{-sP}\rmd s.
$$
As usual thanks to the exponential decay, we have a preliminary decomposition
as:
\begin{eqnarray*}
e^{-tP} P^{-1}=\int_t^{1} e^{-sP}\rmd s+\underbrace{\int_1^{+\infty} e^{-sP}\rmd s}
\end{eqnarray*}
where the integral underbraced converges absolutely as a smoothing operator thanks to the exponential decay.

 Then in a second step, we shall study the finite integral $\int_t^{1} e^{-sP}\rmd s $, we rely on the heat calculus representation
 $$
 e^{-sP}(x,y) = s^{-\frac{d}{2}}\widetilde{A}\Big(s,\frac{x-y}{\sqrt{s}},x\Big).
 $$
 Set 
 $$
 B(t,x,y) = \int_t^{1} s^{-\frac{d}{2}}\widetilde{A}\Big(s,\frac{x-y}{\sqrt{s}},x\Big)\rmd s.
 $$
By Lemma~\ref{l:ineq2}, we deduce the claimed estimate that reads
\begin{equation*} \begin{split}
\bigg\vert\int_t^{1} s^{-\frac{d}{2}}\widetilde{A}\Big(s,\frac{x-y}{\sqrt{s}},x\Big)\rmd s\bigg\vert &\leqslant C_{N+d} \int_t^{1} s^{-\frac{d}{2}}\left(1+\frac{\vert x-y\vert}{\sqrt{s}}\right)^{-N-d} \rmd s   \\
&\leqslant C_{N+d} \, \big(1+\vert x-y\vert\big)^{-N} \big(\sqrt{t}+\vert x-y\vert \big)^{-d+2}.
\end{split} \end{equation*}

Now we need to repeat the above bounds for derivatives
$$
\big\vert\partial_t^\alpha \left(e^{-tP}P^{-1}\right)(x,y)\big\vert = \big\vert\partial_t^{\alpha-1} \left(e^{-tP}\right)(x,y)\big\vert \leqslant \left( \sqrt{t}+\vert x-y\vert\right)^{-d-2\vert\alpha \vert+2}
$$
since $e^{-tP}$ belongs to the heat calculus.

Repeating the same estimates for the derivatives
$\partial_h^\alpha$, we get:
\begin{eqnarray*}
 \partial^\alpha_{h} \int_t^{1} s^{-\frac{d}{2}}\widetilde{A}\left(s,\frac{h}{\sqrt{s}},x\right) \rmd s=
 \int_t^{1} s^{-\frac{d+\vert \alpha\vert}{2}} \left(\partial^\alpha_{X}\widetilde{A}\right)\left(s,\frac{h}{\sqrt{s}},x\right)  \rmd s\\
 \leqslant C_{N+d+\vert\alpha\vert} \int_t^{1} s^{-\frac{d+\vert \alpha\vert}{2}} \left(1+\frac{\vert h \vert}{\sqrt{s}}\right)^{-N-d-\vert\alpha\vert}\rmd s\lesssim \left(\sqrt{t}+\vert h \vert \right)^{-d-\vert \alpha\vert+2}.
\end{eqnarray*} 
The estimates involving derivatives with respect to  to $y$ are treated similarly and are left to the reader. Lemma~\ref{l:gainderivative} allows to conclude that $ e^{-tP} P^{-1}$ belongs to $\Psi_P^{-2}$. 
\end{proof}

\medskip

\subsubsection{Parameter-dependent pseudo-differential operators$\boldmath{.}$\hspace{0.1cm}}

In this section we shall treat the elements from the parabolic calculus as parameter-dependent pseudo-differential operators acting on $M$, the time variable is treated as a parameter and controls the pseudo-differential order uniformly in the time parameter. Our main result here is Proposition \ref{l:parabolicasfamilypsidos}. To relate our definition of operators in terms of kernels with the classical Fourier definition in terms of symbols, we need to recall some statement which relates the order of a pseudo-differential operator with the growth of the symbol along the diagonal together with the diagonal growth of derivatives of the kernel. This can also be found in Taylor's book \cite[Prop 2.2 p.~6]{Taylor2} and goes back to the work of Kr\'ee and Seeley. Let $\mathcal{M}\subset C^\infty(T(M\times M))$ be the module of vector fields tangent to the diagonal in $M\times M$.

\smallskip

\begin{prop} [Characterization of pseudo-differential kernels] \label{prop:characterizationpsidokernels}
Pick $-d < m<0$. An element $K\in \mathcal{D}^{\prime}(M\times M)$ is the pseudo-differential kernel of some operator in $\Psi^m_{1,0}(M)$ if and only if
\begin{enumerate}
\item $K\in L^1(M\times M)$, 
\item near the diagonal and in some product chart of the form $\kappa\times \kappa: U\times U\subset M\times M\mapsto \mathbb{R}^d\times \mathbb{R}^d$ one has, for any $L_1,\dots,L_p\in \mathcal{M}^p$
$$
\big\vert \left(\kappa\times \kappa\right)_*(L_1\dots L_p K)(x,y) \big\vert \lesssim_{L_1,\dots,L_p} \vert x-y \vert^{-d-m}.
$$
\end{enumerate}
Furthermore, for any open cover $(U_i)_{i\in I}$ of $M$ we can use the constants 
$$
\sup_{(x,y)\in U^2}\vert (\kappa_i\times \kappa_i)_*(L_1\dots L_p K)(x,y) \vert  \vert x-y \vert^{d+m},
$$ 
together with the topology of smooth functions on $C^\infty( V )$, where $V\cup \left(\bigcup_{i\in I} U_i^2 \right)$ forms an open cover of $M\times M$, to get back the topology of pseudo-differential operators in $\Psi^{m}_{1,0}(M)$. 
\end{prop}

\smallskip

We characterize only those pseudo-differential kernels which are $L^1$; they correspond to operators with some smoothing properties. By the usual invariance properties of the pseudo-differential calculus, it is enough to prove the statement on $\mathbb{R}^d$. To go back to manifolds one can use a partition of unity as in~\cite[p.~84-87]{Horm3}. 

\smallskip

We shall deduce the Proposition \ref{prop:characterizationpsidokernels} from some elementary results which are of independent interest and will be used later. In the sequel, for every bounded open subset $U\subset \mathbb{R}^d$ we shall denote by $\mathcal{S}(U\times \mathbb{R}^d)$ the set of smooth functions $a\in C^\infty(U\times \mathbb{R}^d)$ which are Schwartz in the second variable uniformly in the first one. This locally convex topological vector space is determined by the seminorms 
$$ 
\sup_{x\in K, h\in \mathbb{R}^d} \big\vert(1+\vert h\vert)^N \left(\partial_x^\alpha\partial_h^\beta a \right) (x;h) \big\vert,
$$ 
where $K\subset U$ is compact. We use below the notation $(\Delta_j)_{j=1}^{\infty}$ be the sequence of Littlewood-Paley-Stein projectors on $\mathbb{R}^d$.

\begin{lemm}\label{l:technical1}
Pick $A\in \Psi^m_{1,0}(\mathbb{R}^d)$ for an arbitrary $m\in\bbR$. We can decompose $A$ as a series
$$
A = \sum_{j=1}^\infty A\Delta_j+R
$$
where $R\in \Psi^{-\infty}$ and the Schwartz kernel $[A\Delta_j]$ of $A\Delta_j$ can be represented as
\begin{eqnarray*}
[A\Delta_j]\big(x,x-y\big) = 2^{j(d+m)}K_j\big(x,2^j(x-y)\big)
\end{eqnarray*}
where $(K_j)_j$ is a bounded sequence in $\mathcal{S}(U\times \mathbb{R}^d)$, in the sense that one has
\begin{eqnarray*}
\big\vert(\partial_x^\alpha\partial_h^\beta K_j)(x,h)\big\vert \leqslant C_{N,\alpha,\beta} \big(1+\vert h\vert\big)^{-N}
\end{eqnarray*}
for all $N$. Moreover, all the kernels $K_j$ have vanishing moments in the second variable:
$$
\int K_j(x,h)h^\alpha \rmd h=0
$$ 
for all multiindex $\alpha$.
\end{lemm}

\smallskip

\begin{proof}
The proof is similar to the proof of Proposition 2.2 in Taylor's book \cite{Taylor2}. We start from the expression of the pseudo-differential kernel $K(x,y)=\widetilde{K}(x,x-y)$ of $A$ in terms of the symbol of $A$
\begin{eqnarray*}
\widetilde{K}(x,h)=\frac{1}{(2\pi)^d}\int_{\mathbb{R}^d} e^{i\xi\cdot h}a(x;\xi)\rmd\xi
\end{eqnarray*}
where the integral is understood as a Fourier integral distribution -- a non-convergent integral due to the slow decay of $a$ in the variable $\xi$. We need to define $\tilde{K}$ in the sense of distributions. This involves pairing against test functions. Indeed, we define for any test function $\varphi\in C^\infty_c(\mathbb{R}^d)$:
\begin{eqnarray*}
\int_{\mathbb{R}^d}\int_{\mathbb{R}^d} e^{i\xi\cdot h}a(x;\xi) \varphi(h) \rmd\xi \rmd h &=& \int_{\mathbb{R}^d}\int_{\mathbb{R}^d}\left( \left(\frac{\sum_{j=1}^d\xi_j\partial_{h^j}}{i\vert \xi\vert^2}\right)^N e^{i\xi\cdot h}\right)a(x;\xi)\varphi(h)\rmd\xi \rmd h\\
&=&(-1)^N\int_{\mathbb{R}^d}\int_{\mathbb{R}^d}\left( \left(\frac{\sum_{j=1}^d\xi_j\partial_{h^j}}{i\vert \xi\vert^2}\right)^Na(x;\xi)\varphi(h) \right)e^{i\xi\cdot h}\rmd\xi \rmd h
\end{eqnarray*}   
where the rightmost integral is seen to converge when $N$ is large enough since the integration by parts brings in some decay. The function $\widetilde{K}$ is even smooth outside $h=0$. Recall the Littlewood-Paley-Stein projectors were associated to a dyadic partition of unity in frequency
$1=\psi_0+\sum_{n=1}^\infty \psi(2^{-n}\cdot)$, $\psi$ is supported on some corona $1\leqslant\vert \xi\vert\leqslant 4 $.
The kernel of $A\Delta_j$ reads 
\begin{eqnarray*}
\frac{1}{(2\pi)^d}\int_{\mathbb{R}^d} e^{i\xi\cdot h}a(x;\xi) \psi(2^{-j}\xi)\rmd\xi.
\end{eqnarray*}
This yields a series decomposition of the form
\begin{equation*} \begin{split}
\widetilde{K}(x,h) &= \frac{1}{(2\pi)^d}\int_{\mathbb{R}^d} e^{i\xi\cdot h}\sum_{j=1}^\infty \psi(2^{-j}\xi) a(x;\xi)\rmd\xi + R   \\
&= \sum_{j=1}^\infty 2^{jd}\frac{1}{(2\pi)^d}\int_{\mathbb{R}^d} e^{i\xi\cdot(2^jh)} \psi(\xi) a(x;2^j\xi)\rmd\xi + R
\end{split} \end{equation*}
where $R\in C^\infty(U\times \bbR^d)$. Now, we use the growth of the symbol and the annulus support to derive the estimate
$$ 
\vert \partial_\xi^\beta \left(\psi(\xi) a(x;2^j\xi)\right)\vert \lesssim  1_{ \{ 1\leqslant \vert \xi\vert \leqslant 4 \} }   \big(1+2^j\vert \xi\vert\big)^{m-\vert\beta\vert} 2^{j\vert \beta\vert}  \leqslant  C_\beta 2^{jm} 
$$ 
since
$$
\vert\left(\partial_\xi^\beta a\right)(x;2^j\xi)\vert\lesssim \big(1+2^j\vert \xi\vert\big)^{m-\vert\beta\vert}
$$ 
because $a$ is a classical symbol. Above, the constant $C_\beta$ does not depend on $m$. This means that $\left(2^{-jm}\psi(\xi) a(x;2^j\xi)\right)_j $ is a bounded sequence of Schwartz functions of $\xi$ uniformly in $x\in U$. This implies that we decomposed our kernel $\widetilde{K}$ as an infinite series:
\begin{eqnarray*}
\widetilde{K}(x,h)=\sum_{j=1}^\infty 2^{j(d+m)} K_j(x,2^jh) + R
\end{eqnarray*} 
where 
$$
\left(K_j(x,h)= \frac{2^{-jm}}{(2\pi)^d}\int_{\mathbb{R}^d} e^{i\xi\cdot h} \psi(\xi) a(x;2^j\xi)\rmd\xi \right)_{j\geqslant 1}
$$ 
is a bounded sequence of smooth functions in $\mathcal{S}(U\times \mathbb{R}^d)$. These are functions smooth in the first variable $x$ and smooth with fast decay in the second variable $h$. Moreover, $\int_{\mathbb{R}^d}K_j(x,h) h^\alpha \rmd h=0$ for all multiindices $\alpha$, the vanishing moment condition immediately follows from the fact that $K_j(x,\cdot)$ has Fourier support away from the origin. Note that the smooth remainder kernel $R(x,.)$ does not satisfy the vanishing moments condition. 
\end{proof}

\smallskip

From the above series representation we deduce a bound on the kernel of each piece $A\Delta_j$. 

\smallskip

\begin{coro}\label{coro:estimateeachpiece}
One has for all $N\geqslant 0$ and multiindices $(\alpha,\beta)$:
$$ 
\big\vert\partial_x^\alpha \partial_h^\beta  [A\Delta_j](x,h) \big\vert \leqslant C_{N,\alpha,\beta} \, 2^{j(d+m+\vert \beta\vert)} \big(1+2^j\vert h\vert\big)^{-N} 
$$
for some positive constant $C_{N,\alpha,\beta}$ independent of $j$.
\end{coro}

\smallskip

Now we specialize to the case where $ m\in (-d,0) $ and try to deduce growth estimates on the Schwartz kernel of $A$ from the series decomposition. We also discuss the limiting cases $m=0$, $m=-d$ as well as $m<-d$ with $m$ non-integer.

\smallskip

\begin{lemm} \label{LemDiagonalScalingM}
For $A\in \Psi_{1,0}^m(\bbR^d)$ with $-d<m<0$ one has
$$ 
\big\vert \big( L_1\dots L_p K \big) (x,y)\big\vert \lesssim \vert  x-y\vert^{-d-m}
$$
for all vector fields $L_1,\dots,L_p$ that are tangent to the diagonal. For $A\in \Psi_{1,0}^{-d}(\bbR^d)$ one has
$$
\big\vert \big( L_1\dots L_p K \big) (x,y) \big\vert \lesssim \big\vert\log \vert  x-y\vert\big\vert
$$
for all vector fields $L_1,\dots,L_p$ which are tangent to the diagonal.
\end{lemm}

\smallskip

\begin{proof}
The series decomposition of $\widetilde{K}$ implies some decay of the form
\begin{eqnarray*}
\big\vert \widetilde{K}(x,h)\big\vert \lesssim \sum_{j=1}^\infty 2^{j(d+m)} \big(1+2^j\vert h\vert\big)^{-N}\lesssim \sum_{j=1}^\infty \big(2^{-j}+\vert h\vert\big)^{-(d+m)}\lesssim \vert h\vert^{-m-d}
\end{eqnarray*}
when $m\in (-d,0)$. When $m=-d$, the above bound reads
\begin{eqnarray*}
\big\vert \widetilde{K}(x,h)\big\vert \lesssim \sum_{n=1}^\infty  \big(1+2^n\vert h\vert\big)^{-N} \lesssim \int_1^\infty \big(1+u\vert h\vert\big)^{-N}\frac{\rmd u}{u} = \int_{\vert h\vert}^\infty (1+u)^{-N}\frac{\rmd u}{u}\lesssim \big\vert \log\vert h\vert\big\vert.
\end{eqnarray*}
Note that when $m<-d$ and $m$ non-integer, each $K_j(x,\cdot)=2^{-j(m+d)} \big(\Delta_{j}\widetilde{K}(x,\cdot)\big)$~\footnote{The Littlewood-Paley projection is in the second variable $h$} is bounded and Fourier supported in a corona where $\vert\xi\vert\simeq 2^j$. Hence we recognize that 
$$
\sup_{x\in U} \Vert\partial_x^\alpha\widetilde{K}(x,.)\Vert_{C^{-m-d}(\mathbb{R}^d)}<+\infty 
$$ 
by the Fourier definition of the norm of $C^{-m-d}={B}^{-m-d}_{\infty,\infty}$ for all multiindex $\alpha$.

The estimates for tangential derivatives work similarly, as follows. As in the proof of Lemma \ref{l:gainderivative} it suffices to prove an estimate of the form
$$
\big\vert\partial_x^\alpha\partial_h^\beta\widetilde{K}(x,h)\big\vert \lesssim \vert h\vert^{-m-d-\vert \beta\vert}.
$$ 
We differentiate
\begin{eqnarray*}
\big\vert \partial_x^\alpha \partial_h^\beta \widetilde{K}(x,h) \big\vert \leqslant \sum_{j=1}^\infty 2^{j(d+m)} \big\vert \partial_x^\alpha \partial_h^\beta K_j(x,2^jh) \big\vert \lesssim \sum_{j=1}^\infty 2^{j(d+m+ \vert\beta\vert)} \big\vert\big(\partial_x^\alpha \partial_h^\beta K_j\big)(x,2^jh)\big\vert
\end{eqnarray*}
where we work with the new bounded sequence of smooth functions $\big(\partial_x^\alpha \partial_h^\beta K_j\big)_j$. Repeating the above steps for this new sequence taking into account the extra factor $2^{j\vert\beta\vert}$ yields the desired estimate
$$
\big\vert\partial_x^\alpha\partial_h^\beta\widetilde{K}(x,h)\big\vert \lesssim \vert h\vert^{-m-d-\vert \beta\vert}.
$$ 
\end{proof}

The direct sense of Proposition~\ref{prop:characterizationpsidokernels} follows from Lemma \ref{LemDiagonalScalingM}. Conversely, assume we are given a bounded sequence $(K_j)_j$ of smooth functions in $\mathcal{S}(U\times \mathbb{R}^d)$. Under which condition on $m$ does the series
\begin{eqnarray*}
\sum_{j=1}^\infty 2^{j(d+m)} K_j\big(x,2^j(x-y)\big)
\end{eqnarray*}
converge in pseudo-differential kernels in $\Psi^m_{1,0}(U)$? An answer is provided by the following statement.

\smallskip

\begin{prop}[Reconstructing pseudo-differential operators by renormalization]
The following holds.
\begin{enumerate}
	\item For $-d<m<0$ the above series converges to some pseudo-differential kernel in $\Psi^m_{1,0}(U)$. 
	\item For $m> 0$, $m\notin \mathbb{N}$, there exists a non-unique sequence $c_{j,\alpha}(x), \vert\alpha\vert\leqslant [m]$, of \textbf{counterterms} depending smoothly on $x$ such that the \textbf{renormalized series}
\begin{eqnarray*}
\sum_{j=1}^\infty \left( 2^{j(d+m)} K_j\big(x,2^j(x-y)\big) - \sum_{\vert\alpha\vert\leqslant [m]} c_{j,\alpha}(x)\partial^\alpha\delta_{\{0\}}(x-y)\right)
\end{eqnarray*}
converges as a pseudo-differential kernel in $\Psi^m_{1,0}(U)$.
	\item For $m<-d$, $|m|\notin \mathbb{N}$, there exists a non-unique sequence $c_{j,\alpha}, \vert\alpha\vert\leqslant [|m|-d]$, of \textbf{counterterms} such that the \textbf{renormalized series}
\begin{eqnarray*}
\sum_{j=1}^\infty \left( 2^{j(d+m)} K_j\big(x,2^j(x-y)\big) - \sum_{\vert\alpha\vert\leqslant [|m|-d]} c_{j,\alpha}(x-y)^\alpha\right)
\end{eqnarray*}
converges as a pseudo-differential kernel in $\Psi^m_{1,0}(U)$.
\end{enumerate}
\end{prop}

\smallskip

\begin{proof}
\textsf{\textit{(a)}} We start by proving the easy case where $m\in (-d,0)$. The idea is just to Fourier transform each $2^{jd} K_j(x,2^jh)$ in the second variable $h$ and translate in terms of estimates in Fourier space what it means for $(K_j)_j$ to be bounded in the space $\mathcal{S}(U\times \mathbb{R}^d)$ of Schwartz functions of $h$.
We get a series
\begin{eqnarray*}
a(x;\xi)=\sum_{j=1}^\infty 2^{jm} \widehat{K}_j(x;2^{-j}\xi)
\end{eqnarray*} 
where $\widehat{K}_j$ is a bounded sequence in  $\mathcal{S}(U\times \mathbb{R}^d)$.
Let us now prove that the series 
$$
\sum_{j=1}^\infty 2^{j(m-\vert\beta\vert)}\big( \partial_x^\alpha \partial_\xi^\beta \widehat{K}_j\big)(x;2^{-j}\xi)
$$ 
converges for all multiindices $\alpha,\beta$
and bound it uniformly in $\xi$.
If $m<0$, the series converges absolutely uniformly in $\xi$ in some arbitrary compact region, it satisfies the estimate:
{
\begin{align}\nonumber
\bigg\vert\sum_{j=1}^\infty 2^{j(m-\vert\beta\vert)} \big(\partial_x^\alpha \partial_\xi^\beta \widehat{K}_j\big)(x;2^{-j}\xi)\bigg\vert &\leqslant \sum_{j=1}^\infty 2^{j(m-\vert\beta\vert)}\big\vert\big(\partial_x^\alpha \partial_\xi^\beta \widehat{K}_j\big)(x;2^{-j}\xi) \big\vert   \\
&\lesssim \sum_{j=1}^\infty  2^{j(m-\vert\beta\vert)} \big(1+2^{-j}\vert \xi\vert\big)^{-N}   \nonumber\\ \label{eq:splitkilllogtrick}
&\lesssim
 \sum_{j\leqslant\log_2(1+|\xi|)}  2^{j(m-\vert\beta\vert)} \big(1+2^{-j}\vert \xi\vert\big)^{m-\vert\beta\vert-\varepsilon}
+
 \sum_{j>\log_2(1+|\xi|)}  2^{j(m-\vert\beta\vert)}
   \\ \nonumber
&\lesssim
 \sum_{j\leqslant\log_2(1+|\xi|)}  2^{j\varepsilon} \big(1+\vert \xi\vert\big)^{m-\vert\beta\vert-\varepsilon}
+\big(1+|\xi|\big)^{m-\vert\beta\vert}
 \\ \nonumber
&\lesssim\big(1+|\xi|\big)^{m-\vert\beta\vert}\,.
\end{align} 
}

\textsf{\textit{(b)}} We now treat the more difficult case of a non-integer positive $m$. In that case, the series $\sum_{j=1}^\infty 2^{jm} \widehat{K}_j(x;2^{-j}\xi)$ is highly divergent. Instead of considering $\widehat{K}_j(x;\xi)$, we subtract its Taylor polynomial at $\xi=0$ to increase the vanishing order at $\xi=0$ of
$\widehat{K}_j(x;\xi)$. This yields
\begin{eqnarray*}
R_j(x;\xi)\defeq \widehat{K}_j(x;\xi)-\sum_{\vert\alpha\vert\leqslant [m]} \frac{\xi^\alpha}{\alpha!}\partial_\xi^\alpha \widehat{K}_j(x;0).
\end{eqnarray*}  
Note that $R_j(x;\xi)=\mathcal{O}(\vert\xi\vert^{[m]+1})$ near $\xi=0$ but it is no longer Schwartz in $\xi$ since we subtracted some polynomial.
Instead, for $|\beta|\leqslant [m]$, it satisfies new estimates of the form
$$
\big\vert \partial_x^\alpha\partial_\xi^\beta R_j(x;\xi)\big\vert \lesssim \big(1+\vert\xi\vert\big)^{[m]-\vert\beta\vert} 
$$
for large $\vert\xi\vert$ because the Taylor polynomial has growth at most $[m]$ and differentiating $\beta$ times, we are left with a polynomial of degree at most $[m]-\vert \beta\vert$ and 
$$
\big\vert \partial_x^\alpha\partial_\xi^\beta R_j(x;\xi)\big\vert \lesssim \vert\xi\vert^{[m]+1-\vert\beta\vert} 
$$ 
for small $\vert \xi\vert$ because we deal with Taylor remainders. Now consider the new renormalized series $\sum_j 2^{jm}R_j(x;2^{-j}\xi) $. First, {if we restrict $\xi$ to any compact region then it converges absolutely uniformly in $\xi$}, since using $m<[m]+1$ we have
\begin{eqnarray*}
\Big\vert \sum_j 2^{j(m-|\beta|)} \big(\partial_x^\alpha\partial_\xi^\beta R_j\big)(x;2^{-j}\xi) \Big\vert \lesssim \sum_j 2^{j(m-[m]-1)} < \infty\,.
\end{eqnarray*}
It also satisfies for $|\beta|\leqslant [m]$ and $\vert \xi\vert\geqslant 1$ the estimate
\begin{equation*} \begin{split}
\Big\vert \sum_j 2^{j(m-|\beta|)}& \big(\partial_x^\alpha\partial_\xi^\beta R_j\big)(x;2^{-j}\xi) \Big\vert   \\
&\leqslant\sum_{j\leqslant \log_2(1+|\xi|)}2^{j(m-|\beta|)} \big\vert \big(\partial_x^\alpha\partial_\xi^\beta R_j \big)  (x;2^{-j}\xi) \big\vert+\sum_{j> \log_2(1+|\xi|)}2^{j(m-|\beta|)} \big\vert  \big( \partial_x^\alpha\partial_\xi^\beta R_j \big) (x;2^{-j}\xi) \big\vert   \\
&\lesssim\sum_{j\leqslant \log_2(1+|\xi|)}2^{j(m-|\beta|)} \big(1+2^{-j}|\xi|\big)^{[m]-|\beta|}+\sum_{j> \log_2(1+|\xi|)}2^{j(m-[m]-1)} \vert\xi\vert^{[m]+1-\vert\beta\vert}   \\ 
&\lesssim \sum_{j\leqslant \log_2(1+|\xi|)}2^{j(m-[m])} \big(1+\vert\xi\vert\big)^{[m]-\vert\beta\vert} + 
|\xi|^{m-|\beta|}\\
&\lesssim\big(1+\vert\xi\vert\big)^{[m]-\vert\beta\vert} 
\end{split} \end{equation*}
where we used separately the estimate on $\partial_x^\alpha\partial_\xi^\beta R_j$ for the large (resp small) $\vert\xi\vert$ regime to control the sum 
$\sum_{j\leqslant \log_2(1+|\xi|)}$ (resp $ \sum_{j>\log_2(1+|\xi|)} $).
Finally, the estimate for $|\beta|> [m]$ is easier to prove, since in this case we do not have to handle the presence of the Taylor expansion  and $m-\vert \beta\vert<0$ so we can handle it with the same methods as in situation \textit{(a)}.

This proves that the renormalized series $\sum_j 2^{jm}R_j(x;2^{-j}\xi)$ converges to $S^{m}_{1,0}(U\times \mathbb{R}^d) $. Going back to position space, this implies that the renormalized series
$$
\sum_j\Big( 2^{j(d+m)}K_j(x,2^j(x-y))-\sum_{\vert \alpha\vert\leqslant m} \frac{ \partial_\xi^\alpha \widehat{K}_j(x;0)}{i^{\vert\alpha\vert}\alpha!}
\partial_x^\alpha\delta_{\{0\}}(x-y)\Big)
$$
converges in pseudo-differential kernels of order $m$.

\smallskip

\textsf{\textit{(c)}} The case $m<-d, m\notin\mathbb{N}$, involves a renormalization by subtraction of some derivatives of $\delta_{\{0\}}$ in Fourier space, by inverse Fourier transform this yields the floating polynomials that we need to subtract to get the correct renormalized convergence as pseudo-differential kernel. We leave the details to the reader.
\end{proof}

\smallskip

Now we may conclude the proof of Proposition \ref{prop:characterizationpsidokernels} by proving only the converse sense.

\smallskip

\begin{proof}
The converse sense uses dyadic decomposition in space. We will assume we are given a kernel $K\in L^1(\mathbb{R}^d\times \mathbb{R}^d)$, smooth outside the diagonal and such that 
for any bounded open $U\subset \mathbb{R}^d$, $L_1,\dots,L_p\in \mathcal{M}^p$
$$
\sup_{(x,y)\in U^2} \big\vert (L_1\dots L_p K)(x,y) \big\vert \leqslant C_{L_1,\dots,L_p}  \vert x-y \vert^{-d-m}.
$$  
Start from $1= \sum_{j=1}^\infty \psi(2^{j}.) +\chi $ where $\chi$ vanishes near $0$ and $\psi$ is supported on an annulus $ \{\frac{1}{2}\leqslant\vert x\vert\leqslant 4\}$. The central fact is to prove that the family 
$$
K_j(x,.)\defeq 2^{-j(m+d)} \psi(h) K(x,2^{-j}\cdot) 
$$ 
is bounded in $C^\infty_0(\mathbb{R}^d)$ uniformly in $x\in U$. Note that the key Lemma \ref{l:gainderivative} yields
\begin{equation*} \begin{split}
\sup_{x\in U,\frac{1}{2}\leqslant \vert h\vert\leqslant 4} \big\vert D_h^\alpha \left(K(x,2^{-j}h)\right)\big\vert &\leqslant \sup_{x\in U,\frac{1}{2}\leqslant \vert h\vert\leqslant 4} 2^{-j\vert\alpha\vert} \big\vert \left(D_h^\alpha K\right)(x,2^{-j}h)\big\vert   \\
&\leqslant C_{\alpha,U}  2^{-j\vert\alpha\vert} \sup_{\frac{1}{2}\leqslant\vert h\vert\leqslant 4}\vert 2^{-j}h \vert^{-d-m-\vert \alpha\vert}   \\
&\leqslant C_{\alpha,U}  2^{-j\vert\alpha\vert} 2^{(j+1)(d+m+\vert \alpha\vert)}\leqslant2^{\vert \alpha\vert}C_{\alpha,U} 2^{(j+1)(d+m)},
\end{split} \end{equation*}
therefore 
$$ 
\sup_{x\in U,h} \Big\vert D_h^\alpha\Big(2^{-j(m+d)}\psi(h) \left(K(x,2^{-j}h)\right)\Big)\Big\vert \lesssim C_{\alpha,U}.
$$ 
Now we get the decomposition
$$
K(x,h) = K\chi+\sum_{j=1}^\infty \psi(2^jh)K(x,h) = K\chi+\sum_{j=1}^\infty 2^{j(m+d)} K_j(x,2^jh),
$$
with 
$$
K_j=2^{-j(m+d)} \psi(h) K(x,2^{-j}h).
$$
The sequence $(K_j)_{j\in \mathbb{N}}$ is a bounded family of smooth functions supported in the fixed annulus $\{ \frac{1}{2}\leqslant \vert h\vert\leqslant 4 \}$ and $K\chi$ is Schwartz. Hence the Fourier transform in the $h$ variable yields
\begin{eqnarray*}
\widehat{K}(x;\xi)=\widehat{K\chi}+\sum_{j=1}^\infty 2^{jm} \widehat{K_j}(x,2^{-j}\xi),
\end{eqnarray*}
where each $K_j(x;\xi)$ is Schwartz in $\xi$ uniformly in $x\in U$.
We just need to prove that $\widehat{K}(x;\xi)\in S^m_{1,0}$. We have the estimate 
$$
\vert K_j(x;\xi) \vert\leqslant C_N \big(1+\vert\xi\vert\big)^{-N}
$$
which holds uniformly in $j$. Hence using the same method as for the estimate~\ref{eq:splitkilllogtrick}, we deduce by summing over $j$ that
\begin{equation*} \begin{split} 
&\sum_{j=1}^\infty 2^{jm} \big\vert\widehat{K_j}(x,2^{-j}\xi)\big\vert \leqslant C_N  \sum_{j=1}^\infty 2^{jm} \big(1+2^{-j}\vert\xi\vert\big)^{-N} \\
&\leqslant C_N\sum_{j\leqslant \log_2(1+\vert \xi\vert)}  2^{jm}   \big(1+2^{-j}\vert\xi\vert\big)^{m-\varepsilon}   + C_N\sum_{j> \log_2(1+\vert \xi\vert)}  2^{jm}\big(1+2^{-j}\vert\xi\vert\big)^{-N}   
\lesssim C_N \big(1+\vert\xi\vert\big)^{m}.
\end{split} \end{equation*}
The estimate for $\big\vert\partial^\alpha_x\partial^\beta_\xi \widehat{K}(x;\xi)\big\vert$ follows from Bernstein inequality. For every multi-index $\alpha$, the family $\big(\partial_\xi^\alpha \widehat{K}_j(x,.)\big)_j$ is bounded in $\mathcal{S}(\mathbb{R}^d)$ uniformly in $x\in U$ since $\partial_\xi^\alpha \widehat{K}_j=i^{\alpha}\widehat{x^\alpha K_j}$ and each $x^\alpha K_j$ is supported by the annulus $\{\frac{1}{2}\leqslant\vert h\vert \leqslant 4 \}$.
\end{proof}

\smallskip

\begin{prop} [From parabolic calculus to time dependent pseudo-differential operators] \label{l:parabolicasfamilypsidos}
 Let $K(t,x,y)$ be some kernel on $(0,+\infty)\times M^2$ which belongs to the parabolic calculus $\Psi_P^{a}$ for some $a<-1$. Then $A(t)$ is \textbf{continuous} in $\Psi_{1,0}^{2+2a}(M)$ uniformly in the parameter $t\in (0,1]$.
\end{prop}

\smallskip

\begin{proof}
Observe that
$t\mapsto K(t,x,y) $ is continuous and the domination 
bound:
$$
\big\vert (L_1\dots L_k K)(t,x,y) \big\vert \leqslant C_{L_1\dots L_k} \big(\sqrt{t}+\vert x-y \vert \big)^{-d-2-2a} \leqslant C_{L_1\dots L_k} \vert x-y \vert^{-d-2-2a} \in L^1_{loc} 
$$
shows that $t\mapsto (L_1\dots L_k K)(t,.,.)\in L^1(\mathbb{R}^d\times \mathbb{R}^d)$ is \textbf{continuous}. So we just need to repeat the proof of Proposition~\ref{prop:characterizationpsidokernels} in a family version with the parameter $t$, each function $K_{j,t}$ depends \textbf{continuously} on $t$. We check that this is a bounded family of smooth functions supported on the annulus which depends continuously on $t$ and the boundedness in $C^\infty_0\big(\{ \frac{1}{2}\leqslant\vert h\vert\leqslant 4\}\big)$ is uniform in $t\in (0,a]$, for any $a<+\infty$. Then the proof follows by dominated convergence. 
\end{proof}

\smallskip

\subsection{Estimate on the kernel of the resonant term.  \hspace{0.1cm}}
\label{SectionProofMainThm2}

The following statement was proved in Section 3.2 of \cite{BDFT}. We recall its detailed proof to make the present work self-contained.

\smallskip

\begin{lemm} \label{l:topologytwopointfunctiong}
Let $M$ be a closed manifold and $K_t(x,y)$ be a smooth kernel on $M^2\backslash{\bf d}_2$ such that one can associate to any small enough open set $U$ a coordinate system in which one has for all multiindices $\alpha,\beta$
\begin{equation} \label{EqHomogeneityBoundsKernels}
\big\vert\partial^\alpha_{s,t}\partial^\beta_{x,y} K_{\vert t-s\vert}(x,y)\big\vert \lesssim \left(\sqrt{t-s}+\vert y-x\vert\right)^{a-2\vert\alpha\vert - \vert\beta\vert}.
\end{equation}
Denote by $\rho_2$ a scaling field on $M^2$ for the inclusion $\mathbf{ d}_2\subset M^2$ and set 
$$
\rho = 2(t-s)\partial_t + \rho_2,
$$
and for $n\geqslant 2$, we denote by $\pi : (x_1,\dots,x_n)\in M^n\rightarrow (x_1,x_2)\in M^2$ the canonical projection on the first two components. Then the family 
$$
\Big(e^{-\ell a}(e^{-\ell\rho})^*\pi^*K_{\vert t-\cdot\vert}(\cdot,\cdot)\Big)_{\ell\geqslant 0}
$$
is bounded in $\mathcal{D}^\prime_{N^*(\{s=t\})}\left((M^n\times\mathbb{R})\backslash(\pi^*\mathbf{d_2}\cap\{s=t\})\right)$, that is 
$$
\pi^*K_{\vert t-\cdot\vert}(\cdot,\cdot)\in\mathcal{S}^{0,(a,\rho)}_{N^*(\{s=t\})}\big((M^n\times\mathbb{R})\backslash(\pi^*\mathbf{d_2}\cap\{s=t\})\big).
$$
\end{lemm}

\smallskip

In the sequel, we denote by $\mathcal{K}^a$ the $C^\infty$--module of kernels $K_t(x,y)$ as above depending on two variables endowed with the weakest topology containing the $C^\infty\left([0,+\infty)\times  M^2\backslash\mathbf{d}_2\right)$ topology and which makes all the seminorms defined by the estimates \eqref{EqHomogeneityBoundsKernels} continuous.

\smallskip

\begin{proof}
We first localize in a neighbourhood $U\times U$ of the diagonal since $K$ is smooth off-diagonal. It is enough to prove the claim for $K(x,y)\chi_1(y)\chi_2(x)$ where $\chi_i\in C_c^\infty(U)$ and use a partition of unity to get the global result. In $U\times U$ we pull-back everything to the configuration space, which we write with a slight abuse of notations
$$ 
\pi^*(K\chi_1\chi_2) \big(t,s,x_1,\dots,x_n\big) = K\big(t,s,x_1,x_2\big) \, \chi_1(x_1)\chi_2(x_2).
$$
We already know that this kernel satisfies some bound of the form
\begin{eqnarray*}
\big\vert K(t,s,x_1,x_2)\chi_1(x_1)\chi_2(x_2)\big\vert \lesssim \left(\sqrt{\vert t-s\vert}+\vert x_1-x_2\vert \right)^{a}.
\end{eqnarray*}
Somehow we would like to flow both sides of the inequality by the parabolic dynamics $(e^{-r\rho})^*$ and bound the term $e^{-r\rho*}\left(\sqrt{\vert t-s\vert}+\vert x_1-x_2\vert \right)^{a}$ asymptotically when $r$ goes to $+\infty$. We use for that purpose the Normal Form Theorem for the space part of the Euler vector fields:
$$
\rho_{[n]} = \sum_{k=2}^n h_k\cdot \partial_{h_k}, 
$$
for some new coordinates $(h_k)_{k=2}^n$ that vanish at order $1$ along the deep space diagonal ${\bf d}_n$. The fact that $x_1-x_2$ vanishes at first order along ${\bf d}_n$ implies by Taylor expansion at first order that
$$
x_1-x_2=A(h)+\mathcal{O}(\vert h\vert^2)
$$
where $A(h)$ is a linear function of $(h_k)_{k=2}^n$. One then has
$$
(e^{-t\rho_{[n]}})^*\left(x_1-x_2\right) = (e^{-t\rho_{[n]}})^* A(h) + \mathcal{O}(e^{-2t}\vert h\vert^2) = A(e^{-t}h) + \mathcal{O}(e^{-2t}\vert h\vert^2), 
$$
and an exponential lower bound of the form
$$
e^{-t}\vert x_1-x_2\vert\lesssim \big\vert (e^{-t\rho_{[n]}})^*(x_1-x_2) \big\vert
$$
which yields the desired bound
$$
\big\vert e^{-u\rho*}D_{t}^\alpha D_x^\beta\pi^*(K\chi_1\chi_2)(t,s,x_1,\dots,x_n)\big\vert \lesssim e^{u(-a+2\vert \alpha\vert+\vert \beta\vert)} \Big(\sqrt{\vert t-s\vert}+\vert x_1-x_2\vert \Big)^{a-2\vert \alpha\vert-\vert \beta\vert}
$$
and proves the claim. The above bound allows, for instance, to justify that the singularities when $x_1\neq x_2$ are conormal along the equal time region $t=s$ since we are smooth on each half region $t\geqslant s$ and $s>t$.
\end{proof}

\smallskip

We are now ready to handle the case of the Schwartz kernel $[\odot_i]$ of the resonant product $\odot_i$ localized in the chart with index $i$.

\smallskip

\begin{lemm}\label{lem:7_2}
The Schwartz kernel $[\odot_i]$ belongs to $\mathcal{S}^{-6}_\Gamma(M^3)$ for $\Gamma=N^*\left( \{x=y=z\}\subset M^3 \right)$.
\end{lemm}

\smallskip

In other words, $[\odot_i]$ is a particular case of a conormal distribution whose wavefront set is concentrated along the deepest diagonal of $M^3$.

\smallskip

\begin{proof}
We shall assume without loss of generality that $[\odot_i]$ is a compactly supported distribution on $(\mathbb{R}^d)^3$.
Recall $[\odot_i]$ is expressed as a series
$$ [\odot_i](x,y,z)=\sum_{\vert k-\ell\vert\leqslant 1} P_k^i(x,y)\widetilde{P}^i_\ell(y,z)  $$
where $P,\widetilde{P}$ are generalized Littlewood-Paley-Stein projectors in our sense.
We use the diagonal bound on the Littlewood-Paley-Stein projectors $P^i_k,\widetilde{P}^i_\ell$ 
We use the fact that we control the scaling behaviour of each kernel, in local chart $U^3$ near the smallest diagonal $(x=y=z)$, we have the behaviour for $(h_1,h_2)\neq (0,0)\in \mathbb{R}^d$
\begin{eqnarray*}
\big\vert [\odot_i](y+h_1,y,y+h_2) \big\vert &\leqslant & \sum_{\vert k-\ell\vert\leqslant 1} \big\vert P_k^i(y+h_1,y)\widetilde{P}^i_\ell(y,y+h_2)\big\vert  \\
& \lesssim & \sum_{\vert k-\ell\vert\leqslant 1} 2^{(k+\ell)d} \big(1+2^k\vert h_{1} \vert\big)^{-N} \big(1+2^\ell\vert h_{2} \vert\big)^{-N}   \\
&\lesssim & \sum_{\ell=1}^\infty 2^{2d\ell} \big(1+2^{\ell}\vert h_{1} \vert\big)^{-N} \big(1+2^{\ell}\vert h_{2} \vert\big)^{-N}
\end{eqnarray*}
where the dimension $d=3$. Beware that the right-hand side  of the above estimate blows up when $h_1=h_2=0$ and the kernel $[\odot_i] $ is not even in $L^1_{loc}$. However the series  $\sum_{\vert k-\ell\vert\leqslant 1} P_k^i(y+h_1,y)\widetilde{P}^i_\ell(y,y+h_2)$ defining $[\odot_i] $ converges in the sense of distributions of order $0$. The same estimate with derivatives reads
\begin{eqnarray*}
\big\vert \partial_{h_1}^\alpha\partial_y^\beta \partial_{h_2}^\gamma [\odot_i] \big(y+h_1,y,y+h_2\big) \big\vert &\lesssim &\sum_{\ell=1}^\infty 2^{(2d+\vert\alpha \vert+\vert \gamma\vert )\ell } \big(1+2^{\ell}\vert h_{1} \vert\big)^{-N} \big(1+2^{\ell}\vert h_{2} \big\vert)^{-N}
\end{eqnarray*}
where $(\alpha,\beta,\gamma)$ are arbitrary multi-indices. The above bound implies that the series 
$$
\sum_{\vert k-\ell\vert\leqslant 1} P_k^i(y+h_1,y) \, \widetilde{P}^i_\ell(y,y+h_2)
$$ 
converges absolutely in $C^\infty$ when $(h_1,h_2)\neq (0,0)$ which means that 
$[\odot_i]$ is smooth outside the deepest diagonal. The distribution $[\odot_i]$ is compactly supported hence its Fourier transform is well-defined. We need to carefully justify the series
$\sum_{\vert k-\ell\vert\leqslant 1} P_k^i(y+h_1,y) \, \widetilde{P}^i_\ell(y,y+h_2)$ converges in conormal distributions. It suffices to control the microlocal convergence in one chart of $U\times U\times U$ since the wave front set behaves functorially under pull-backs \cite{BDH}. Note that from the definition
$$
[\odot_i]=\sum_{\vert k-\ell\vert\leqslant 1} \psi(\kappa(x))\Delta_k(\kappa(x)-\kappa(y))\chi(x)\widetilde{\psi}(\kappa(y))\Delta_\ell(\kappa(y)-\kappa(z))\widetilde{\chi}(z)  
$$
where $\kappa:U_i\mapsto \kappa(U_i)\subset \mathbb{R}^d$ is a given chart,
we get
$$
T= \left(\kappa\times \kappa \times \kappa\right)_*[\odot_i]=\sum_{\vert k-\ell\vert\leqslant 1} \psi(x)\Delta_k(x-y)\kappa^*\chi(x)\widetilde{\psi}(y)\Delta_\ell(y-z)\kappa^*\widetilde{\chi}(z).
$$
Since the functions $\chi,\widetilde{\chi},\psi,\widetilde{\psi}$ are smooth compactly supported,  
an explicit calculation using the Fourier transform yields
\begin{equation*} \begin{split}
&\big\vert\widehat{T}(\xi,\eta,\zeta)\big\vert   \\
&= \bigg\vert\sum_{\vert k-\ell\vert\leqslant 1} \frac{1}{(2\pi)^{2d}}\int  \psi(x)e^{i\theta_1.(x-y)}\psi(2^{-k}\theta_1) \kappa^*\chi(y)\widetilde{\psi}(y) e^{i\theta_2.(y-z)}\psi(2^{-k}\theta_2)\kappa^*\widetilde{\chi}(z)   \\
&\hspace{5cm}\times e^{-i(\xi\cdot x+\eta\cot y+\zeta\cdot z)} \rmd\theta_1\rmd\theta_2\rmd x\rmd y\rmd z \bigg\vert   \\
&= \bigg\vert \sum_{\vert k-\ell\vert\leqslant 1} \frac{1}{(2\pi)^{2d}} \int \psi(2^{-k}\theta_1) \psi(2^{-\ell}\theta_2)\widehat{\psi}(\xi-\theta_1)\widehat{\kappa^*\chi\widetilde{\psi} }(\eta+\theta_1-\theta_2)\widehat{\kappa^*\widetilde{\chi}}(\zeta+\theta_2) \rmd\theta_1\rmd\theta_2 \bigg\vert\,.
\end{split} \end{equation*}
We need to argue geometrically to control the Fourier decay of $\vert\widehat{T}(\xi,\eta,\zeta)\vert$ on small closed conic set avoiding the subspace $\{\xi+\eta+\zeta=0\}$ which is the fibre of the conormal of $\{x=y=z\}$. For any $(\xi_0,\eta_0,\zeta_0)\neq (0,0,0)$ such that $\xi_0+\eta_0+\zeta_0\neq 0 $. Assume without loss of generality that $\xi_0\neq 0$ (the other cases are treated symmetrically), then there exists a closed conic neighbourhood $V\subset (\mathbb{R}^{d})^3$ of $(\xi_0,\eta_0,\zeta_0)$ which does not meet $\{\xi+\eta+\zeta=0\}$ such that for some $\delta>0$,
$$
(\xi,\eta,\zeta)\in V \implies \vert \xi+\eta+\zeta \vert \geqslant\delta\vert \xi\vert.
$$
The first geometric inequality reads for three vectors $(A,B,C)\in (\mathbb{R}^d)^3$:
\begin{eqnarray*}
(1+\vert A\vert)^{-N}(1+\vert B\vert)^{-N}(1+\vert C\vert)^{-N}\lesssim
(1+\vert A\pm B\pm C \vert)^{-N}.
\end{eqnarray*} 
The second geometric inequality we shall use reads
\begin{eqnarray*}
1+\vert A \vert \leqslant (1+\vert A-B \vert)(1+\vert B\vert)\implies
(1+\vert A-B\vert)^{-d-1}\leqslant \frac{(1+\vert A\vert)^{d+1}}{(1+\vert B\vert)^{d+1}}.
\end{eqnarray*}
Applying a change of variables then using both inequalities to $\vert\widehat{T}(\xi,\eta,\zeta)\vert$ yields that for $(\xi,\eta,\zeta)\in V$, we have 
$$
\widehat{T}(\xi,\eta,\zeta) = \sum_{\vert k-\ell\vert\leqslant 1} \frac{1}{(2\pi)^{2d}} \int \psi(2^{-k}\theta_1) \psi(2^{-\ell}\theta_2)\widehat{\psi}(\xi-\theta_1)\widehat{\kappa^*\chi\widetilde{\psi} }(\eta+\theta_1-\theta_2)\widehat{\kappa^*\widetilde{\chi}}(\zeta+\theta_2) \rmd\theta_1\rmd\theta_2
$$
and an upper bound for $\vert\widehat{T}(\xi,\eta,\zeta)\vert$ of the form
\begin{equation*} \begin{split}
&\bigg\vert \sum_{\vert k-\ell\vert\leqslant 1}  2^{(k+\ell)d} \int \psi(\theta_1) \psi(\theta_2)\widehat{\psi}(\xi-2^k\theta_1)\widehat{\kappa^*\chi\widetilde{\psi} }(\eta+2^k\theta_1-2^\ell\theta_2)\widehat{\kappa^*\widetilde{\chi}}(\zeta+2^\ell\theta_2) \rmd\theta_1\rmd\theta_2 \bigg\vert   \\
&\lesssim \sum_k 2^{2kd} \int \psi(\theta_1) \psi(\theta_2) (1+\vert \xi-2^k\theta_1\vert)^{-N-4d-2} (1+\vert \zeta+2^k\theta_2\vert)^{-N-2d-1}   \\
&\hspace{5cm}\times \big(1+\vert\eta+2^k\theta_1-2^k\theta_2 \vert\big)^{-N-2d-1}\rmd\theta_1\rmd\theta_2   \\
&\lesssim \sum_k 2^{2kd} \int \psi(\theta_1) \psi(\theta_2) \, \frac{(1+\vert \xi\vert)^{2d+1}}{(1+2^k)^{2d+1}} \, \big(1+\vert \xi-2^k\theta_1\vert\big)^{-N-2d-1}\rmd\theta_1\rmd\theta_2   \\
&\times (1+\vert \zeta+2^k\theta_2\vert)^{-N-2d-1} \big(1+\vert\eta+2^k\theta_1-2^k\theta_2 \vert\big)^{-N-2d-1}\rmd\theta_1\rmd\theta_2   \\
&\lesssim \big(1+\vert \xi+\eta+\zeta\vert\big)^{-N-2d-1}(1+\vert \xi\vert)^{d+1} \sum_k 2^{-k}\lesssim \big(1+\vert \xi+\eta+\zeta\vert\big)^{-N} 
\end{split} \end{equation*}
since one has 
$$
1+\vert \xi\vert \lesssim 1+\vert\xi+\eta+\zeta\vert
$$ 
on $V$. This proves the convergence of the series defining $[\odot_i]$ in the conormal distributions whose wavefront set is contained in $N^*(\{(x=y=z)\}\subset M^3)$.

\smallskip

To probe the microlocal regularity under scalings, we just need to scale the representation of $[\odot_i]$  by a small factor $\lambda=2^{-j}$ in the chart $\kappa_i$ used to define both $P^i_k,\widetilde{P}_\ell^i$, then we will use the invariance of wave front sets under pull-backs together with the normal form result on scaling fields to conclude. First in the chart $\kappa_i\times \kappa_i\times
\kappa_i:U\times U\times U\mapsto (\mathbb{R}^d)^3$, we have
\begin{eqnarray*}
&&[\odot_i](y+2^{-j}h_1,y,y+2^{-j}h_2)   \\
&= &\sum_{\vert k-\ell\vert\leqslant 1} 2^{(k+\ell)d} \, \kappa_{i*}\widetilde{\chi}(y+2^{\ell-j}h_1) \, \widehat{\psi}(2^{-j}h_1) \, \psi(y) \, \widehat{\psi}(2^{k-j}h_2) \, \kappa_{i*}\chi(y+2^{-j}h_2)\\
&= & 2^{2jd} \sum_{\vert k-\ell\vert\leqslant 1} 2^{(\ell-j)d} \,  \kappa_{i*}\widetilde{\chi}(y+2^{-j}h_1) \, \widehat{\psi}(2^{\ell-j}h_1) \, \psi(y) 2^{(k-j)d}  \, \widehat{\psi}(2^{k-j}h_2) \, \kappa_{i*}\chi(y+2^{-j}h_2).
\end{eqnarray*}

We need to justify that the series 
$$
\sum_{\vert k-\ell\vert\leqslant 1} 2^{(\ell-j)d}  \, \kappa_{i*}\widetilde{\chi}(y+2^{-j}h_1) \, \widehat{\psi}(2^{\ell-j}h_1) \, \psi(y) 2^{(k-j)d} \, \widehat{\psi}(2^{k-j}h_2) \, \kappa_{i*}\chi(y+2^{-j}h_2)
$$ 
is bounded in conormal distributions uniformly in the index $j$. Beware that the above series only converges in the sense of distributions of order $0$ (one can think of them as elements in the dual of the Banach space $C^0$). Just rewrite the above series as a sum 
\begin{equation*} \begin{split}
&\underbrace{ \sum_{\vert k-\ell\vert\leqslant 1,(k,\ell)\geqslant 1} 2^{\ell d} K^i_{\ell+j}(y,2^{\ell}h_1) 2^{kd} \widetilde{K}^i_{k+j}(y,2^{k}h_2)}   \\
&+ \sum_{\vert k-\ell\vert\leqslant 1, 0\leqslant k,\ell \leqslant j} 2^{-\ell d} K^i_\ell(y,2^{-\ell}h_1) 2^{-kd} \widetilde{K}^i_k(y,2^{-k}h_2)  
\end{split} \end{equation*}
where the first term underbraced converges in conormal distributions and is bounded uniformly in $j$, and the second term is bounded uniformly in $j$ in the space of smooth functions. So noting that $d=3$ and for $2^{-j}=\lambda$, we conclude that the family $\lambda^{-6} [\odot_i](y+\lambda .,y,y+\lambda .), \lambda\in (0,1]$ forms a bounded family of distributions in $\mathcal{D}^\prime_\Gamma(U^3)$ for $\Gamma=N^*\{h_1=h_2=0\}$ where we fixed a very specific scaling towards the deepest diagonal. The scaling depends on the choice of chart $\kappa_i$. It remains to show that the statement is intrinsic, it holds true for any scaling field in the sense of~\cite{BDFT}, \cite[def 2.1]{DangWrochna} w.r.t. the deep diagonal $d_3\subset M^3$. For any pair of scaling fields $\rho_1,\rho_2$ defined near $d_3\subset M^3$, $e^{-t\rho_2*}= \Psi(t)^* e^{-t\rho_1*}$ where $\Psi(t):\Omega\subset M^3\mapsto \Omega\subset M^3$ is some family of local diffeomorphisms defined in some neighborhood $\Omega$ of $d_3$ fixing $d_3\subset M^3$ which has a well--defined limit when $t\rightarrow +\infty$ by~\cite[Prop 2.3 p.~826]{DangAHP}. If $e^{-6t} e^{-t\rho_1*}[\odot_i], t\in [0,+\infty)$ is bounded in $\mathcal{D}^\prime_{N^*d_3}$, then
\begin{eqnarray*}
e^{-6t} e^{-t\rho_2*}[\odot_i]= e^{6t}\Psi(t)^* e^{-t\rho_1*}[\odot_i]
\end{eqnarray*}
where the family
$\Psi(t)^*e^{6t} e^{-t\rho_1*}[\odot_i]$ is bounded in $\mathcal{D}^\prime_{N^*d_3}$ since $e^{6t} e^{-t\rho_1*}[\odot_i]$ bounded in $\mathcal{D}^\prime_{N^*d_3}$, continuity of the pull--back by $\Psi(t)^*$~\cite{BDH} and
the fact that the family $\Psi(t)$ has a well--defined smooth limit when $t\rightarrow +\infty$.  
\end{proof}

\smallskip

\subsection{Composing $\Psi_P^a$ with $\Psi_H^b$.  \hspace{0.1cm}}
\label{sect_compostion_para_ker}

In this section, we prove a weak form of composition theorem for our parabolic calculus. We denote by $\circ$ the composition of kernels in the space variables. More precisely,
\begin{eqnarray*}
K_1\circ K_2(x,y)\defeq\int_{z\in M} K_1(x,z)K_2(z,y) \mu_g(\rmd z)
\end{eqnarray*}
where $\mu_g$ is the Riemannian volume form on $M$. Recall that the heat calculus $\Psi_H$ was defined in Theorem~\ref{thm:heatcalculus} and the parabolic calculus in Definition \ref{def:paraboliccalculus}. We prove the following composition Theorem:

\smallskip

\begin{thm}[Composition Theorem]\label{t:comp2}
Let $M$ be a smooth closed manifold of dimension $d$. Pick $A\in\Psi^{a}_P(M)$ and $B\in \Psi^b_H(M)$ with
$$
\left\{\begin{array}{ll} a,b\leqslant-1\,, &  \\  d+2+2a+2b \geqslant 0 \,.&   \end{array} \right.
$$
Set
\begin{eqnarray*}
C(t_1,t_2,x,y)\defeq\int_{-L}^{\inf(t_1,t_2)} A(t_2-s)\circ B(t_1-s)\rmd s.
\end{eqnarray*}
One has, for all $\varepsilon>0$,
$$
\left\{\begin{array}{ll} C\in \Psi_P^{a+b} & \textrm{ if } d+2+2a+2b > 0 \,,\\  C\in \Psi_P^{a+b+\varepsilon} & \textrm{ if } d+2+2a+2b = 0  \,.\end{array} \right.
$$
Moreover, the composition is bilinear hypocontinuous for the respective topologies.
\end{thm}

\smallskip

Note that the hypocontinuity implies the sequential continuity for the composition. Note also that the composition cannot remain in the heat calculus, since we no longer have the off-diagonal small time decay. It makes natural that the result of the composition should be valued in the parabolic calculus. For applications to the renormalization of the quartic and quintic trees, especially for the explicit extraction of the counterterms, we use the above result with $a=-\frac{3}{2}, b=-1$ and $\dim(M)=3$. In that case one has $3+2-3-2=0 $.

\smallskip

\begin{proof}
Assume $t_2>t_1$, the other case is symmetrical. 
Without loss of generality, we shall work on $\mathbb{R}^{1+d}$ since the parabolic calculus is defined first on flat space and then transferred on manifolds. The composition result proved on $\mathbb{R}^{1+d}$ will automatically transfer to the manifold setting.

We localize the pair $(x,y)$ in some convex bounded region $U\subset \mathbb{R}^{1+d}$.
The two kernels we shall compose are denoted by
$A(t_2-s,x,x-z)$ and $B(t_1-s,z,z-y)$ respectively.
We would like to study and bound the 
kernel
\begin{eqnarray*}
C(t_2-t_1,x,x-y)\defeq\int_{-L}^{\inf(t_1,t_2)}\int_{\mathbb{R}^d} A(t_2-s,x,x-z) B(t_1-s,z,z-y)\rmd z\rmd s
\end{eqnarray*}

Observe that when $z$ is at distance $\geqslant 1$ from $U$, then both kernels $A,B$ in the above integral are smoothing in the space variable uniformly in $s$, since parabolic kernels are smoothing off-diagonal. Therefore, we may assume without loss of generality that $A, B$ are compactly supported in the variables $(x,z)$ respectively, they are proper operators, so that we may insert a first cut-off function $\chi_1\in C^\infty_c$ in the variable $z$ in the composition without affecting the analytical properties of $C$.

We now work with
$$
C = \int_{-L}^{\inf(t_1,t_2)}\int_{\mathbb{R}^d} A(t_2-s,x,x-z) B(t_1-s,z,z-y) \chi_1(z)\rmd z\rmd s+\text{smoothing}
$$
where $\chi_1=1$ on the support of $B$.
We use the following simple argument to justify $C$ is well-defined when $t_2>t_1$ as can be seen from the explicit bound
\begin{align*} 
\vert C(t_2-t_1,x,x-y)\vert &= \bigg\vert\int_{-L}^{t_1}\int_{z\in \mathbb{R}^d} A(t_2-s,x,x-z)B(t_1-s,z,z-y) \chi_1(z)\rmd z\rmd s\bigg\vert   \\ 
&\lesssim \int_{\mathbb{R}^d}\int_{-L}^{t_1} \frac{\chi_1(z)}{ \big(\sqrt{\vert t_2-s\vert }+\vert x-z \vert\big)^{d+2+2a} \big(\sqrt{\vert t_1-s\vert }+\vert y-z \vert\big)^{d+2+2b} } \,\rmd s \rmd z.
 \end{align*}
Since $t_2>t_1$, only one factor $\big(\sqrt{\vert t_1-s\vert }+\vert y-z \vert\big)^{-2-d-2b}$ blows up when $(z,s)=(y,t_1)$. But this is integrable since for all $b\leqslant -1$ and all test function $\chi\in C^\infty_c(\mathbb{R}_{\geqslant 0}\times \mathbb{R}^d)$ the following integral is bounded 
$$  
\int_{\mathbb{R}^{1+d}} \frac{\chi(u,z)}{ (\sqrt{u}+\vert z\vert)^{d+2+2b} } \rmd u  \rmd z \lesssim \int_{\mathbb{R}^{1+d}} \frac{\chi(u,z)}{ (\sqrt{u}+\vert z\vert)^{d} } \rmd u \rmd z \lesssim \int_{\mathbb{R}^{1+d}} \frac{\chi(v^2,z)}{ (\vert v\vert+\vert z\vert)^{d} }2v \, \rmd v\rmd z <+\infty,
$$ 
the other factor $\chi_1(z) \big(\sqrt{\vert t_2-s\vert }+\vert x-z \vert\big)^{-(d+2+2a)}$ is treated as test function of $z,s$, which shows the existence of the integral. We next localize the integral over the diagonals. Choose some function $\chi_2(t_2-t_1,t_1-s,x-z,z-y)$ which equals $1$ in some neighborhood of the subspace $\big\{x=z,y=z,t_2=t_1,t_1=s\big\}$. When we are outside the subspace $\big\{x=z,y=z,t_2=t_1,t_1=s\big\}$, then there is at least one of the two kernels $A, B$ that has either strictly positive time argument or is smoothing in the space variables and therefore $A(t_2-s)\circ B(t_1-s)$ will be smoothing in space variables. The next step is to localize the composition near the triple diagonal. We choose some function $\chi_2\big(t_2-t_1,t_1-s,x-z,z-y\big) $ which equals $1$ near the triple diagonal $\big\{x=z,y=z,t_2=t_1,t_1=s\big\}$. Then the composition decomposes as
\begin{equation*} \begin{split}
&\int_{-L}^{\inf(t_1,t_2)}\int_{\mathbb{R}^d} A(t_2-s,x,x-z) \chi_2(t_2-t_1,t_1-s,x-z,z-y) B(t_1-s,z,z-y) \chi_1(z)\rmd z\rmd s   \\
&+\underbrace{\int_{-L}^{\inf(t_1,t_2)}\int_{\mathbb{R}^d} A(t_2-s,x,x-z) (1-\chi_2(t_2-t_1,t_1-s,z-y)) B(t_1-s,z,z-y) \chi_1(z)\rmd z\rmd s}.
\end{split} \end{equation*}
Now observe that $B$ is smoothing off-diagonal and compactly supported in the $z$ variable, therefore the second piece underbraced is well-defined and is going to be smoothing in the space variables. So from now on, we focus on the first piece 
$$
\int_{-L}^{\inf(t_1,t_2)}\int_{\mathbb{R}^d} A(t_2-s,x,x-z) \chi_2(t_2-t_1,t_1-s,x-z,z-y) B(t_1-s,z,z-y) \chi_1(z)\rmd z\rmd s 
$$ 
which contains all singularities of $C$. We use the following notation for the parabolic action on a space-time point $(t_2-t_1,t_1-s,x-z,z-y)\in \mathbb{R}^{2+2d}$, for every $\lambda>0$
$$
\lambda\cdot\big(t_2-t_1,t_1-s,x-z,z-y\big) \defeq \Big(\lambda^2(t_2-t_1),\lambda^2(t_1-s),\lambda(x-z),\lambda(z-y)\Big),
$$
where $(0,+\infty)$ acts by parabolic scaling on space-time points of $\mathbb{R}^{2+2d}$. We will use a multiple scale decomposition of the cut-off function $\chi_2$ as follows
$$
\chi_2 = \chi_2\Big(\lambda^{-1}\cdot\big(t_2-t_1,t_1-s,x-z,z-y\big)\Big) + \int_{\lambda}^1 \psi\Big(\mu^{-1}\cdot\big(t_2-t_1,t_1-s,x-z,z-y\big)\Big) \frac{\rmd\mu}{\mu}
$$
where the piece $\chi_2\big(\lambda^{-1}(t_2-t_1,t_1-s,x-z,z-y)\big)$ is concentrated at scale $\lambda$ near $\big\{x=z,y=z,t_2=t_1,t_1=s\big\}$, and $\psi=-\lambda \frac{\rmd}{\rmd\lambda}\chi_2(\lambda\cdot)|_{\lambda=1} $ and the integral is a continuous decomposition at every scale ranging from $\lambda$ to $1$. Replacing the above decomposition in the definition of the composite operator $C$ yields
\begin{equation*} \begin{split}
& C\big(t_2-t_1,x,x-y\big)   \\
&= \int_{-L}^{t_1}\int_{z\in \mathbb{R}^d} A(t_2-s,x,x-z)\chi_2\Big(\lambda^{-1}\cdot\big(t_2-t_1,t_1-s,z-y\big)\Big)\,B\big(t_1-s,z,z-y\big) \chi_1(z)\rmd z\rmd s   \\
&+ \int_{\lambda}^1\int_{-L}^{t_1}\int_{z\in \mathbb{R}^d} \hspace{-0.3cm}A(t_2-s,x,x-z) \psi\Big(\mu^{-1}\cdot(t_2-t_1,t_1-s,z-y)\Big) B(t_1-s,z,z-y) \chi_1(z)\rmd z\rmd s \frac{\rmd\mu}{\mu}\,.
\end{split} \end{equation*}
The next step is to scale the composite operator exactly at scale $\lambda$
\begin{eqnarray*}
&& C\big(\lambda^2(t_2-t_1),x,\lambda(x-y)\big) = C_1+C_2   \\
C_1&=& \lambda^{d+2}\int_{t_1-(t_1+L)\lambda^{-2}}^{t_1}\int_{z\in \mathbb{R}^d} A\big(\lambda^2(t_2-s),x,\lambda(x-z)\big) \, \chi_2\big(t_2-t_1,t_1-s,z-y\big)\times   \\
&&B\big(\lambda^2(t_1-s),z,\lambda(z-y)\big) \, \chi_1\big(\lambda(z-y)+y\big) \,\rmd z\rmd s   \\
C_2& =&\lambda^{d+2} \int_{\lambda}^1\int_{t_1-(t_1+L)\lambda^{-2}}^{t_1}\int_{z\in \mathbb{R}^d} A\big(\lambda^2(t_2-s),x,\lambda(x-z)\big)\times   \\
&&
\psi\Big((\lambda\mu^{-1})\cdot\big(t_2-t_1,t_1-s,z-y\big)\Big) B\big(\lambda^2(t_1-s),z,\lambda(z-y)\big) \, \chi_1(\lambda(z-y)+y) \,\rmd z\rmd s \frac{\rmd\mu}{\mu}
\end{eqnarray*}
where we made a change of variables $s\mapsto \lambda^2(s-t_1)+t_1$, $z\mapsto \lambda(z-y)+y$, in the integrals. Then we bound the above two different terms in terms of bounds on $A, B$. The assumptions on $A,B$ imply the bounds
$$
\big\vert A(t_2-s,x,x-z) \big\vert \lesssim \left(\sqrt{t_2-s}+\vert x-z \vert \right)^{-d-2-2a},
$$
and 
$$
\big\vert B(t_1-s,z,z-y) \big\vert \lesssim \left( \sqrt{t_1-s}+\vert z-y \vert \right)^{-d-2-2b},
$$
which in turn imply
\begin{eqnarray*}
C_1
&\lesssim & \lambda^{d+2}\lambda^{-2d-4-2a-2b}\int_{t_1-(t_1+L)\lambda^{-2}}^{t_1}\int_{z\in \mathbb{R}^d}\left(\sqrt{t_1-s}+\vert x-z \vert \right)^{-d-2-2b}   \\
&\times & \left( \sqrt{t_2-s}+\vert z-y \vert \right)^{-d-2-2a} \, \chi_2\big(t_2-t_1,t_1-s,x-z,z-y\big) \, \chi_1(\lambda(z-y)+y) \,\rmd z\rmd s   \\
&\lesssim &\lambda^{-d-2-2a-2b},
\end{eqnarray*}
since the product $\chi_2\big(t_2-t_1,t_1-s,x-z,z-y\big) \chi_1(\lambda(z-y)+y)$ is compactly supported in $z,s$ uniformly in $y\in U$, $\lambda\in (0,1]$. For $C_2$ we have the upper bound
\begin{equation*} \begin{split}
&C_2\lesssim\lambda^{-d-2-2a-2b} \int_{\lambda}^1\int_{t_1-(t_1+L)\lambda^{-2}}^{t_1}\int_{z\in \mathbb{R}^d} \left(\sqrt{t_1-s}+\vert x-z \vert \right)^{-d-2-2b}\times   \\
&\psi\Big(\big(\lambda\mu^{-1})\cdot\big(t_2-t_1,t_1-s,x-z,z-y\big) \Big) \left( \sqrt{t_2-s}+\vert z-y \vert \right)^{-d-2-2a}\chi_1(\lambda(z-y)+y) \,\rmd z\rmd s \frac{\rmd\mu}{\mu}.
\end{split} \end{equation*} 
Now we use the fact that 
$$
\left( \sqrt{t_1-s}+\vert z-y \vert \right)^{-d-2-2b}\simeq \Big(\frac{\mu}{\lambda}\Big)^{-d-2-2b}
$$ 
and 
$$ 
\left(\sqrt{t_2-s}+\vert x-z \vert \right)^{-d-2-2a} \simeq \Big(\frac{\mu}{\lambda}\Big)^{-d-2-2a}
$$ 
on the support of $\psi\big((\lambda\mu^{-1})\cdot(t_2-t_1,t_1-s,x-z,z-y))$ because this function is supported on a corona of radius $\simeq \frac{\mu}{\lambda}$ and also that the integral 
$$
\int_{\mathbb{R}^{d+1}} \psi\Big((\lambda\mu^{-1})\cdot\times(t_2-t_1,t_1-s,x-z,z-y\big)\Big) \,\rmd z\rmd s \lesssim \Big\vert\Big\{ \sqrt{t_1-s}+\vert z-y \vert\leqslant \frac{\mu}{\lambda} \Big\} \Big\vert \lesssim \left(\frac{\mu}{\lambda}\right)^{d+2}, 
$$
since we are just bounding by the volume of some parabolic ball of radius $\frac{\mu}{\lambda}$. Combining the three previous bounds yields the estimate
\begin{eqnarray*}
C_2&\lesssim & \lambda^{-d-2-2a-2b} \int_{\lambda}^1\left(\frac{\mu}{\lambda}\right)^{-d-2-2b}\left(\frac{\mu}{\lambda}\right)^{-d-2-2a}\left(\frac{\mu}{\lambda}\right)^{d+2}\frac{\rmd\mu}{\mu}\\
&=&\int_\lambda^1 \mu^{-d-2-2a-2b} \frac{\rmd\mu}{\mu}\lesssim \lambda^{-d-2-2a-2b}
\end{eqnarray*}
if $ d+2+2a+2b\neq 0 $. If $d+2+2a+2b= 0$ then $C_1=\mathcal{O}(1)$ when $\lambda>0$ goes to $0$ and we get 
a logarithmic bound for $C_2$ of the form
$$
C_2\lesssim \vert\hspace{-0.05cm} \log\lambda\vert.
$$
So, for the moment, we proved that
\begin{eqnarray*}
C\big(\lambda^2(t_1-t_2),x,\lambda(x-y)\big) = \mathcal{O}\big(\lambda^{-d-2-2a-2b}\big)
\end{eqnarray*}
when $d+2+2a+2b>0$ and 
\begin{eqnarray*}
C(\lambda^2(t_1-t_2),x,\lambda(x-y))=\mathcal{O}\big(\vert\hspace{-0.05cm}\log\lambda\vert\big)
\end{eqnarray*}
when $d+2+2a+2b=0$.
To conclude that the composite operator $C$ still belongs to $\Psi_P$, we need to prove that the above bounds still hold when we test $C$ against elements of the module $\mathcal{M}$ of vector fields tangent to the diagonal $\{t_1=t_2,x=y\}\subset \mathbb{R}^{1+d}\times \mathbb{R}^{1+d}$.

The stability of the bounds by testing against tangent vector fields is treated separately in Lemma~\ref{l:testingtangentliealgebra}.  
\end{proof}

\smallskip

The proof of the composition theorem is not yet finished: we still need to show that we have the same estimates when we differentiate with vector fields $L_1,\dots, L_k$ that belong to the generators of the module $\mathcal{M}$ of vector fields tangent to $\{t_1=t_2,x=y\}\subset \mathbb{R}^{1+d}\times \mathbb{R}^{1+d}$.

\smallskip

\begin{lemm}[Composition and testing with tangent algebra]\label{l:testingtangentliealgebra}
Under the assumptions of Theorem~\ref{t:comp2}, for every $L\in \mathcal{M}$ in the tangent module, the kernel $LC$ satisfies the same estimate as $C$:
\begin{eqnarray*}
\big\vert LC(t_1,t_2,x,y)\big\vert \leqslant C_L \left(\vert t_2-t_1\vert +\vert x-y\vert^2 \right)^{-\frac{d+2+2a+2b}{2}}.
\end{eqnarray*}
\end{lemm}

\smallskip

\begin{proof}
We reduce the proof to some local computation in local coordinates involving generators of $\mathcal{M}$. We will do the detailed calculation for translations of the form $\partial_{x^i}+\partial_{y^i} $ (translation) and for general linear vector fields fixing the diagonal $\{x-y=0\}\subset U\times U$ of the form $M(x-y)\cdot\partial_x$. This covers the following important examples: $(x^i-y^i)\partial_{x^k}-(x^k-y^k)\partial_{x^i}$ (rotation), $(x^i-y^i)\partial_{x^i}$ (scaling), $(x^i-y^i)\partial_{x^k}+(x^k-y^k)\partial_{x^i} $ (boosts) in the situation of Theorem \ref{t:comp2} and we leave to the reader the other computations which follow the same pattern. We start with the translations
\begin{eqnarray*}
\left(\partial_{x^i}+\partial_{y^i}\right) C(t_1,t_2,x,y)&=&\left(\partial_{x^i}+\partial_{y^i}\right) \int_{\mathbb{R}^d}\int_{-L}^{\inf(t_1,t_2)} A(t_2-s,x,z) B(t_1-s,z,y) \chi(z)\rmd s\rmd z\\
&=&\int_{\mathbb{R}^d}\int_{-L}^{\inf(t_1,t_2)} ((\partial_{x^i}+\partial_{z^i} )A)(t_2-s,x,z)B(t_1-s,z,y)  \chi(z)\rmd s\rmd z\\
&+& \int_{\mathbb{R}^d}\int_{-L}^{\inf(t_1,t_2)} A(t_2-s,x,z) ((\partial_{y^i}+\partial_{z^i} )B)(t_1-s,z,y) \chi(z)\rmd s\rmd z\\
&+& \int_{\mathbb{R}^d}\int_{-L}^{\inf(t_1,t_2)} A(t_2-s,x,z)B(t_1-s,z,y) (\partial_{z^i}\chi)(z)\rmd s\rmd z.
\end{eqnarray*}
We see we can repeat the bounds of the proof of Theorem~\ref{t:comp2} on each term using the crucial information that both $(\partial_{x^i}+\partial_{z^i} )$  and $(\partial_{y^i}+\partial_{z^i}) $ are in the tangent algebra of the respective diagonals $\{x=z\}$ and $\{y=z\}$ and the stability of the two kernels $A, B$ in $\Psi_P$ by derivation by the tangent Lie algebra. Given a matrix $M\in M_d(\mathbb{R})$, we use the short hand notation $M(x-y)\cdot\partial_x$ for the vector field $M_i^j(x-y)^i\partial_{x^j}$ where we sum over repeated indices. Differentiating at $t=0$ yields the exact identities
\begin{equation*} \begin{split}
M(x-y)\cdot\partial_x &\int_{-L}^{\inf(t_1,t_2)} A(t_2-s)\circ \chi B(t_1-s) \rmd s   \\
&= \int_{-L}^{t_1} \left(\underbrace{ M(x-z)\cdot\partial_x} A\right)\circ \chi B+\left(\underbrace{\left(M(z-y)\cdot\partial_x + M(z-y)\cdot\partial_y  \right)}A\right) \circ \chi B \\
&\hspace{1cm}+  A\circ \chi\left( \underbrace{ M(z-y)\cdot\partial_y}B\right) + A\circ \left(M(z-y)\cdot\partial_y\chi\right)B\rmd s
\end{split} \end{equation*}
we decompose in different groups where we underbrace the vector fields which are tangent to the diagonal of the corresponding kernel. As usual, we decomposed $M(x-y)\cdot\partial_x \int_{-L}^{\inf(t_1,t_2)} A(t_2-s)\circ \chi B(t_1-s) \rmd s $ as a sum of three compositions of operators where the operators still satisfy the assumptions of Theorem~\ref{t:comp2}, so we are done.
\end{proof}

\medskip

\appendix
\section{A commutator identity on $\mathbb{R}^d$}
\label{SubsectionCommutatorIdentity}

We prove in this appendix a commutator estimate for triple paraproducts on $\mathbb{R}^d$, the result we establish is originally due to Bony~\cite[Thm 2.3 p.~215]{Bony} but we thought it would be useful to include a complete detailed proof here since it plays a central role for the proof of Theorem~\ref{thm_A_2_mfd} and the original paper~\cite{Bony} is written in French.

\smallskip

We are given $(f,g,h)$ where $f\in \mcC^{\alpha_1},g\in \mcC^{\alpha_2}$ and $h\in \mcC^\beta$ where $\alpha_1,\alpha_2>0$ and $\beta <0$. We would like to control a commutator:
\begin{eqnarray*}
f\prec\left( g\prec h \right)-(fg)\prec h.
\end{eqnarray*}

\smallskip

\begin{lemm}\label{lemm:comm_paramultiplication_flat}
Let $f\in \mcC^{\alpha_1}(\mathbb{R}^d),g\in \mcC^{\alpha_2}(\mathbb{R}^d)$ and $h\in \mcC^\beta(\mathbb{R}^d)$ where $\alpha_1,\alpha_2>0$ and $\beta <0$. Then we have
\begin{eqnarray*}
\|f\prec\left( g\prec h \right)-(fg)\prec h\|_{\alpha_1\wedge \alpha_2+\beta}\lesssim \|f\|_{\alpha_1}\|g\|_{\alpha_2}\|h\|_{\beta}.
\end{eqnarray*}
\end{lemm}

\smallskip

\begin{proof}
We first deal with the term $f\prec\left( g\prec h \right) $. By definition, we write:
\begin{eqnarray*}
f\prec\left( g\prec h \right)=\sum_{i=2}^\infty S_{i-2}(f) \Delta_i\left( \sum_{j=2}^\infty S_{j-2}(g)\Delta_j(h)\right).
\end{eqnarray*}
In the sequel, we shall repeatedly use the following result that can be found in \cite[Lemma 3 p.~280]{Meyer2}.

\smallskip

\begin{lemm}[Decomposition H\"older]
Let $p$ be a real number, $p\notin \mathbb{N}$ and $0<a<b$ be given.
If we are given a sequence $a_j$ of smooth functions such that 
$\Vert a_j \Vert_{L^\infty}=\mathcal{O}(2^{-jp})$
and each $a_j$ is Fourier supported in coronas $\big\{ a2^j\leqslant \vert \xi\vert \leqslant b2^j\big\}$, then the series
$ \sum_{j=0}^\infty a_j$
converges in the H\"older space $\mcC^p=B^p_{\infty,\infty}$.  
\end{lemm}

\smallskip

The first crucial observation, since $\Delta_i$ localizes in Fourier space on the corona $\{ 2^{i-1}\leqslant \vert \xi\vert\leqslant 2^{i+1} \}$ and that each $S_{j-2}(g)\Delta_j(h)$ is supported in the corona $2^{j-2}\leqslant \vert \xi\vert\leqslant  2^{j+2}$, necessarily the double sum over both $i,j$ localizes on the diagonal $ \vert i-j \vert\leqslant 3$.
So we rewrite the previous term as a double sum
\begin{eqnarray*} \label{eq:diaglocalization}
f\prec\left( g\prec h \right)=\sum_{\vert i-j\vert\leqslant 3, i,j\geqslant 2} S_{i-2}(f) \Delta_i\big( S_{j-2}(g)\Delta_j(h)\big).
\end{eqnarray*}
The second observation is that if we fix $j$, then the sum of the five terms:
\begin{eqnarray} \label{eq:asumgivesidentity}
\sum_{i=j-3}^{j+3}\Delta_{i}\big( S_{j-2}(g)\Delta_j(h)\big) = S_{j-2}(g)\Delta_j(h)
\end{eqnarray}
this is because 
$ \psi(2^{-j+3}\xi)+\dots + \psi(2^{-j-3}\xi)=1$ on the corona 
$2^{j-2}\leqslant \vert \xi\vert\leqslant  2^{j+2}$ by construction of the Littlewood-Paley-Stein partition of unity. 
Therefore at fixed $j$, we can add and subtract as follows
\begin{eqnarray*}
 &&\sum_{i=j-3}^{j+3} S_{i-2}(f) \Delta_i\big( S_{j-2}(g)\Delta_j(h)\big)   \\
 &=&\sum_{i=j-3}^{j+3} S_{j-5}(f) \Delta_i\big( S_{j-2}(g)\Delta_j(h)\big) - \sum_{i=j-3}^{j+3} (S_{i-2}-S_{j-5})(f) \Delta_i\big( S_{j-2}(g)\Delta_j(h)\big)   \\
 &=&S_{j-5}(f)  S_{j-2}(g)\Delta_j(h)
 -\sum_{i=j-3}^{j+3} (S_{[j-5,i-2]})(f) \Delta_i\big( S_{j-2}(g)\Delta_j(h)\big)
\end{eqnarray*}
then in the second line, we used the second miracle equation~(\ref{eq:asumgivesidentity}). We define $(S_{[j-5,i-2]})$ as the difference $(S_{i-2}-S_{j-5})$ and we observe that $(S_{[j-5,i-2]})(f)$ is Fourier supported on some corona contained in the shell $ \vert \xi\vert\leqslant 2^{i-2}$. Since $\Delta_i\left( S_{j-2}(g)\Delta_j(h)\right)$ is supported in the shell $2^{i-1}\leqslant \vert \xi\vert\leqslant 2^{i+1}$ because of the localizing property of $\Delta_i$, the discrepancy $i-2,i$ makes the support of the product $(S_{[j-5,i-2]})(f) \Delta_i\left( S_{j-2}(g)\Delta_j(h)\right)$ a corona around $\vert\xi\vert\simeq 2^i$. From the H\"older regularities assumptions on the functions $(f,g,h)$, we get the bound
\begin{eqnarray*}
\Big\Vert \sum_{i=j-3}^{j+3}  (S_{[j-5,i-2]})(f) \Delta_i\left( S_{j-2}(g)\Delta_j(h)\right) \Big\Vert_{L^\infty} \lesssim \sum_{i=j-3}^{j+3} \big\Vert (S_{[j-5,i-2]})(f)\big\Vert_{L^\infty}\Vert \Delta_j(h)\Vert_{L^\infty}\\ \lesssim 2^{-(j-5)\alpha_1}2^{-j\beta}\lesssim 2^{-j(\alpha_1+\beta)}
\end{eqnarray*}
where we use the fact that $S_{j-2}(g)$ is bounded uniformly in the index $j$ since $g\in \mcC^{\alpha_2}$ for $\alpha_2>0$. 
So the series
\begin{eqnarray*}
\sum_j \left( \sum_{i=j-3}^{j+3}  (S_{[j-5,i-2]})(f) \Delta_i\left( S_{j-2}(g)\Delta_j(h)\right)\right)
\end{eqnarray*}
is a series of functions 
supported in coronas $ a2^j\leqslant \vert \xi\vert \leqslant b2^j$ for $0<a<b$ and thus converges absolutely in $\mcC^{\alpha_1+\beta}$. This tells us that in the equality
\begin{equation*} \begin{split}
\sum_{i=j-3}^{j+3} S_{i-2}(f) \Delta_i&\big( S_{j-2}(g)\Delta_j(h)\big)   \\
&= S_{j-5}(f)  S_{j-2}(g)\Delta_j(h) - \sum_{i=j-3}^{j+3} (S_{[j-5,i-2]})(f) \Delta_i\big( S_{j-2}(g)\Delta_j(h)\big)
\end{split} \end{equation*}
the sum $\sum_{i=j-3}^{j+3} (S_{[j-5,i-2]})(f) \Delta_i\big( S_{j-2}(g)\Delta_j(h)\big)$ is a good term absorbed in a good remainder and we should only keep $S_{j-5}(f)  S_{j-2}(g)\Delta_j(h)$.

So for the moment, we just proved that
\begin{eqnarray*} \label{eq:simplificationhuge}
f\prec\left( g\prec h \right)=\sum_{j\geqslant 5} S_{j-5}(f)  S_{j-2}(g)\Delta_j(h)+\mcC^{\alpha_1+\beta}.
\end{eqnarray*}
 Now we would like to compare this quantity with
\begin{eqnarray*}
\left(fg\right)\prec h=\sum_{j\geqslant 2} S_{j-2}(fg)\Delta_j(h)
\end{eqnarray*}
so the difference $f\prec\left( g\prec h \right)-\left(fg\right)\prec h$ reads
\begin{eqnarray*}
\sum_{j\geqslant 5} \big(S_{j-5}(f)  S_{j-2}(g)-S_{j-2}(fg)\big) \Delta_j(h)+\mcC^{\alpha_1+\beta}.
\end{eqnarray*}
Everything boils down to studying the difference $S_{j-5}(f)  S_{j-2}(g)-S_{j-2}(fg)$ which we treat as follows. First decompose $fg=S_{j-5}(f)S_{j-5}(g)+R$ where the remainder contains at least either one of the two terms $ \sum_{i\geqslant j-5} \Delta_i(f)$ or  $\sum_{i\geqslant j-5} \Delta_i(g) $ in factor. Observe that 
\begin{eqnarray*}
\sum_{i\geqslant j-5} \Delta_i(f)=\mathcal{O}_{\mcC^{\alpha_1}}(2^{-j\alpha_1})\\
\sum_{i\geqslant j-5} \Delta_i(g)=\mathcal{O}_{\mcC^{\alpha_2}}(2^{-j\alpha_2})
\end{eqnarray*} 
this is almost by construction of these objects and by definition of the H\"older norms. Therefore using the continuity of 
$S_{j-2}:\mcC^\bullet\mapsto \mcC^\bullet$ acting on H\"older spaces where this is bounded uniformly in $j$, we deduce that
$S_{j-2}(R)=\mathcal{O}_{\mcC^{\alpha_1\wedge\alpha_2}}(2^{-j(\alpha_1\wedge\alpha_2)})$
and $\sum_j S_{j-2}(R)\Delta_j(h)$ is a series of functions each term 
supported in coronas $ a2^j\leqslant \vert \xi\vert \leqslant b2^j$ for $0<a<b$, $ \Vert S_{j-2}(R)\Delta_j(h)\Vert_{L^\infty}=\mathcal{O}(2^{-j(\alpha_1\wedge\alpha_2+\beta)}) $  and thus
converges absolutely in $C^{\alpha_1\wedge\alpha_2+\beta}$.
Using the magic identity 
$$
S_{j-2}\big(S_{j-5}(f)S_{j-5}(g)\big) = S_{j-5}(f)S_{j-5}(g),
$$ 
this means that the difference can be simplified as 
$$
S_{j-5}(f)  S_{j-2}(g)-S_{j-2}(fg) = S_{j-5}(f)\big(  S_{j-2}(g)-S_{j-5}(g)\big) - S_{j-2}(R)
$$ 
and combining with the fact that $\sum_j S_{j-2}(R)\Delta_j(h)\in \mcC^{\alpha_1\wedge\alpha_2+\beta}$, the difference $f\prec\left( g\prec h \right)-\left(fg\right)\prec h$ now reads
$$
\sum_{j\geqslant 5} S_{j-5}(f)S_{[j-5,j-2]}(g)\Delta_j(h)+\mcC^{\alpha_1\wedge\alpha_2+\beta},
$$
we are done using again the fact that it is a series of functions supported in annular domains and that
$\Vert S_{j-5}(f)S_{[j-5,j-2]}(g)\Delta_j(h)\Vert_{L^\infty}=\mathcal{O}(2^{-j(\alpha_2+\beta)})$.
So we get
$$
f\prec\left( g\prec h \right)-\left(fg\right)\prec h \in \mcC^{\alpha_1\wedge\alpha_2+\beta}
$$
as required.
\end{proof}


\vspace{1cm}

\noindent \textcolor{gray}{$\bullet$} {\sf I. Bailleul} -- Univ Brest, CNRS UMR 6205, Laboratoire de Math\'ematiques de Bretagne Atlantique, France. {\it E-mail}: ismael.bailleul@univ-brest.fr   

\bigskip

\noindent \textcolor{gray}{$\bullet$} {\sf N.V. Dang} -- Sorbonne Université and Université Paris Cité, CNRS, IMJ-PRG, F-75005 Paris, France. Institut Universitaire de France, Paris, France.
{\it E-mail}: dang@imj-prg.fr

\bigskip

\noindent \textcolor{gray}{$\bullet$} {\sf L. Ferdinand} -- 
Laboratoire de Physique des 2 infinis Irène Joliot-Curie, UMR 9012, Universit\'e Paris-Saclay, Orsay, France.
{\it E-mail}: lferdinand@ijclab.in2p3.fr

\bigskip

\noindent \textcolor{gray}{$\bullet$} {\sf T.D. T\^o} -- Sorbonne Université and Université Paris Cité, CNRS, IMJ-PRG, F-75005 Paris, France.  Current adress: Institut de Math\'ematiques de Toulouse,
Universit\'e de Toulouse, 118 route de Narbonne, 31400 Toulouse, France. {\it E-mail}: tat-dat.to@imj-prg.fr, tat-dat.to@math.univ-toulouse.fr


\begin{thebibliography}{0}

\bibitem{BCD}
H. Bahouri and J.Y. Chemin and R. Danchin. {\em Fourier analysis and nonlinear partial differential equations}. Vol. 343. Berlin: Springer, 2011.

 \bibitem{BailleulUniqueness}
I. Bailleul,
\newblock {\em Uniqueness of the $\Phi^4_3$ measure on closed Riemannian manifolds}.
\newblock arXiv preprint, (2023).

\bibitem{BB1}
I. Bailleul and F. Bernicot. {\em Heat semigroup and singular PDEs.} J. Funct. Anal., {\bf 270}(9):3344--3452, (2016).

\bibitem{BB2}
I. Bailleul and F. Bernicot and D. Frey. {\em Space-time paraproducts for paracontrolled calculus, 3d-PAM and multiplicative Burgers equations.} Ann. Sci. École Norm. Sup., {\bf 51}(6):1399--1456, (2018).

\bibitem{BB3}
I. Bailleul and F. Bernicot. {\em High order paracontrolled calculus.} Forum Math. Sigma, {\bf 7}, e44:1--94, (2019).

\bibitem{BDFT}
I. Bailleul and N.V. Dang and L. Ferdinand and T.D. T\^o. {\em $\Phi^4_3$ measure on compact Riemannian $3$--manifolds.}
\newblock arXiv:2304.10185, (2023).

\bibitem{BG}
N. Barashkov and M. Gubinelli,
\newblock {\em A variational method for $\Phi^4_3$}.
\newblock Duke Math. J., {\bf 169}(17): 3339--3415, (2020).

\bibitem{Beals}
R. Beals "Characterization of pseudo-differential operators and applications." Duke Math. J. 44(1)--(1977): 45--57.

\bibitem{BerlineGetzlerVergne}
N. Berline and E. Getzler and M. Vergne. Heat kernels and Dirac operators. Springer, 2003.

\bibitem{Bernicot}
F. Bernicot,
\newblock {\em A $T(1)$-theorem in relation to a semigroup of operators and applications to new paraproducts}.
\newblock Trans. Amer. Math. Soc., {\bf 364}:257--294, (2012).

\bibitem{BernicotSire}
F. Bernicot and Y. Sire,
\newblock {\em Propagation of low regularity for solutions of nonlinear PDEs on a Riemannian manifold with a sub-Laplacian structure.}
\newblock Ann. I.H. Poincar\'e -- AN, {\bf 30}:935--958, (2013).

\bibitem{Bony}
J.M. Bony. "Calcul symbolique et propagation des singularités pour les équations aux dérivées partielles non linéaires." Ann. Sci. Éc. Norm. Sup., {\bf 114}(2):209--246, (1981).

\bibitem{Bony2}
J.M. Bony et al. {\em Analyse microlocale des équations aux dérivées partielles non linéaires.} Microlocal Analysis and Applications: Lectures given at the 2nd Session of the Centro Internazionale Matematico Estivo (CIME) held at Montecatini Terme, Italy, July 3–11, 1989. Springer Berlin Heidelberg, 1991.

\bibitem{Bony3}
J.M. Bony. {\em Second Microlocalization and Propagation of Singularities for Semi-Linear Hyperbolic Equations.} Hyperbolic Equations and Related Topics: Proceedings of the Taniguchi International Symposium, Katata and Kyoto, 1984. Academic Press, 2014.

\bibitem{BDH}
Ch. Brouder and N.V. Dang and F. Hélein. {\em Continuity of the fundamental operations on distributions having a specified wave front set (with a counter example by Semyon Alesker).} Studia Math.,  {\bf 232}(3):201--226, (2016).

\bibitem{CatellierChouk}

Catellier, Rémi, and Khalil Chouk. "Paracontrolled distributions and the 3-dimensional stochastic quantization equation." The Annals of Probability 46.5 (2018): 2621-2679.


\bibitem{DangAHP} 
N.V. Dang,
\newblock {\em The extension of distributions on manifolds, a microlocal approach}.
\newblock Annales Henri Poincaré, {\bf 17}(4), (2016).


\bibitem{DangWrochna}
N.V. Dang and M Wrochna. {\em Dynamical residues of Lorentzian spectral zeta functions.} arXiv:2108.07529 (2021), to appear in J.\'Ec. Polytechnique (2023+).

\bibitem{Eskin}
G.I. Eskin. {\em Lectures on linear partial differential equations}. Vol. 123. American Mathematical Soc., 2011.

\bibitem{FermKamlectures}
C. Fermanian-Kammerer.
\newblock {\em Semiclassical analysis of Schrödinger operators}.
\newblock Von-Neumann Lecture, Technische Universit\"at M\"unchen.

\bibitem{HairerFriz}
P.K. Friz  and M. Hairer.
\newblock \emph{ A course on rough paths.}
\newblock Springer International Publishing, 2020.

\bibitem{Gilkey}
P.B. Gilkey. {\em Invariance theory: the heat equation and the Atiyah-Singer index theorem}. Vol. 16. CRC press, 2018.

\bibitem{Glimm}
J. Glimm. {\em Boson fields with the $:\phi^4:$ interaction in three dimensions.} Comm. Math. Phys., {\bf 10}(1):1--47, (1968).

\bibitem{GJ}
J. Glimm and A. Jaffe. {\em Positivity of the $\phi^4_3$ Hamiltonian.} Fortschritte der Physik. Progress of Physics, {\bf 21}:327–376, (1973).

\bibitem{Grieser}
D. Grieser. {\em Notes on heat kernel asymptotics.} Available on his website: http://www. staff. unioldenburg. de/daniel. grieser/wwwlehre/Schriebe/heat. pdf (2004).

\bibitem{GH}
M. Gubinelli and M. Hofmanová. {\em A PDE construction of the Euclidean $\Phi^4_3$ quantum field theory.} Comm. Math. Phys., {\bf 384}(1):1--75, (2021).

\bibitem{GIP}
M. Gubinelli and P. Imkeller and N. Perkowski,
\newblock {\em Paracontrolled distributions and singular PDEs}.
\newblock Forum Math. Pi, {\bf 3}-e6:1--75, (2015).


\bibitem{GubinelliPerkowskiKPZReloaded}
M. Gubinelli and N. Perkowski,
\newblock {\em KPZ reloaded}.
\newblock Comm. Math. Phys., {\bf 349}:165--269, (2017).

\bibitem{BonLef}
Y. Guedes Bonthonneau and Th. Lefeuvre. {\em Radial source estimates in H\" older-Zygmund spaces for hyperbolic dynamics.} arXiv:2011.06403 (2020), to appear in Ann. Henri Lebesgue (2023+).

\bibitem{BGP}
C. Guillarmou and Th. De Poyferré (with an appendix by Y. Guedes Bonthonneau). {\em A paradifferential approach for hyperbolic dynamical systems and applications.} Tunisian J. Math. {\bf 4}(4): 673--718, (2023).

\bibitem{Hairer}
M. Hairer,
\newblock {\em A theory of regularity structures}.
\newblock Invent. Math., {\bf 198}(2):269--504, (2014).

\bibitem{HM}
M. Hairer and J. Mattingly,
\newblock {\em The strong Feller property for singular stochastic PDEs}.
\newblock Ann. Institut H. Poincar\'e B, {\bf 54}(3):1314--1340, (2018).

\bibitem{Horm3}
L. H\"ormander. {\em The analysis of linear partial differential operators}, vol. 3. Pseudo-differential operators, volume 274 of Grundlehren der Mathematischen Wissenschaften." (1985).

\bibitem{Horm97}
L. Hörmander. {\em Lectures on nonlinear hyperbolic differential equations}. Vol. 26. Springer Science, 1997.

\bibitem{HormanderNashMoser}
L. Hörmander. {\em The Nash-Moser Theorem and
Paradifferential Operators}. Analysis, et Cetera: 429-449 (1990).

\bibitem{JP}
A. Jagannath and N. Perkowski,
\newblock {\em A simple construction of the dynamical $\phi^4_3$ model}.
\newblock Trans. Amer. Math. Soc., {\bf 376}:1507--1522, (2023).

\bibitem{JoshiPAMS}
M.S. Joshi. {\em An intrinsic characterisation of polyhomogeneous Lagrangian distributions.} Proc. Am. Math. Soc., {\bf 125}(5):1537--1543, (1997).

\bibitem{Joshisymb}
M.S. Joshi. {\em A symbolic construction of the forward fundamental solution of the wave operator.} Comm. Partial Diff. Eq., {\bf 23}(7-8):1349--1417, (1998).

\bibitem{Joshi}
M.S. Joshi. {\em Geometric proofs of composition theorems for generalized Fourier integral operators.} Portugaliae Math., {\bf 56}(2):129--154, (1999).

\bibitem{KR}
S. Klainerman and I. Rodnianski. {\em A geometric approach to the Littlewood–Paley theory.} Geom. Funct. Anal. GAFA, {\bf 16}:126--163, (2006).

\bibitem{ML1}
Y. Loizides and E. Meinrenken. {\em Differential geometry of weightings.} arXiv:2010.01643 (2020).

\bibitem{ML2}
Y. Loizides and E. Meinrenken. {\em Singular Lie filtrations and weightings.} arXiv:2201.00442 (2022).

\bibitem{Meinrenkensurvey}
E. Meinrenken. {\em Euler-like vector fields, normal forms, and isotropic embeddings.} Indagationes Math., {\bf 32}(1):224--245, (2021).

\bibitem{MelroseAPS}
R. Melrose. {\em The Atiyah-Patodi-singer index theorem}. AK Peters/CRC Press, 1993.

\bibitem{MelRitt}
R. Melrose and N. Ritter. {\em Interaction of nonlinear progressing waves for semilinear wave equations.} Ann. Math., {\bf 121}(1):187--213, (1985).

\bibitem{Meyer2}
Y. Meyer and R. Coifman. {\em Wavelets: Calder\'on-Zygmund and multilinear operators}. Vol. 48. Cambridge University Press, 1997.

\bibitem{MW20}
A. Moinat and H. Weber. {\em Space-Time Localisation for the Dynamic $\Phi^4_3$ Model.} Comm. Pure Appl. Math., {\bf 73}(12):2519--2555, (2020).

\bibitem{MWX}
J.C. Mourrat and H. Weber, and Weijun Xu. {\em Construction of $\varPhi^ 4_3 $ Diagrams for Pedestrians.} Meeting on Particle Systems and PDE's. Springer, Cham, 2015.

\bibitem{MW17}
J.C. Mourrat and H. Weber. {\em The dynamic $\Phi^4_3$ model comes down from infinity.} Comm. Math. Phys., {\bf 356}:673--753, (2017).

\bibitem{MouzardWeyl}
A. Mouzard,
\newblock {\em Weyl law for the Anderson Hamiltonian on a two-dimensional manifold}.
\newblock Ann. Inst. H. Poincar\'e -- Probab. Statist., {\bf 58}(3):1385--1425, (2022).

\bibitem{MouzardThesis}
A. Mouzard,
\newblock {\em Un calcul paracontr\^ol\'e pour les EDP stochastiques singuli\`eres sur les vari\'et\'es}.
\newblock https://amouzard.perso.math.cnrs.fr/manuscrit.pdf

\bibitem{NZZ}

Nam, Phan Thành, Rongchan Zhu, and Xiangchan Zhu. "$\Phi^ 4_3 $ Theory from many-body quantum Gibbs states." arXiv preprint arXiv:2502.04884 (2025).

\bibitem{Roe}
J. Roe. {\em Elliptic operators, topology, and asymptotic methods}. CRC Press, 1999.

\bibitem{Taylor2}
M.E. Taylor. {\em Partial differential equations II: Qualitative studies of linear equations}. Vol. 116. Springer Science, 2013.

\bibitem{Taylor3}
M.E. Taylor. {\em Partial differential equation III–Nonlinear equations.} Applied Mathematical Sciences 117 (1996).

\bibitem{Taylorpsido}
M.E. Taylor. {\em Pseudodifferential Operators (PMS-34).} Pseudodifferential Operators (PMS-34). Princeton University Press, 2017.

\bibitem{Witten}
E. Witten. {\em Why Does Quantum Field Theory In Curved Spacetime Make Sense? And What Happens To The Algebra of Observables In The Thermodynamic Limit?.} Dialogues Between Physics and Mathematics: CN Yang at 100. Cham: Springer International Publishing, 2022. 241-284.

\bibitem{Zworski}
M. Zworski. {\em Semiclassical analysis}. Vol. 138. American Mathematical Society, 2022.

\end{thebibliography}
\end{document}